\pdfoutput=1
\RequirePackage{ifpdf}
\ifpdf % We are running pdfTeX in pdf mode
\documentclass[pdftex]{sigma}
\else
\documentclass{sigma}
\fi

\usepackage{slashed}
\usepackage{epic}
\usepackage{slashbox}
\usepackage{makecell}
\newcolumntype{x}[1]{>{\centering\arraybackslash}p{#1}}
\usepackage{tikz}
\newcommand\diagonal[4]{
\multicolumn{1}{p{#2}|}{\hskip-\tabcolsep
$\vcenter{\begin{tikzpicture}[baseline=0,anchor=south west,inner sep=#1]
\path[use as bounding box] (0,0) rectangle (#2+2\tabcolsep,\baselineskip);
\node[minimum width={#2+2\tabcolsep},minimum height=\baselineskip+\extrarowheight] (box) {};
\draw (box.north west) -- (box.south east);
\node[anchor=south west] at (box.south west) {#3};
\node[anchor=north east] at (box.north east) {#4};
\end{tikzpicture}}$\hskip-\tabcolsep}}
\makeatletter
\newcommand{\doublewidetilde}[1]{{
\mathpalette\double@widetilde{#1}
}}
\newcommand{\double@widetilde}[2]{
\sbox\z@{$\m@th#1\widetilde{#2}$}
\ht\z@=.9\ht\z@
\widetilde{\box\z@}
}
\makeatother
\renewcommand{\setminus}{-}

\newcommand{\Ad}{\textup{Ad}}
\newcommand{\ad}{\textup{ad}}

\newcommand{\id}{\textup{id}}
\newcommand{\1}{\mathbf{1}}
\newcommand{\0}{\mathbf{0}}
\newcommand{\Ind}{\textup{Ind}}
\newcommand{\sgn}{\textup{sgn}}

\newcommand{\diag}{\textup{diag}}
\newcommand{\Hom}{\textup{Hom}}
\newcommand{\blank}{\,\cdot\,}

\newcommand{\even}{{\textup{even}}}
\newcommand{\odd}{{\textup{odd}}}
\DeclareMathOperator{\supp}{supp}
\DeclareMathOperator{\tr}{tr}

\DeclareMathOperator{\Cl}{Cl}
\DeclareMathOperator{\dirac}{\slashed{D}\hspace{-.05em}}
\DeclareMathOperator{\Pol}{Pol}

\newcommand{\Kslash}{\slashed{K}}
\newcommand{\Lslash}{\slashed{L}}
\newcommand{\pislash}{\slashed{\pi}}
\newcommand{\tauslash}{\slashed{\tau}}
\newcommand{\proj}{\textup{proj}}
\newcommand{\End}{\textup{End}}
\newcommand{\HC}{\textup{HC}}

\newcommand{\GL}{\textup{GL}}

\newcommand{\upO}{\textup{O}}
\newcommand{\SO}{\textup{SO}}
\newcommand{\so}{\mathfrak{so}}
\newcommand{\Spin}{\textup{Spin}}

\newcommand{\Pin}{\textup{Pin}}

\renewcommand{\AA}{\mathbb{A}}
\newcommand{\BB}{\mathbb{B}}
\newcommand{\CC}{\mathbb{C}}

\newcommand{\NN}{\mathbb{N}}

\newcommand{\RR}{\mathbb{R}}
\renewcommand{\SS}{\mathbb{S}}

\newcommand{\ZZ}{\mathbb{Z}}

\newcommand{\calD}{\mathcal{D}}
\newcommand{\calE}{\mathcal{E}}
\newcommand{\calF}{\mathcal{F}}

\newcommand{\calM}{\mathcal{M}}

\newcommand{\calT}{\mathcal{T}}

\newcommand{\calV}{\mathcal{V}}
\newcommand{\calW}{\mathcal{W}}

\newcommand{\fraka}{\mathfrak{a}}

\newcommand{\frakc}{\mathfrak{c}}

\newcommand{\frakg}{\mathfrak{g}}
\newcommand{\frakh}{\mathfrak{h}}
\newcommand{\frakk}{\mathfrak{k}}

\newcommand{\frakm}{\mathfrak{m}}
\newcommand{\frakn}{\mathfrak{n}}

\newcommand{\fraks}{\mathfrak{s}}
\newcommand{\frakt}{\mathfrak{t}}

\newcommand{\LocalPoles}{\,\,\backslash \kern-0.8em{\backslash} \; }
\newcommand{\GlobalPoles}{\,\,/ \kern-0.8em{/} \; }

\numberwithin{equation}{section}

\newtheorem{Theorem}{Theorem}[section]
\newtheorem{Corollary}[Theorem]{Corollary}
\newtheorem{Lemma}[Theorem]{Lemma}
\newtheorem{Proposition}[Theorem]{Proposition}
\newtheorem{thmalph}{Theorem}
\newtheorem{coralph}[thmalph]{Corollary}

{ \theoremstyle{definition}

\newtheorem{Remark}[Theorem]{Remark} }

\begin{document}
%\allowdisplaybreaks

\newcommand{\arXivNumber}{1702.02326}

\renewcommand{\PaperNumber}{084}

\FirstPageHeading

\ShortArticleName{Knapp--Stein Type Intertwining Operators for Symmetric Pairs II}

\ArticleName{Knapp--Stein Type Intertwining Operators\\ for Symmetric Pairs~II. -- The Translation Principle\\ and Intertwining Operators for Spinors}

\Author{Jan FRAHM and Bent {\O}RSTED}
\AuthorNameForHeading{J.~Frahm and B.~{\O}rsted}
\Address{Department of Mathematics, Aarhus University,\\ Ny Munkegade 118, 8000 Aarhus C, Denmark}
\Email{\href{mailto:frahm@math.au.dk}{frahm@math.au.dk}, \href{mailto:orsted@math.au.dk}{orsted@math.au.dk}}

\ArticleDates{Received May 17, 2019, in final form October 29, 2019; Published online November 02, 2019}

\Abstract{For a symmetric pair $(G,H)$ of reductive groups we extend to a large class of generalized principal series representations our previous construction of meromorphic families of symmetry breaking operators. These operators intertwine between a possibly vector-valued principal series of~$G$ and one for~$H$ and are given explicitly in terms of their integral kernels. As an application we give a complete classification of symmetry breaking operators from spinors on a Euclidean space to spinors on a hyperplane, intertwining for a~double cover of the conformal group of the hyperplane.}

\Keywords{Knapp--Stein intertwiners; intertwining operators; symmetry breaking operators; symmetric pairs; principal series; translation principle}

\Classification{22E45; 47G10}

\section{Introduction}

In the study of representations of real reductive Lie groups, intertwining operators play a decisive role. The most prominent family of such operators is given by the standard Knapp--Stein operators which intertwine between two principal series representations of a~group~$G$. More recently, intertwining operators have been studied and used in the framework of branching problems, i.e., the restriction of a representation to a subgroup $H\subseteq G$ and its decomposition. Here one is interested in operators from a representation of $G$ to a representation of $H$, intertwining for the subgroup. Such operators are also called \textit{symmetry breaking operators}, a term coined by T.~Kobayashi in his program for branching problems (see, e.g.,~\cite{Kob15}).

\looseness=-1 Such symmetry breaking operators have been studied in great detail in the special case of the conformal groups corresponding to a Euclidean space and a hyperplane by Kobayashi--Speh~\cite{KS15}. In this case some of the operators turn out to be differential operators, namely exactly the operators found by A.~Juhl~\cite{Juh09} in connection with his study of $Q$-curvature and holography in conformal geometry. Further connections to elliptic boundary value problems~\cite{MOZ16} and automorphic forms~\cite{MO14b} indicate the broad spectrum of applications that these operators provide.

In our recent work with Y.~Oshima~\cite{MOO16} we generalized the construction of symmetry breaking operators by Kobayashi--Speh to a large class of symmetric pairs $(G,H)$ and spherical principal series representations, proving meromorphic dependence on the parameters in general and generic uniqueness in some special cases. In this paper we shall further extend our construction to the case of vector-valued principal series representations; we establish many of the properties, now for arbitrary vector bundles, in particular the meromorphic dependence on the parameters. The main argument is a version of a well-known translation principle, namely by tensoring with a finite-dimensional representation of the group. As an application and illustration we give all details for the case of spinors on a Euclidean space carrying a representation of the spin cover of the conformal group; here we invoke a variation of the method of finding the eigenvalues on $K$-types of Knapp--Stein operators. The result is a complete classification of symmetry breaking operators from spinors on the Euclidean space to spinors on a hyperplane.

Let us now explain our results in more detail.

\subsection{The translation principle}

Let $G$ be a real reductive group and $H\subseteq G$ a reductive subgroup. For parabolic subgroups $P=MAN\subseteq G$ and $P_H=M_HA_HN_H\subseteq H$ we form the generalized principal series representations (smooth normalized parabolic induction)
\begin{gather*}
 \Ind_P^G\big(\xi\otimes e^\lambda\otimes\1\big) \qquad \mbox{and} \qquad \Ind_{P_H}^H\big(\eta\otimes e^\nu\otimes\1\big),
\end{gather*}
where $\xi$ and $\eta$ are finite-dimensional representations of $M$ and $M_H$, and $\lambda\in\fraka_\CC^*$, $\nu\in\fraka_{H,\CC}^*$, the complexified duals of the Lie algebras of $A$ and $A_H$. Consider the space
\begin{gather*}
 \Hom_H\big(\Ind_P^G\big(\xi\otimes e^\lambda\otimes\1\big),\Ind_{P_H}^H\big(\eta\otimes e^\nu\otimes\1\big)\big)
\end{gather*}
of continuous $H$-intertwining operators. Realizing $\Ind_P^G\big(\xi\otimes e^\lambda\otimes\1\big)$ and $\Ind_{P_H}^H\big(\eta\otimes e^\nu\otimes\1\big)$ on the spaces of smooth sections of the vector bundles
\begin{gather*}
 \calV_\lambda=G\times_P\big(\xi\otimes e^{\lambda+\rho}\otimes\1\big)\to G/P \qquad \mbox{and} \qquad \calW_\nu=H\times_{P_H}\big(\eta\otimes e^{\nu+\rho_H}\otimes\1\big)\to H/P_H,
\end{gather*}
where $\rho\in\fraka^*$ and $\rho_H\in\fraka_H^*$ are the half sums of positive roots, one can identify continuous $H$-intertwining operators with their distribution kernels. More precisely, Kobayashi--Speh~\cite{KS15} showed that taking distribution kernels is a linear isomorphism
\begin{gather*}
 \Hom_H\big(\Ind_P^G\big(\xi\otimes e^\lambda\otimes\1\big),\Ind_{P_H}^H\big(\eta\otimes e^\nu\otimes\1\big)\big) \stackrel{\sim}{\to} \big(\calD'(G/P,\calV_\lambda^*)\otimes W_\nu\big)^{\Delta(P_H)}, \qquad\! A \mapsto K^A,
\end{gather*}
where $\calV_\lambda^*$ is the dual bundle of $\calV_\lambda$ and $W_\nu$ the $P_H$-representation defining $\calW_\nu$ (see Section~\ref{sec:DistributionSections} for the precise definition).

Now suppose $(\tau,E)$ is an irreducible finite-dimensional representation of $G$. Then the restriction $\tau|_P$ of $\tau$ to $P$ contains a unique irreducible subrepresentation $i\colon E'\hookrightarrow E$ (see Lemma~\ref{lem:UniquePstableSubspace}). With respect to the Langlands decomposition $P=MAN$ this subrepresentation is of the form $\tau'\otimes e^{\mu'}\otimes\1$ for some irreducible finite-dimensional representation $\tau$ of $M$ and $\mu'\in\fraka^*$. Let $(E')^\vee=\Hom_\CC(E',\CC)$ denote the dual of $E'$.

\begin{thmalph}[see Theorem~\ref{thm:TranslationPrinciple} and Proposition~\ref{prop:TranslationPrincipleDistributionKernel}]\label{thm:IntroA}
For every $P_H$-equivariant quotient $p\colon E|_{P_H}\allowbreak \twoheadrightarrow E''$ with $E''\simeq\tau''\otimes e^{\mu''}\otimes\1$ as $P_H$-representations, there is a unique linear map
\begin{gather*}
\Phi\colon \ \Hom_H\big(\Ind_P^G\big(\xi\otimes e^\lambda\otimes\1\big),\Ind_{P_H}^H\big(\eta\otimes e^\nu\otimes\1\big)\big)\\
\qquad {}\to \Hom_H\big(\Ind_P^G\big((\xi\otimes\tau')\otimes e^{\lambda+\mu'}\otimes\1\big),\Ind_{P_H}^H\big((\eta\otimes\tau'')\otimes e^{\nu+\mu''}\otimes\1\big)\big)
\end{gather*}
with the property that for every intertwining operator $A$ with distribution kernel $K^A$, the distribution kernel $K^{\Phi(A)}$ of $\Phi(A)$ is given by $K^{\Phi(A)}=\varphi\otimes K^A$, the multiplication of $K^A$ with the smooth section $\varphi\in C^\infty\big(G/P,G\times_P(E')^\vee\big)\otimes E''$ defined by
\begin{gather*} \varphi(g) = p\circ\tau(g)\circ i \in \Hom_\CC(E',E'') \simeq (E')^\vee\otimes E'', \qquad g\in G. \end{gather*}
\end{thmalph}

The translation principle allows to construct new intertwining operators from existing ones. But one can also reverse the roles and in some cases use the translation principle to classify intertwining operators (see, e.g., Theorem~\ref{thm:SBOsOppositeSignClassification}).

We note that although $\varphi$ is a non-trivial analytic section, the map $\Phi$ might be trivial for certain parameters. However, being a multiplication operator, $\Phi$ behaves nicely when applied to holomorphic/meromorphic families of intertwining operators (see Remark~\ref{rem:PropertiesMultMap} for details).

\subsection{Knapp--Stein type intertwining operators}

Now assume that $(G,H)$ is a symmetric pair, i.e., $H$ is an open subgroup of the fixed points~$G^\sigma$ of an involution $\sigma$ of~$G$. For simplicity we further assume that $G$ is in the Harish-Chandra class. Let $P=MAN\subseteq G$ be a $\sigma$-stable parabolic subgroup, then $P_H=P\cap H$ is a parabolic subgroup of $H$ and we write $P_H=M_HA_HN_H$ for its Langlands decomposition.

In our previous paper \cite{MOO16} with Y.~Oshima we constructed meromorphic families of intertwining operators between the spherical principal series representations $\Ind_P^G\big(\1\otimes e^\lambda\otimes\1\big)$ and $\Ind_{P_H}^H\big(\1\otimes e^\nu\otimes\1\big)$. We now generalize this construction and obtain intertwining operators between vector-valued principal series representations using the translation principle.

Assume that $P$ and its opposite parabolic $\overline{P}$ are conjugate via the Weyl group, i.e., $P=\tilde{w}_0\overline{P}\tilde{w}_0^{-1}$, where $\tilde{w}_0$ is a representative of the longest Weyl group element $w_0$. Then for any finite-dimensional representation $\xi$ of $M$ and $\alpha,\beta\in\fraka_\CC^*$ we consider the function
\begin{gather*} K_{\xi,\alpha,\beta}(g) = \xi\big(m\big(\tilde{w}_0^{-1}g^{-1}\big)\big)^{-1}a\big(\tilde{w}_0^{-1}g^{-1}\big)^\alpha a\big(\tilde{w}_0^{-1}g^{-1}\sigma(g)\big)^\beta, \qquad g\in G, \end{gather*}
where $m(g)$ and $a(g)$ are the densely defined projections of $g\in\overline{N}MAN$ onto the $M$- and $A$-component. For $\xi=\1$ the trivial representation, these are the kernel functions constructed in \cite{MOO16}. Since $a(g)$ is only defined on an open dense subset, it may happen that the factor $a\big(\tilde{w}_0^{-1}g^{-1}\sigma(g)\big)$ is not defined for any $g\in G$, whence we additionally assume that the domain of definition for $a\big(\tilde{w}_0^{-1}g^{-1}\sigma(g)\big)$ is not empty. In \cite{MOO16} we showed that in this case $K_{\xi,\alpha,\beta}(g)$ is defined on an open dense subset in $G$, and also gave a criterion to check this.

Under the above assumptions, the translation principle combined with our previous results from \cite{MOO16} yields:

\begin{coralph}[see Corollary~\ref{cor:VectorValuedKnappSteinKernels}]\label{cor:IntroB}
Assume that the finite-dimensional representation $\xi$ of $M$ is extendible to $G$ $($see Section~{\rm \ref{sec:VectorValuedKnappSteinKernels}} for the precise definition$)$. Then the functions $K_{\xi,\alpha,\beta}(g)$ extend to a meromorphic family of distributions
\begin{gather*} K_{\xi,\alpha,\beta} \in (\calD'(G/P,\calV_\lambda^*)\otimes W_\nu)^{\Delta(P_H)}, \qquad \alpha,\beta\in\fraka_\CC^*, \end{gather*}
where
\begin{gather}
\lambda = -w_0\alpha+\sigma\beta-w_0\beta+\rho, \qquad \nu = -\alpha|_{\fraka_{H,\CC}}-\rho_H.\label{eq:IntroLambdaNu}
\end{gather}
Therefore, they give rise to a meromorphic family of intertwining operators
\begin{gather*} A(\xi,\alpha,\beta) \in \Hom_H\big(\Ind_P^G\big(\tilde{w}_0\xi\otimes e^\lambda\otimes\1\big),\Ind_{P_H}^H(\xi|_{M_H}\otimes e^\nu\otimes\1)\big). \end{gather*}
\end{coralph}

As in \cite[Corollary B]{MOO16} this construction gives in particular lower bounds on multiplicities:
\begin{gather*} \dim\Hom_H\big(\Ind_P^G\big(\tilde{w}_0\xi\otimes e^\lambda\otimes\1\big),\Ind_{P_H}^H\big(\xi|_{M_H}\otimes e^\nu\otimes\1\big)\big) \geq 1 \end{gather*}
for all parameters $(\lambda,\nu)$ of the form \eqref{eq:IntroLambdaNu}.

\subsection{Symmetry breaking operators for rank one orthogonal groups}

We illustrate the translation principle in two examples, the first one being scalar-valued principal series representations for the symmetric pair $(G,H)=(\upO(n+1,1),\upO(n,1))$. The parabolic subgroups $P$ and $P_H$ satisfy
\begin{gather*} M_H \simeq \upO(n-1)\times\upO(1) \subseteq \upO(n)\times\upO(1) \simeq M, \\
 A_H = A \simeq \RR_+, \qquad N_H \simeq \RR^{n-1} \subseteq \RR^n \simeq N. \end{gather*}
We identify $\fraka_\CC^*\simeq\CC$ such that $\rho=\frac{n}{2}$, and denote by $\sgn$ the non-trivial character of $O(1)\subseteq M_H\subseteq M$. Then for $\delta,\varepsilon\in\ZZ/2\ZZ$ and $\lambda,\nu\in\CC$ we consider intertwining operators between
\begin{gather*} \pi_{\lambda,\delta} = \Ind_P^G\big(\sgn^\delta\otimes e^\lambda\otimes\1\big) \qquad \mbox{and} \qquad \tau_{\nu,\varepsilon} = \Ind_{P_H}^H\big(\sgn^\varepsilon\otimes e^\nu\otimes\1\big). \end{gather*}
Tensoring with characters of $G$, it is easy to see that
\begin{gather*} \Hom_H(\pi_{\lambda,\delta}|_H,\tau_{\nu,\varepsilon}) \simeq \Hom_H(\pi_{\lambda,1-\delta}|_H,\tau_{\nu,1-\varepsilon}). \end{gather*}

For the pair $(G,H)$, the restriction of a distribution kernel $K\in(\calD'(G/P,\calV_\lambda^*)\otimes W_\nu)^{\Delta(P_H)}$ to the open dense Bruhat cell in $G/P$, which is isomorphic to $\overline{N}\simeq\RR^n$, defines an isomorphism onto a subspace of $\calD'\big(\RR^n\big)$ (see Kobayashi--Speh~\cite[Theorem 3.16]{KS15} or Theorem~\ref{thm:DistributionKernelOnNbar} for details). Let $\calD'\big(\RR^n\big)_{\lambda,\nu}^+$, resp.~$\calD'\big(\RR^n\big)_{\lambda,\nu}^-$, denote the space of distribution kernels defining intertwining operators in $\Hom_H(\pi_{\lambda,\delta}|_H,\tau_{\nu,\varepsilon})$ with $\delta+\varepsilon\equiv0(2)$, resp.~$\delta+\varepsilon\equiv1(2)$.

The space $\calD'\big(\RR^n\big)_{\lambda,\nu}^+$ was classified by Kobayashi--Speh~\cite{KS15}, and we briefly describe this classification, borrowing their notation. First, they construct a meromorphic family of distributions $K^{\AA,+}_{\lambda,\nu}\in\calD'\big(\RR^n\big)_{\lambda,\nu}^+$, $(\lambda,\nu)\in\CC^2$, given by
\begin{gather*} K^{\AA,+}_{\lambda,\nu}(x',x_n) = |x_n|^{\lambda+\nu-\frac{1}{2}}\big(|x'|^2+x_n^2\big)^{-\nu-\frac{n-1}{2}}, \qquad (x',x_n)\in\RR^{n-1}\times\RR=\RR^n. \end{gather*}
Then \looseness=-1 they show that every distribution in $\calD'\big(\RR^n\big)_{\lambda,\nu}^+$ is given by $K_{\lambda,\nu}^{\AA,+}$ or a regularization of it. By a detailed analysis of the meromorphic nature, the poles and all possible residues of the fa\-mi\-ly~$K^{\AA,+}_{\lambda,\nu}$ they obtain a complete description of $\calD'\big(\RR^n\big)_{\lambda,\nu}^+$ for all $(\lambda,\nu)\in\CC^2$. More precisely, let
\begin{gather*} L_\even = \big\{(-\rho-i,-\rho_H-j)\colon i,j\in\NN,i-j\in2\NN\big\}, \end{gather*}
then the corresponding statement for symmetry breaking operators is: For $\delta+\varepsilon\equiv0(2)$ we have
\begin{gather*} \dim\Hom_H(\pi_{\lambda,\delta}|_H,\tau_{\nu,\varepsilon}) = \dim\calD'\big(\RR^n\big)_{\lambda,\nu}^+ = \begin{cases}2,&\mbox{for $(\lambda,\nu)\in L_\even$,}\\ 1,&\mbox{for $(\lambda,\nu)\in\CC^2\setminus L_\even$,}\end{cases} \end{gather*}
and \looseness=-1 every intertwining operator is given by the distribution kernel $K^{\AA,+}_{\lambda,\nu}$ or a regularization of it.

We apply the translation principle to the kernels $K^{\AA,+}_{\lambda,\nu}$ to obtain a meromorphic family $K_{\lambda,\nu}^{\AA,-}\in\calD'\big(\RR^n\big)_{\lambda,\nu}^-$ given by
\begin{gather*} K^{\AA,-}_{\lambda,\nu}(x) = x_n\cdot K^{\AA,+}_{\lambda-1,\nu}(x) = \sgn(x_n)|x_n|^{\lambda+\nu-\frac{1}{2}}\big(|x'|^2+x_n^2\big)^{-\nu-\frac{n-1}{2}}, \qquad x\in\RR^n. \end{gather*}
In Theorem~\ref{thm:SBOsOppositeSign} we derive from the poles and residues of $K^{\AA,+}_{\lambda,\nu}$ all poles and all possible residues of the family $K^{\AA,-}_{\lambda,\nu}$ and use them to classify intertwining operators. For the statement let
\begin{gather*} L_\odd = \big\{(-\rho-i,-\rho_H-j)\colon i,j\in\NN,i-j\in2\NN+1\big\}. \end{gather*}

\begin{thmalph}[see Theorems~\ref{thm:SBOsOppositeSign} and \ref{thm:SBOsOppositeSignClassification}]\label{thm:IntroC}
For $\delta+\varepsilon\equiv1(2)$ we have
\begin{gather*} \dim\Hom_H(\pi_{\lambda,\delta}|_H,\tau_{\nu,\varepsilon}) = \begin{cases}2,&\text{for $(\lambda,\nu)\in L_\odd$,}\\1,&\text{for $(\lambda,\nu)\in\CC^2\setminus L_\odd$,}\end{cases} \end{gather*}
and \looseness=-1 every intertwining operator is given by the distribution kernel $K^{\AA,-}_{\lambda,\nu}$ or a regularization of it.
\end{thmalph}

We remark that for $\lambda-\nu=-\frac{3}{2}-2\ell$, $\ell\in\NN$, there exists a residue of $K^{\AA,-}_{\lambda,\nu}$ which is supported at the origin in $\RR^n$, inducing a differential intertwining operator. In this setting $G/P\simeq S^n$ and $H/P_H\simeq S^{n-1}$, and the differential intertwining operators form the family of odd order conformally invariant differential operators $C^\infty\big(S^n,\calV_\lambda\big)\to C^\infty\big(S^{n-1},\calW_\nu\big)$ studied previously by Juhl~\cite{Juh09} (see also \cite{KOSS15}). The even order Juhl operators were already obtained as residues of the family $K^{\AA,+}_{\lambda,\nu}$ by Kobayashi--Speh~\cite{KS15}.

\subsection{Symmetry breaking operators for rank one pin groups}

The second (and more involved) illustration of the translation principle is for the symmetric pair $\big(\widetilde{G},\widetilde{H}\big)=(\Pin(n+1,1),\Pin(n,1))$. Here $\Pin(p,q)$ denotes a certain double cover of the group $\upO(p,q)$ (see Appendix~\ref{app:CliffordSpin} for details) so that we have compatible double covers $\widetilde{G}\to G$ and $\widetilde{H}\to H$ with $G$ and $H$ as in the previous section. The preimages of the parabolic subgroups~$P$ and $P_H$ under the covering maps are $\widetilde{P}\simeq\widetilde{M}AN$ and $\widetilde{P}_H\simeq\widetilde{M}_HA_HN_H$ with $A$, $N$, $A_H$ and~$N_H$ as in the previous section and
\begin{gather*} \widetilde{M}_H \simeq \Pin(n-1)\times\upO(1) \subseteq \Pin(n)\times\upO(1) \simeq \widetilde{M}. \end{gather*}
The fundamental representations of the Lie algebra $\so(n)$ of $\widetilde{M}$ are the exterior power representations $\wedge^k\RR^n$ and the spin representations. Symmetry breaking operators for the exterior power representations have been investigated in detail by Kobayashi--Speh~\cite{KS18} (see also Fischmann--Juhl--Somberg~\cite{FJS16} and Kobayashi--Kubo--Pevzner~\cite{KKP16} for the case of differential symmetry breaking operators). Here we focus on the fundamental spin representations.

The group $\Pin(n)$ can be realized inside the Clifford algebra $\Cl(n)$ with $n$ generators (see Appendix~\ref{app:CliffordSpin} for details). The fundamental spin representations of $\Pin(n)$ are the restrictions of irreducible representations of the complex Clifford algebra $\Cl(n;\CC)=\Cl(n)\otimes_\RR\CC$ to $\Pin(n)$, and therefore we do not distinguish between the representations of $\Cl(n;\CC)$ and their restrictions to $\Pin(n)$. For even $n$ the Clifford algebra $\Cl(n;\CC)$ has a unique irreducible representation, and for odd $n$ it has two inequivalent irreducible representations. Let $(\zeta_n,\SS_n)$ be an irreducible representation of $\Cl(n;\CC)$ and $(\zeta_{n-1},\SS_{n-1})$ an irreducible representation of $\Cl(n-1;\CC)$. If $\sgn$ denotes the non-trivial representation of $\upO(1)$, we have for all $\lambda,\nu\in\CC$ and $\delta,\varepsilon\in\ZZ/2\ZZ$ principal series representations
\begin{gather*}
\slashed{\pi}_{\lambda,\delta} = \Ind_{\widetilde{P}}^{\widetilde{G}}\big(\big(\zeta_n\otimes\sgn^\delta\big)\otimes e^\lambda\otimes\1\big) \qquad \mbox{and} \qquad \slashed{\tau}_{\nu,\varepsilon} = \Ind_{\widetilde{P}_H}^{\widetilde{H}}\big(\big(\zeta_{n-1}\otimes\sgn^\varepsilon\big)\otimes e^\nu\otimes\1\big),
\end{gather*}
and we study intertwining operators in the space
\begin{gather*} \Hom_{\widetilde{H}}\big(\slashed{\pi}_{\lambda,\delta}|_{\widetilde{H}},\slashed{\tau}_{\nu,\varepsilon}\big). \end{gather*}
As above, taking distribution kernels and restricting them to the open dense Bruhat cell in $G/P$ identifies the space of intertwining operators with a subspace
\begin{gather*} \calD'\big(\RR^n;\Hom_\CC(\SS_n,\SS_{n-1})\big)_{\lambda,\nu}^\pm \subseteq \calD'\big(\RR^n;\Hom_\CC(\SS_n,\SS_{n-1})\big) \simeq \calD'\big(\RR^n\big)\otimes\Hom_\CC(\SS_n,\SS_{n-1}), \end{gather*}
where the sign $+$ represents $\delta+\varepsilon\equiv0(2)$ and the sign $-$ represents $\delta+\varepsilon\equiv1(2)$.

The translation principle applied to the distribution kernels $K_{\lambda,\nu}^{\AA,\pm}\in\calD'\big(\RR^n\big)_{\lambda,\nu}^\pm$ yields meromorphic families of distribution kernels $P\Kslash_{\lambda,\nu}^{\AA,\pm}\in\calD'\big(\RR^n;\Hom_\CC(\SS_n,\SS_{n-1})\big)_{\lambda,\nu}^\pm$ given by
\begin{gather*} P\Kslash_{\lambda,\nu}^{\AA,\pm}(x) = (P\zeta_n(x))\cdot K_{\lambda-\frac{1}{2},\nu+\frac{1}{2}}^{\AA,\mp}(x), \end{gather*}
where $\zeta_n(x)\in\End_\CC(\SS_n)$ is the value of the representation $\zeta_n$ of the Clifford algebra $\Cl(n;\CC)$ at the vector $x\in\RR^n\subseteq\Cl(n)\subseteq\Cl(n,\CC)$ and $0\neq P\in\Hom_{\Pin(n-1)}([\zeta_n\otimes\det]|_{\Pin(n-1)},\zeta_{n-1})$. (Note that the spin representation $\zeta_{n-1}$ of $\Pin(n-1)$ occurs in the restriction $[\zeta_n\otimes\det]|_{\Pin(n-1)}$ with multiplicity one, where $\det\colon \Pin(n)\to\upO(n)\to\{\pm1\}$ denotes the determinant character.)

To state our result on the classification of intertwining operators, let
\begin{gather*}
\Lslash_\even = \big\{\big({-}\rho-\tfrac{1}{2}-i,-\rho_H-\tfrac{1}{2}-j\big)\colon i,j\in\NN,i-j\in2\NN\big\},\\
\Lslash_\odd = \big\{\big({-}\rho-\tfrac{1}{2}-i,-\rho_H-\tfrac{1}{2}-j\big)\colon i,j\in\NN,i-j\in2\NN+1\big\}.
\end{gather*}

\begin{thmalph}[see Theorems~\ref{thm:SBOsSpinors+}, \ref{thm:SBOsSpinors-} and \ref{thm:ClassificationSBOsSpinors}]\label{thm:IntroD}\quad
\begin{enumerate}\itemsep=0pt
\item[$1.$] For $\delta+\varepsilon\equiv 0(2)$ we have
\begin{gather*} \dim\Hom_{\widetilde{H}}\big(\slashed{\pi}_{\lambda,\delta}|_{\widetilde{H}},\slashed{\tau}_{\nu,\varepsilon}\big) = \begin{cases}2,&\mbox{for $(\lambda,\nu)\in\Lslash_\even$,}\\ 1,&\mbox{for $(\lambda,\nu)\in\CC^2\setminus\Lslash_\even$,}\end{cases} \end{gather*}
and every intertwining operator is given by the distribution kernel $P\Kslash_{\lambda,\nu}^{\AA,+}$ or a regularization of it.
\item[$2.$] For $\delta+\varepsilon\equiv 1(2)$ we have
\begin{gather*} \dim\Hom_{\widetilde{H}}\big(\slashed{\pi}_{\lambda,\delta}|_{\widetilde{H}},\slashed{\tau}_{\nu,\varepsilon}\big) = \begin{cases}2,&\mbox{for $(\lambda,\nu)\in\Lslash_\odd$,}\\1, &\mbox{for $(\lambda,\nu)\in\CC^2\setminus\Lslash_\odd$,}\end{cases} \end{gather*}
and every intertwining operator is given by the distribution kernel $P\Kslash_{\lambda,\nu}^{\AA,-}$ or a regularization of it.
\end{enumerate}
\end{thmalph}

In Theorems~\ref{thm:SBOsSpinors+} and \ref{thm:SBOsSpinors-} we determine all poles and residues of the meromorphic families $P\Kslash_{\lambda,\nu}^{\AA,\pm}$, and hence give an explicit description of the distribution kernels of all intertwining operators between spinor-valued principal series representations.

We remark that for $\lambda-\nu=-\frac{1}{2}-2\ell$, resp.~$\lambda-\nu=-\frac{3}{2}-2\ell$, $\ell\in\NN$, there exists a residue of~$P\Kslash_{\lambda,\nu}^{\AA,+}$, resp.~$P\Kslash_{\lambda,\nu}^{\AA,-}$, which is supported at the origin in $\RR^n$, inducing a differential intertwi\-ning operator. These families of spinor-valued differential intertwining operators were previously obtained by Kobayashi--{\O}rsted--Somberg--Sou{\v{c}}ek~\cite{KOSS15}, and it was conjectured that these are all differential intertwining operators in this setting. Our classification confirms this conjecture (see Remark~\ref{rem:JuhlOperatorsSpinors}).

In contrast to the proof of Theorem~\ref{thm:IntroC} which only uses the translation principle, we employ the method developed in \cite{MO14a} for the proof of Theorem~\ref{thm:IntroD}. This method describes intertwining operators between the underlying Harish-Chandra modules of principal series representations in terms of their action on the different $K$-types. The explicit knowledge of the action on $K$-types also allows us to determine the dimensions of intertwining operators between the irreducible constituents of $\slashed{\pi}_{\lambda,\delta}$ and $\slashed{\tau}_{\nu,\varepsilon}$ at reducibility points.

The representation $\slashed{\pi}_{\lambda,\delta}$ is reducible if and only if $\lambda=\pm\big(\rho+\frac{1}{2}+i\big)$, $i\in\NN$. More precisely, for $\lambda=-\rho-\frac{1}{2}-i$ the representation $\slashed{\pi}_{\lambda,\delta}$ has a finite-dimensional irreducible subrepresentation~$\calF_\delta(i)$ and the quotient $\calT_\delta(i)=\slashed{\pi}_{\lambda,\delta}/\calF_\delta(i)$ is irreducible. The composition factors at $\lambda=\rho+\frac{1}{2}+i$ can be described in terms of $\calF_\delta(i)$ and $\calT_\delta(i)$ by tensoring with the determinant character (see Lemma~\ref{lem:SpinReducibility} for the precise statement). We use the analogous notation for the composition factors~$\calF'_\varepsilon(j)$ and $\calT'_\varepsilon(j)$ of~$\slashed{\tau}_{\nu,\varepsilon}$ at $\nu=-\rho_H-\frac{1}{2}-j$, $j\in\NN$.

\begin{thmalph}[see Theorem~\ref{thm:MultiplicitiesCompositionFactors}]\label{thm:IntroE}
For $\pi\in\{\calT_\delta(i),\calF_\delta(i)\}$ and $\tau\in\{\calT'_\varepsilon(j),\calF'_\varepsilon(j)\}$ the multiplicities $\dim\Hom_{\widetilde{H}}(\pi|_{\widetilde{H}},\tau)$ are given by
\begin{center}
\begin{tabular}{c|cc}
\diagonal{.2em}{.83cm}{$\pi$}{$\tau$} & $\calF'_\varepsilon(j)$ & $\calT'_\varepsilon(j)$ \\
\hline
$\calF_\delta(i)$ & $1$ & $0$\\
$\calT_\delta(i)$ & $0$ & $1$\\
\multicolumn{3}{c}{for $0\leq j\leq i$,}\\
\multicolumn{3}{c}{$i+j\equiv\delta+\varepsilon(2)$,}
\end{tabular}
\qquad
\begin{tabular}{c|cc}
\diagonal{.2em}{.83cm}{$\pi$}{$\tau$} & $\calF'_\varepsilon(j)$ & $\calT'_\varepsilon(j)$ \\
\hline
$\calF_\delta(i)$ & $0$ & $0$\\
$\calT_\delta(i)$ & $1$ & $0$\\
\multicolumn{3}{c}{otherwise.}\\
\multicolumn{3}{c}{\phantom{otherwise.}}
\end{tabular}
\end{center}
\end{thmalph}

We remark that the multiplicities are the same as in the case of spherical principal series (see \cite[Theorem 1.2]{KS15}).

\subsection{Relation to conformal geometry}

Let $(X,g)$ be a connected oriented Riemannian manifold of dimension $n$ with a spin structure. Let $G$ denote the conformal group of $X$, i.e., the group of diffeomorphisms $h\colon X\to X$ such that there exists a conformal factor $\Omega(h,\cdot)\in C^\infty(X)$ with $(h^*g)_{gx}=\Omega(h,x)^2g_x$ for all $h\in G$ and $x\in X$. Write ${\rm or}\colon G\to\{\pm1\}$ for the character of $G$ which takes the value $+1$ on orientation preserving diffeomorphisms and $-1$ on orientation reversing diffeomorphisms. Let $\Sigma X\to X$ denote the spin bundle and $C^\infty(\Sigma X)$ its smooth sections, also called spinors. For a spinor $\sigma\in C^\infty(\Sigma X)$ and a diffeomorphism $h\in G$ the pullback $h^*\sigma$ can in general not be defined unambiguously, but there exists a double covering $\widetilde{G}\to G$ and a smooth action of $\widetilde{G}$ on $C^\infty(\Sigma X)$ which resolves this ambiguity. Abusing notation, we lift the orientation character ${\rm or}$ and the conformal factor $\Omega$ to the double cover $\widetilde{G}$. This gives rise to a family $\varpi_{\lambda,\delta}$ of representations of~$\widetilde{G}$ on $C^\infty(\Sigma X)$ depending on two parameters $\lambda\in\CC$ and $\delta\in\ZZ/2\ZZ$ given by
\begin{gather*} \slashed{\pi}_{\lambda,\delta}(h)\sigma(x) = {\rm or}(h)^\delta\Omega\big(h^{-1},x\big)^{\lambda+\frac{n}{2}}(h^*\sigma)(x), \qquad x\in X,h\in G,\sigma\in C^\infty(\Sigma X). \end{gather*}
For $X=S^n$ with the Euclidean metric the group $\widetilde{G}$ is essentially ${\rm Pin}(n+1,1)$ and the definition of $\slashed{\pi}_{\lambda,\delta}$ agrees with previous definition since $G/P\simeq S^n$.

For an oriented submanifold $Y\subseteq X$ we may consider the group $H=\{h\in G\colon hY=Y\}$ which acts conformally on $(Y,g|_Y)$. Fixing a spin structure on $Y$ we denote by $C^\infty(\Sigma Y)$ the space of spinors and by $\slashed{\tau}_{\nu,\varepsilon}$ ($\nu\in\CC$, $\varepsilon\in\ZZ/2\ZZ$) the corresponding representations of~$H$ on~$C^\infty(\Sigma Y)$.

In this context it is natural to ask for a construction and classification of (differential) ope\-rators $A\colon C^\infty(\Sigma X)\to C^\infty(\Sigma Y)$ such that $A\circ\slashed{\pi}_{\lambda,\delta}(h)=\slashed{\tau}_{\nu,\varepsilon}(h)\circ A$ for all $h\in H$. The analogous question for differential forms was previously discussed by Kobayashi--Kubo--Pevzner~\cite{KKP16} and Kobayashi--Speh~\cite{KS18}. In the model case $(X,Y)=\big(S^n,S^{n-1}\big)$ a complete classification for differential forms was obtained in \cite{FJS16,KKP16,KS18}, and our results in Theorem~\ref{thm:IntroD} can be viewed as the analogous classification for spinors. We also refer to~\cite{CO14} for the case $X=Y=S^n$.

\subsection{Structure of the paper}

In Section~\ref{sec:Preliminaries} we fix the notation for (generalized) principal series representations of real reductive groups, explain how to describe intertwining operators between them in terms of invariant distributions, and recall the construction of Knapp--Stein type intertwining operators for symmetric pairs from our joint work with Y.~Oshima~\cite{MOO16}. Section~\ref{sec:TranslationPrinciple} explains the idea of the translation principle in detail and contains the proofs of Theorem~\ref{thm:IntroA} and Corollary~\ref{cor:IntroB}. We then apply the translation principle in two different situations. Firstly, in Section~\ref{sec:ExScalar} we construct and classify intertwining operators between principal series representations of $(G,H)=(\upO(n+1,1),\upO(n,1))$ induced from one-dimensional representations (see Theorems~\ref{thm:SBOsOppositeSign} and \ref{thm:SBOsOppositeSignClassification} for a detailed version of Theorem~\ref{thm:IntroC}). Secondly, in Section~\ref{sec:ExSpinors} we construct intertwining operators between principal series representations of $\big(\widetilde{G},\widetilde{H}\big)=(\Pin(n+1,1),\Pin(n,1))$ induced from spin-representations (see Theorems~\ref{thm:SBOsSpinors+} and~\ref{thm:SBOsSpinors-}). To also obtain a classification in the second situation, we employ in Section~\ref{sec:CompactPictureSBOsSpinors} the method developed in~\cite{MO14a}, yielding the classification results in Theorems~\ref{thm:ClassificationSBOsSpinors} and~\ref{thm:MultiplicitiesCompositionFactors}. Together with Theorems~\ref{thm:SBOsSpinors+} and~\ref{thm:SBOsSpinors-} this proves Theorems~\ref{thm:IntroD} and \ref{thm:IntroE}. Finally, Appendix~\ref{app:CliffordSpin} contains some elementary material about Clifford algebras and spin representations needed in Sections~\ref{sec:ExSpinors} and \ref{sec:CompactPictureSBOsSpinors}.

{\bf Notation.} $\NN=\{0,1,2,\ldots\}$, $(\lambda)_n=\lambda(\lambda+1)\cdots(\lambda+n-1)$, $A\setminus B=\{a\in A\colon a\notin B\}$.

\vspace{-1.5mm}

\section{Preliminaries}\label{sec:Preliminaries}

\vspace{-1mm}

We recall the basic facts about (generalized) principal series representations of real reductive groups, symmetry breaking operators, and their construction for symmetric pairs. More details can be found in~\cite{KS15,MOO16}.

\subsection{Generalized principal series representations}\label{sec:PrincipalSeries}

Let $G$ be a real reductive group and $P\subseteq G$ a parabolic subgroup with Langlands decomposition $P=MAN$. We write $\frakg$, $\frakm$, $\fraka$ and $\frakn$ for the Lie algebras of $G$, $M$, $A$ and $N$. Then $A=\exp(\fraka)$ and $N=\exp(\frakn)$. Let $\overline{\frakn}$ denote the nilradical opposite to $\frakn$ and $\overline{N}=\exp(\overline{\frakn})$ the corresponding closed connected subgroup of $G$. Then $\overline{P}=MA\overline{N}\subseteq G$ is the parabolic subgroup opposite to~$P$.

The multiplication map $\overline{N}\times M\times A\times N\to G$ is a diffeomorphism onto an open dense subset of $G$, and for $g\in\overline{N}MAN$ we define $\overline{n}(g)\in\overline{N}$, $m(g)\in M$, $a(g)\in A$ and $n(g)\in N$ by $g=\overline{n}(g)m(g)a(g)n(g)$.

We consider representations of $G$ which are parabolically induced from finite-dimensional representations of $P$. Let $(\xi, V)$ be a finite-dimensional representation of $M$, $\lambda\in \fraka_{\CC}^*$ and denote by $\1$ the trivial representation of $N$, then $\xi\otimes e^\lambda\otimes\1$ is a finite-dimensional representation of $P=MAN$. Write $\rho=\frac{1}{2}\tr\ad_\frakn\in\fraka^*$ for half the sum of all positive roots and let $V_\lambda=\xi\otimes e^{\lambda+\rho}\otimes\1$. We define the generalized principal series representation $\Ind^G_P\big(\xi\otimes e^{\lambda}\otimes \1\big)$ as the left-regular representation on the space
\begin{gather*} C^\infty(G,V_\lambda)^P = \big\{f\in C^\infty(G,V)\colon f(gman) = a^{-\lambda-\rho}\xi(m)^{-1}f(g)\,\forall\,g\in G,\, man\in MAN\big\}. \end{gather*}
If we write
\begin{gather*} \calV_\lambda = G\times_P V_\lambda\to G/P \end{gather*}
for the homogeneous vector bundle associated to the representation $V_\lambda$ of $P$, then $C^\infty(G,V_\lambda)^P$ can be identified with the space $C^\infty(G/P,\calV_\lambda)$ of smooth sections of~$\calV_\lambda$.

\subsection{Distribution sections of vector bundles}\label{sec:DistributionSections}

Let $V_\lambda^*=\xi^\vee\otimes e^{-\lambda+\rho}\otimes\1$, where $\xi^\vee$ is the contragredient representation of $\xi$, and write $\calV_\lambda^*=G\times_P V_\lambda^*$ for the dual bundle of $\calV_\lambda$. We define the space $\calD'(G/P,\calV_\lambda)$ of distribution sections of the bundle $\calV_\lambda$ as the (topological) dual of $C^\infty(G/P,\calV_\lambda^*)$:
\begin{gather*} \calD'(G/P,\calV_\lambda) = C^\infty(G/P,\calV_\lambda^*)'. \end{gather*}
Note that since $G/P$ is compact, smooth sections on $G/P$ are automatically compactly supported. Then $C^\infty(G/P,\calV_\lambda)\simeq C^\infty(G,V_\lambda)^P$ embeds $G$-equivariantly into $\calD'(G/P,\calV_\lambda)$ by \smash{$f\mapsto T_f$}, where
\begin{gather*}
\langle T_f,\varphi\rangle = \int_K \langle f(k),\varphi(k)\rangle\,{\rm d}k \qquad \forall\varphi\in C^\infty(G,V_\lambda^*)^P,
\end{gather*}
and ${\rm d}k$ denotes the normalized Haar measure on $K$.

Let $(\tau,E)$ be another finite-dimensional representation of $P$ and $\calE=G\times_P E$ the corresponding vector bundle over $G/P$. Then every smooth section $f\in C^\infty(G/P,\calE)$ defines a continuous linear multiplication operator
\begin{gather*} \calD'(G/P,\calV_\lambda)\to\calD'(G/P,\calE\otimes\calV_\lambda), \qquad u\mapsto f\otimes u, \end{gather*}
which is dual to the composition
\begin{gather*} C^\infty(G,E^\vee\otimes V_\lambda^*)^P \stackrel{f\otimes}{\to} C^\infty(G,E\otimes E^\vee\otimes V_\lambda^*)^P \stackrel{c_*}{\to} C^\infty(G,V_\lambda^*)^P \end{gather*}
of the pointwise tensor product $f\otimes$ and the push-forward by the contraction map $c\colon E\otimes E^\vee\otimes V_\lambda^*\to V_\lambda^*$.

\subsection{Symmetry breaking operators}\label{sec:SBOs}

Now let $H\subseteq G$ be a reductive subgroup of $G$ and $P_H=M_HA_HN_H\subseteq H$ a parabolic subgroup of $H$. Similarly, we define $\Ind^H_{P_H}(\eta \otimes e^{\nu}\otimes \1)$ for a finite-dimensional representation $(\eta,W)$ of $M_H$ and $\nu\in\fraka_{H,\CC}^*$ as the left-regular representation of $H$ on $C^\infty(H,W_\nu)^{P_H}$, where $W_\nu=\eta\otimes e^{\nu+\rho_H}\otimes\1$, $\rho_H=\frac{1}{2}\tr\ad_{\frakn_H}$. Then $C^\infty(H,W_\nu)^{P_H}$ identifies with the space $C^\infty(H/P_H,\calW_\nu)$ of smooth sections of the vector bundle $\calW_\nu=H\times_{P_H}W_\nu\to H/P_H$.

In this paper we study continuous $H$-intertwining operators
\begin{gather*} \Ind^G_P\big(\xi\otimes e^\lambda\otimes \1\big)\to\Ind^H_{P_H}\big(\eta \otimes e^\nu\otimes \1\big). \end{gather*}
By the Schwartz kernel theorem such maps are identified with $H$-invariant distribution sections of some vector bundle over $G/P\times H/P_H$. Using the isomorphism $\Delta(H)\backslash(G\times H)\simeq G$ which is induced by $G\times H\to G$, $(g,h)\mapsto g^{-1}h$, invariant distributions on $G/P\times H/P_H$ reduce to invariant distributions on $G/P$ (see \cite[Section~3.2]{KS15}):

\begin{Proposition}[{\cite[Proposition 3.2]{KS15}}]\label{prop:DistributionKernel}
We have natural isomorphisms of vector spaces
\begin{align*}
\Hom_H\big(\Ind^G_P\big(\xi\otimes e^{\lambda}\otimes \1\big),\Ind^H_{P_H}\big(\eta\otimes e^{\nu}\otimes \1\big)\big) &\simeq \calD'(G/P\times H/P_H,\calV_\lambda^*\otimes\calW_\nu)^{\Delta(H)}\\
&\simeq (\calD'(G/P,\calV_\lambda^*)\otimes W_\nu)^{\Delta(P_H)}.
\end{align*}
\end{Proposition}

Under the isomorphism an $H$-intertwining operator $A\colon \Ind^G_P\big(\xi\otimes e^{\lambda}\otimes\1\big)\to\Ind^H_{P_H}\big(\eta\otimes e^{\nu}\otimes\1\big)$ maps to the distribution kernel $K^A\in(\calD'(G/P,\calV_\lambda^*)\otimes W_\nu)^{\Delta(P_H)}$ such that
\begin{gather*} Af(h) = \int_{G/P} c\big(f(x)\otimes K^A\big(h^{-1}x\big)\big)\,{\rm d}x, \qquad f\in C^\infty(G/P,\calV_\lambda), \end{gather*}
where $c\colon V_\lambda\otimes V_\lambda^*\otimes W_\nu\to W_\nu$ is the contraction map, and the integral has to be understood in the distribution sense.

Sometimes it is convenient to work on the open dense Bruhat cell in $G/P$, because it is isomorphic to the vector space $\overline{\frakn}$. Under certain conditions on the parabolic subgroups~$P$ and~$P_H$ the restriction of the invariant distribution sections in Proposition~\ref{prop:DistributionKernel} to the open dense Bruhat cell is injective:

\begin{Theorem}[{\cite[Theorem 3.16]{KS15}}]\label{thm:DistributionKernelOnNbar}
Assume that
\begin{gather*} P_H=P\cap H, \qquad M_H=M\cap H, \qquad A_H=A\cap H, \qquad N_H=N\cap H. \end{gather*}
If additionally $G=P_H\overline{N}P$ then the restriction to $\overline{\frakn}\simeq\overline{N}\hookrightarrow G/P$ defines a linear isomorphism
\begin{gather*} (\calD'(G/P,\calV_\lambda^*)\otimes W_\nu)^{\Delta(P_H)} \to \calD'(\overline{\frakn},V_\lambda^*\otimes W_\nu)^{M_HA_H,\frakn_H}. \end{gather*}
\end{Theorem}

Here the action of $M_HA_H$ on $\calD'(\overline{\frakn},V_\lambda^*\otimes W_\nu)$ is the obvious action induced by the actions of $M_HA_H$ on $\overline{\frakn}$, $V_\lambda^*$ and $W_\nu$, and the action of $\frakn_H$ is induced by the infinitesimal action on $\overline{\frakn}$ viewed as subset of the generalized flag variety $G/P$.

\subsection{Symmetric pairs and Knapp--Stein type intertwining operators}\label{sec:KStypeIntertwiningOperators}

For symmetric pairs $(G,H)$ we provided in \cite{MOO16} explicit expressions of invariant distributions defining symmetry breaking operators, which we briefly recall.

Let $\sigma$ be an involution of $G$ and let $H$ be an open subgroup of $G^\sigma$, the fixed points of $\sigma$. Then $(G,H)$ forms a symmetric pair. We make the following two additional assumptions:{\samepage
\begin{align}
& \mbox{$P$ and $\overline{P}$ are conjugate via the Weyl group}\tag{G}\label{ass:G}\\
& \mbox{$P$ is $\sigma$-stable}.\tag{H}\label{ass:H}
\end{align}}

\noindent
Then \eqref{ass:G} implies $P=\tilde{w}_0\overline{P}\tilde{w}_0^{-1}$, where $\tilde{w}_0$ is a representative of the longest Weyl group element~$w_0$. Further, \eqref{ass:H} implies that $P_H=P\cap H$ is a parabolic subgroup of $H$.

Recall the $a$-projection $g\mapsto a(g)\in A$ from Section~\ref{sec:PrincipalSeries} which is defined on the open dense subset $\overline{N}MAN\subseteq G$. For $\alpha,\beta\in\fraka_\CC^*$ sufficiently positive we define
\begin{gather*} K_{\alpha,\beta}(g) = a\big(\tilde{w}_0^{-1}g^{-1}\big)^\alpha a\big(\tilde{w}_0^{-1}g^{-1}\sigma(g)\big)^\beta, \qquad g\in G. \end{gather*}
Since the second factor might not be defined for any $g\in G$, we make the additional assumption
\begin{gather}
\mbox{The domain of definition for $K_{\alpha,\beta}$ is non-empty.}\tag{D}\label{ass:D}
\end{gather}
In \cite[Proposition 2.5]{MOO16} we showed that in this case the domain of definition for $K_{\alpha,\beta}$ is already open and dense in $G$, and gave a criterion to check this.

In \cite{MOO16} we studied the meromorphic continuation of $K_{\alpha,\beta}$ in the parameters $\alpha,\beta\in\fraka_\CC^*$, and proved that they give rise to intertwining operators between spherical principal series:

\begin{Theorem}[{\cite[Theorems 3.1 and 3.3]{MOO16}}]\label{thm:KnappSteinKernels}
Under the assumptions \eqref{ass:G}, \eqref{ass:H} and \eqref{ass:D}, the functions $K_{\alpha,\beta}$ extend to a meromorphic family of distributions
\begin{gather*} K_{\alpha,\beta}\in(\calD'(G/P,\calV_\lambda^*)\otimes W_\nu)^{\Delta(P_H)}, \end{gather*}
where
\begin{gather}
\lambda = -w_0\alpha+\sigma\beta-w_0\beta+\rho, \qquad \nu = -\alpha|_{\fraka_{H,\CC}}-\rho_H.\label{eq:LambdaNuAlphaBeta}
\end{gather}
Therefore, they give rise to a meromorphic family of intertwining operators
\begin{gather*} A(\alpha,\beta)\colon \ \Ind_P^G\big(\1\otimes e^\lambda\otimes\1\big) \to \Ind_{P_H}^H\big(\1\otimes e^\nu\otimes\1\big). \end{gather*}
\end{Theorem}

\section{The translation principle}\label{sec:TranslationPrinciple}

We describe a technique, called the \textit{translation principle}, which allows to obtain new symmetry breaking operators from existing ones by tensoring with finite-dimensional representations of $G$.

\subsection{General technique}

Fix a principal series representation $\Ind_P^G\big(\xi\otimes e^\lambda\otimes\1\big)$ of $G$ and let $(\tau,E)$ be a finite-dimensional representation of~$G$, then there is a natural $G$-equivariant isomorphism
\begin{gather}
\iota_\tau\colon \ \Ind^G_P\big(\big(\xi\otimes e^{\lambda}\otimes \1\big) \otimes E|_P\big) \stackrel{\sim}{\to} \Ind^G_P\big(\xi\otimes e^{\lambda}\otimes \1\big)\otimes E.\label{eq:Map2}
\end{gather}
When we view both sides as $(V\otimes E)$-valued functions on $G$, this isomorphism is given by
\begin{gather*} (\iota_\tau f)(g)=(\id\otimes \tau(g))f(g), \qquad g\in G. \end{gather*}
Now, for any $P$-stable subspace $E'\subset E$ we have a natural injective map
\begin{align*}
\Ind^G_P\big(\big(\xi\otimes e^{\lambda}\otimes \1\big) \otimes E'\big) \hookrightarrow \Ind^G_P\big(\big(\xi\otimes e^{\lambda}\otimes \1\big) \otimes E|_P\big).
\end{align*}
Suppose that $N$ acts trivially on $E'$ and $A$ acts by a fixed character $e^{\mu'}$, $\mu\in\fraka^*$, then the $P$-action on $E'$ can be written as $\tau'\otimes e^{\mu'}\otimes \1$. The above map becomes
\begin{gather}
\Ind^G_P\big((\xi\otimes \tau')\otimes e^{\lambda+\mu'}\otimes \1\big) \hookrightarrow \Ind^G_P\big(\big(\xi\otimes e^{\lambda}\otimes \1\big) \otimes E|_P\big).\label{eq:Map1}
\end{gather}

Assuming irreducibility of $\tau$ and $\tau'$, there is essentially only one choice of such a $P$-stable subspace $E'\subseteq E$:

\begin{Lemma}\label{lem:UniquePstableSubspace}
For every irreducible finite-dimensional representation $(\tau,E)$ of $G$ the restric\-tion~$\tau|_P$ contains a unique irreducible subrepresentation~$E'$. Moreover, $E'$ is generated by the highest weight space of $E$, and $P=MAN$ acts on $E'$ by $\tau'\otimes e^{\mu'}\otimes\1$, where $\tau'$ is an irreducible representation of~$M$, $\mu'\in\fraka^*$ and $\1$ the trivial representation of $N$.
\end{Lemma}

\begin{proof}It suffices to treat the case of $G$ connected since $M$ meets all connected components of $G$. We first fix some notation. Let $\frakt\subseteq\frakm$ be a Cartan subalgebra, then $\frakc=\frakt\oplus\fraka$ is a Cartan subalgebra of $\frakg$ and $\frakc_\CC$ is a Cartan subalgebra of $\frakg_\CC$. Choose any system of positive roots $\Sigma^+(\frakg_\CC,\frakc_\CC)\subseteq\Sigma(\frakg_\CC,\frakc_\CC)$ such that the non-zero restrictions of positive roots to $\fraka$ are the roots of $\frakn$. Then the non-zero restrictions of positive roots in $\frakm$ to $\frakt_\CC$ form a positive system of roots $\Sigma^+(\frakm_\CC,\frakt_\CC)\subseteq\Sigma(\frakm_\CC,\frakt_\CC)$. We consider highest weights with respect to these positive systems.\\
Now let $E'\subseteq E$ be any irreducible subrepresentation for $\tau|_P$. Since $M$ is reductive, $E'$ decomposes into the direct sum $E'=E_1'\oplus\cdots\oplus E_m'$ of irreducible $M$-representations $E_i'$ of highest weight $\lambda_i\in\frakt_\CC^*$. Since $M$ and $A$ commute, $A$ acts by a character $\mu_i\in\fraka^*$ on $E_i'$. Now, let $1\leq i\leq m$ such that $\lambda_i+\mu_i$ is maximal among the $\lambda_j+\mu_j$, $1\leq j\leq m$. Then it is easy to see that $\tau|_N$ is trivial on $E_i'$. Hence, $E_i'\subseteq E$ is stable under $P$. Since $E'$ was assumed to be irreducible for $P$, we have $E'=E_i'$ and hence $N$ acts trivially on $E'$.\\
To show that $E'$ is unique, we simply observe that a highest weight vector for the action of $M$ on $E'$ is automatically a highest weight vector for the action of $G$ on $E$ which is unique (up to scalar multiples). Hence $E'$ is the $P$-subrepresentation of $E$ generated by the highest weight space.
\end{proof}

Further, in the case of minimal parabolic subgroups, essentially every irreducible finite-dimensional representation of $M$ extends to $G$:

\begin{Lemma}[{\cite[Theorem 2.1]{Wal71}}]\label{lem:ExistenceOfEnoughGreps}
Assume that $G$ is a linear connected reductive Lie group and $P\subseteq G$ is minimal parabolic. Then every irreducible finite-dimensional representation $(\tau',E')$ of $M$ is conjugate via the Weyl group to a representation that occurs as a direct summand in an irreducible finite-dimensional representation $(\tau,E)$ of $G$ and on which $N$ acts trivially.
\end{Lemma}

Similarly, for a fixed principal series representation $\Ind_{P_H}^H(\eta\otimes e^\nu\otimes\1)$ we have an isomorphism
\begin{gather}
\Ind_{P_H}^H\big(\eta\otimes e^\nu\otimes\1\big)\otimes E\stackrel{\sim}{\to}\Ind^H_{P_H}\big(\big(\eta\otimes e^{\nu}\otimes \1\big)\otimes E|_{P_H}\big).\label{eq:Map4}
\end{gather}
We also take a $P_H$-quotient space $E\twoheadrightarrow E''$ on which $N_H$ acts trivially and $A_H$ acts by a charac\-ter~$e^{\mu''}$. Note that such a quotient always exists since the contragredient representation $E^\vee$ of $E$ possesses a $P_H$-stable subspace on which $N_H$ acts trivially by Lemma~\ref{lem:UniquePstableSubspace}. However, in contrast to Lemma~\ref{lem:UniquePstableSubspace}, there might be several possibilities for $E''$ since $E|_H$ might not be irreducible. Denoting the $P_H$-action on $E''$ by $\tau''\otimes e^{\mu''}\otimes \1$, we get a map
\begin{gather}
\Ind^H_{P_H}\big(\big(\eta\otimes e^{\nu}\otimes \1\big) \otimes E|_{P_H}\big) \twoheadrightarrow \Ind^H_{P_H}\big((\eta\otimes \tau'')\otimes e^{\nu+\mu''}\otimes \1\big).\label{eq:Map5}
\end{gather}

Now suppose that an $H$-intertwining operator
\begin{align*}
A\colon \ \Ind^G_P\big(\xi\otimes e^{\lambda}\otimes \1\big)
\to \Ind^H_{P_H}\big(\eta\otimes e^{\nu}\otimes \1\big)
\end{align*}
is given, and form the tensor product
\begin{gather}
A\otimes\id_E\colon \ \Ind^G_P\big(\xi\otimes e^{\lambda}\otimes \1\big)\otimes E \to \Ind^H_{P_H}\big(\eta\otimes e^{\nu}\otimes \1\big)\otimes E.\label{eq:Map3}
\end{gather}
Then we obtain an $H$-intertwining operator
\begin{align*}
\Phi(A)\colon \ \Ind^G_P\big((\xi\otimes \tau') \otimes e^{\lambda+\mu'}\otimes \1\big)
\to \Ind^H_{P_H}\big((\eta\otimes \tau'')\otimes e^{\nu+\mu''}\otimes \1\big)
\end{align*}
by composing the maps \eqref{eq:Map1}, \eqref{eq:Map2}, \eqref{eq:Map3}, \eqref{eq:Map4} and \eqref{eq:Map5}, namely
\begin{align}
\Ind^G_P\big((\xi\otimes \tau') \otimes e^{\lambda+\mu'}\otimes \1\big)
\hookrightarrow{}&
\Ind^G_P\big(\big(\xi\otimes e^{\lambda}\otimes \1\big) \otimes E|_P\big)\notag\\
\stackrel{\sim}\to{}&
\Ind^G_P\big(\xi\otimes e^{\lambda}\otimes \1\big) \otimes E\notag\\
\to{}&
\Ind^H_{P_H}\big(\eta\otimes e^{\nu}\otimes \1\big)\otimes E\notag\\
\stackrel{\sim}\to{}&
\Ind^H_{P_H}\big(\big(\eta\otimes e^{\nu}\otimes \1\big)\otimes E|_{P_H}\big)\notag\\
\twoheadrightarrow{}&
\Ind^H_{P_H}\big((\eta\otimes \tau'')\otimes e^{\nu+\mu''}\otimes \1\big).\label{eq:TranslationPrincipleComposition}
\end{align}
This proves:

\begin{Theorem}\label{thm:TranslationPrinciple}Let $(\tau,E)$ be a finite-dimensional $G$-representation, $E'\subseteq E$ a $P$-stable subspace with $E'|_P=\tau'\otimes e^{\mu'}\otimes\1$ and $E\twoheadrightarrow E''$ a $P_H$-equivariant quotient with $E''|_{P_H}=\tau''\otimes e^{\mu''}\otimes\1$. Then \eqref{eq:TranslationPrincipleComposition} defines a linear map
\begin{gather*}
\Phi\colon \ \Hom_H\big(\Ind_P^G\big(\xi\otimes e^\lambda\otimes\1\big),\Ind_{P_H}^H\big(\eta\otimes e^\nu\otimes\1\big)\big)\\
\qquad{} \to \Hom_H\big(\Ind_P^G\big((\xi\otimes\tau')\otimes e^{\lambda+\mu'}\otimes\1\big),\Ind_{P_H}^H\big((\eta\otimes\tau'')\otimes e^{\nu+\mu''}\otimes\1\big)\big)
\end{gather*}
for all finite-dimensional representations $\xi$ of $M$ and $\eta$ of $M_H$ and all $\lambda\in\fraka_\CC^*$, $\nu\in\fraka_{H,\CC}^*$.
\end{Theorem}

\subsection{Integral kernels}

Recall that $H$-intertwining operators are given by distribution kernels (see Proposition~\ref{prop:DistributionKernel}). Let us see how the integral kernel behaves under the translation principle. Suppose that
\begin{gather*}
A\colon \ \Ind^G_P\big(\xi\otimes e^{\lambda}\otimes \1\big) \to \Ind^H_{P_H}\big(\eta\otimes e^{\nu}\otimes \1\big)
\end{gather*}
is given by a distribution kernel $K^A\in(\calD'(G/P,\calV_\lambda^*)\otimes W_\nu)^{\Delta(P_H)}$ in the sense that
\begin{gather*} Af(h) = \int_{G/P} c\big(f(x)\otimes K^A\big(h^{-1}x)\big)\,{\rm d}x, \qquad f\in C^\infty(G/P,\calV_\lambda), \end{gather*}
where $c$ denotes the contraction map $V_\lambda\otimes V_\lambda^*\otimes W_\nu\to W_\nu$ and the integral is meant in the distribution sense (see Section~\ref{sec:SBOs} for details). Write $i\colon E'\to E$ for the inclusion map and $p\colon E\to E''$ for the quotient map, and let $\calE'=G\times_P E'$ and $\calE''=H\times_{P_H}E''$ denote the corresponding homogeneous vector bundles over $G/P$ and $H/P_H$. Let $u\in C^\infty(G/P,\calE'\otimes\calV_\lambda)$, then $\Phi(A)u\in C^\infty(H/P_H,\calE''\otimes\calW_\nu)$ is given by
\begin{align*}
 \Phi(A)f(h)&= \big(\big(p\circ\tau\big(h^{-1}\big)\big)\otimes\id_{W_\nu}\big)\\
 &\hphantom{=}{}\times \int_{G/P} (\id_E\otimes c)\big(\big((\tau(x)\circ i)\otimes\id_{V_\lambda}\big) (f(x) )\otimes K^A\big(h^{-1}x\big)\big)\,{\rm d}x\\
&= \int_{G/P} c\big(\big(p\circ\tau\big(h^{-1}x\big)\circ i\big)f(x)\otimes K(h^{-1}x)\big)\,{\rm d}x.
\end{align*}
This implies:

\begin{Proposition}\label{prop:TranslationPrincipleDistributionKernel}
The integral kernel $\Phi\big(K^A\big)=K^{\Phi(A)}$ of $\Phi(A)$ is given by
\begin{gather*}
\Phi\big(K^A\big) = \varphi\otimes K^A,
\end{gather*}
where $\varphi\in C^\infty(G/P,(\calE')^\vee)\otimes E''\simeq C^\infty(G,(E')^\vee\otimes E'')^P$ is given by
\begin{gather*} \varphi(g) = p\circ\tau(g)\circ i \in \Hom_\CC(E',E'') \simeq (E')^\vee\otimes E''. \end{gather*}
\end{Proposition}

\begin{Remark}\label{rem:PropertiesMultMap}Since the operator $\Phi$ is given by multiplication with the fixed smooth function~$\varphi$, and the multiplication map
\begin{gather*} \calD'(G/P,\calV_\lambda^*)\otimes W_\nu \to \calD'\big(G/P,\calV_\lambda^*\otimes(\calE')^\vee\big)\otimes(W_\nu\otimes E''), \qquad K\mapsto\varphi\otimes K \end{gather*}
is continuous, the operator $\Phi$ maps holomorphic families of distributions to holomorphic families. More precisely, if $K_z\in(\calD'(G/P,\calV_\lambda^*)\otimes W_\nu)^{\Delta(P_H)}$ depends holomorphically on $z\in\Omega\subseteq\CC^n$, then~$\Phi(K_z)$ depends holomorphically on $z\in\Omega$. More generally, if $K_z$ depends meromorphically on $z\in\Omega$ with poles in the set $\Sigma\subseteq\Omega$, then~$\Phi(K_z)$ depends meromorphically on $z\in\Omega$ and its poles are contained in $\Sigma$. However, it may of course happen that $K_z$ has a pole at $z=z_0$ whereas $\Phi(K_z)$ is regular at $z=z_0$ since the multiplication map $K\mapsto\varphi\otimes K$ can have a non-trivial kernel (see, e.g., Remark~\ref{rem:PolesDisappear}). This also implies that a holomorphic/meromorphic family $K_z$, which does not vanish identically, might be mapped to $\Phi(K_z)=0$ for all $z\in\Omega$ (see, e.g., Remark~\ref{rem:ZerosAppear}). If, however, the family $K_z$ has generically full support, i.e., $\supp K_z=G/P$ for generic $z\in\Omega$, then $\Phi(K_z)=\varphi\otimes K_z$ cannot be identically zero for all $z\in\Omega$. In fact, $\varphi$ is an analytic function which is non-zero due to the irreducibility of $(\tau,E)$, so it has full support $\supp\varphi=G/P$ and hence
\begin{gather*} \supp K_z=G/P \quad \Rightarrow \quad \supp(\varphi\otimes K_z)=G/P. \end{gather*}
\end{Remark}

\begin{Remark}Since $\varphi(g)$ is a matrix coefficient of a finite-dimensional repesentation, it is obviously smooth in $g\in G$. In view of Theorem~\ref{thm:DistributionKernelOnNbar} one can in some cases study the distribution kernels by their restriction to $\overline{\frakn}\simeq\overline{N}\hookrightarrow G/P$. On $\overline{\frakn}$ the function $\varphi(g)$ is actually a polynomial. In fact, the nilpotency of $\overline{\frakn}$ implies that there exists $N\in\NN$ such that $\tau(X_1)\cdots\tau(X_n)=0$ for all $X_1,\ldots,X_n\in\overline{\frakn}$ and $n\geq N$. This shows that $\varphi|_{\overline{\frakn}}$ is a polynomial of degree at most~$N$.
\end{Remark}

\subsection[Reformulation using $\overline{P}$]{Reformulation using $\boldsymbol{\overline{P}}$}

We can also use the opposite parabolic subgroup $\overline{P}$ instead of $P$ in the above procedure. Assume there exists an element $\tilde{w}_0\in K$ such that $\tilde{w}_0 \overline{P} \tilde{w}_0^{-1}=P$ and $\tilde{w}_0 A \tilde{w}_0^{-1}=A$. Then we have a $G$-equivariant isomorphism
\begin{gather*}
\Ind^G_{P}\big(\xi\otimes e^{\lambda}\otimes \1\big)
\simeq
\Ind^G_{\overline{P}}
\big(\tilde{w}_0^{-1}\xi\otimes e^{w_0^{-1}\lambda}\otimes \1\big), \qquad
f\mapsto f\big(\,\cdot\, \tilde{w}_0^{-1}\big).
\end{gather*}
Here, $\tilde{w}_0^{-1}\xi$ denotes the representation of $M$ on $V$ given by $\big(\tilde{w}_0^{-1}\xi\big)(m)=\xi\big(\tilde{w}_0 m \tilde{w}_0^{-1}\big)$.
Suppose that an $H$-intertwining operator
\begin{gather*}
A\colon \ \Ind^G_P\big(\xi\otimes e^{\lambda}\otimes \1\big)
\to \Ind^H_{P_H}\big(\eta\otimes e^{\nu}\otimes \1\big)
\end{gather*}
is given.
Composing with the above ismorphism we have
\begin{gather*}
\Ind^G_{\overline{P}}\big(\tilde{w}_0^{-1}\xi\otimes
e^{w_0^{-1}\lambda}\otimes \1\big) \to
\Ind^H_{P_H}\big(\eta\otimes e^{\nu}\otimes \1\big).
\end{gather*}
Then in a similar way, using $\overline{P}$ instead of $P$, we obtain an $H$-intertwining operator
\begin{gather*}
\Ind^G_{\overline{P}}
\big(\big(\tilde{w}_0^{-1}\xi\otimes \tau'\big) \otimes e^{w_0^{-1}\lambda+\mu'}\otimes \1\big)
\to \Ind^H_{P_H}\big((\eta\otimes \tau'')\otimes e^{\nu+\mu''}\otimes \1\big)
\end{gather*}
for every $\overline{P}$-stable subspace $i\colon E'\hookrightarrow E$ with $E'\simeq\tau'\otimes e^{\mu'}\otimes \1$ and every $P_H$-stable quotient $p\colon E\to E''$. Composing with the map $f\mapsto f(\,\cdot\,\tilde{w}_0)$, we get
\begin{gather*}
\Psi(A)\colon \ \Ind^G_P\big((\xi\otimes \tilde{w}_0\tau')
\otimes e^{\lambda+w_0\mu'}\otimes \1\big)
\to \Ind^H_{P_H}\big((\eta\otimes \tau'') \otimes e^{\nu+\mu''}\otimes \1\big).
\end{gather*}
Moreover, if $A$ is given by the distribution kernel $K^A\in(\calD'(G/P,\calV_\lambda^*)\otimes W_\nu)^{\Delta(P_H)}$, then we similarly see that $\Psi(A)$ has distribution kernel $\Psi\big(K^A\big)=K^{\Psi(A)}$ given by
\begin{gather*}
\Psi\big(K^A\big) = (p\circ \tau(\blank\tilde{w}_0)\circ i) \otimes K^A.
\end{gather*}

Now assume additionally that $P_H\subseteq P$ with $M_HA_H\subseteq MA$ and $N_H\subseteq N$, then $E'$ and $E''$ can be chosen compatibly. More precisely, let $E'\subseteq E$ be the unique irreducible $\overline{P}$-subrepresentation (see Lemma~\ref{lem:UniquePstableSubspace}), then $E'$ is the lowest $\fraka$-weight space and $E'\simeq\tau'\otimes e^{\mu'}\otimes\1$. Write $E_1$ for the direct sum of all other $\fraka$-weight spaces, then $E=E'\oplus E_1$ and the projection $E\twoheadrightarrow E'$ is $P_H$-equivariant. Hence we can take $E''=E'$:

\begin{Corollary}\label{cor:KnappSteinTranslationPrinciple}
Assume that $\tilde{w}_0\overline{P}\tilde{w}_0^{-1}=P$ and that $M_HA_H\subseteq MA$ and $N_H\subseteq N$. Let $(\tau,E)$ be an irreducible finite-dimensional $G$-representation and $E'\subseteq E$ the unique irreducible $\overline{P}$-subrepresentation with $E'=\tau'\otimes e^{\mu'}\otimes\1$. Then for all finite-dimensional representations $\xi$ of $M$ and $\eta$ of $M_H$ and all $\lambda\in\fraka_\CC^*$, $\nu\in\fraka_{H,\CC}^*$ we obtain a linear map
\begin{gather*}
\Psi\colon \ \Hom_H\big(\Ind_P^G\big(\xi\otimes e^\lambda\otimes\1\big),\Ind_{P_H}^H\big(\eta\otimes e^\nu\otimes\1\big)\big)\\
\qquad {} \to \Hom_H\big(\Ind_P^G\big((\xi\otimes\tilde{w}_0\tau')\otimes e^{\lambda+w_0\mu'}\otimes\1\big),\Ind_{P_H}^H\big((\eta\otimes\tau'|_{M_H})\otimes e^{\nu+\mu'|_{\fraka_H}}\otimes\1\big)\big),
\end{gather*}
which maps an intertwining operator $A$ with distribution kernel $K^A$ to the intertwining operator $\Psi(A)$ with distribution kernel $\Psi(K^A)=K^{\Psi(A)}$ given by
\begin{gather*} \Psi\big(K^A\big) = \psi\otimes K^A, \end{gather*}
where $\psi\in C^\infty(G/P,(\calE')^\vee)\otimes E'\simeq C^\infty(G,(E')^\vee\otimes E')^P$ is given by
\begin{gather*} \psi(g) = \tau'\big(m\big(\tilde{w}_0^{-1}g^{-1}\big)\big)^{-1}\cdot a\big(\tilde{w}_0^{-1}g^{-1}\big)^{-\mu'} \in \End_\CC(E') \simeq (E')^\vee\otimes E'. \end{gather*}
\end{Corollary}

\begin{proof}It remains to show the formula for $\psi(g)$. Write
\begin{gather*} \tilde{w}_0^{-1}g^{-1} = \overline{n}\big(\tilde{w}_0^{-1}g^{-1}\big)m\big(\tilde{w}_0^{-1}g^{-1}\big)a\big(\tilde{w}_0^{-1}g^{-1}\big)n\big(\tilde{w}_0^{-1}g^{-1}\big) \in \overline{N}MAN, \end{gather*}
then we have
\begin{gather*} g\tilde{w}_0 = n\big(\tilde{w}_0^{-1}g^{-1}\big)^{-1}m\big(\tilde{w}_0^{-1} g^{-1}\big)^{-1}a\big(\tilde{w}_0^{-1} g^{-1}\big)^{-1}\overline{n}\big(\tilde{w}_0^{-1}g^{-1}\big)^{-1}\in NMA\overline{N}. \end{gather*}
Since $\overline{N}$ acts trivially on $E'$ we have $\tau(\overline{n})\circ i=i$ for $\overline{n}\in\overline{N}$, and similarly $p\circ\tau(n)=p$ for $n\in N_H\subseteq N$. Hence
\begin{gather*}
p\circ\tau(g\tilde{w}_0)\circ i = \tau'\big(m\big(\tilde{w}_0^{-1} g^{-1}\big)^{-1}\big)\cdot a\big(\tilde{w}_0^{-1} g^{-1}\big)^{-\mu'}.\tag*{\qed}
\end{gather*}
\renewcommand{\qed}{}
\end{proof}

\begin{Remark}\label{rem:ClassicalKnappStein}We note that the function $\psi(g)$ resembles the integral kernel of the standard Knapp--Stein intertwining operator
\begin{gather*} A(\lambda)\colon \ \Ind_P^G\big(\tau'\otimes e^\lambda\otimes\1\big)\to\Ind_P^G\big(\tilde{w}_0\tau'\otimes e^{w_0\lambda}\otimes\1\big), \qquad Af(g) = \int_{\overline{N}}f(g\tilde{w}_0\overline{n})\,{\rm d}\overline{n}. \end{gather*}
More precisely, it is shown in \cite[Chapter~VII, Section~7]{Kna86} that
\begin{gather*} Af(g) = \int_{\overline{N}} \tau'\big(m\big(\tilde{w}_0^{-1}\overline{n}\big)\big)a\big(\tilde{w}_0^{-1}\overline{n}\big)^{\lambda-\rho} f(g\overline{n})\,{\rm d}\overline{n}. \end{gather*}
\end{Remark}

\subsection{Application to Knapp--Stein type intertwining operators}\label{sec:VectorValuedKnappSteinKernels}

We now specialize to the setting of Section~\ref{sec:KStypeIntertwiningOperators}. In this case, all assumptions of Corollary~\ref{cor:KnappSteinTranslationPrinciple} are satisfied and we can apply it to the family $A(\alpha,\beta)$ of intertwining operators
\begin{gather*}
A(\alpha,\beta)\colon \ \Ind^G_P\big(\1 \otimes e^{\lambda}\otimes \1\big)
\to \Ind^H_{P_H}\big(\1 \otimes e^{\nu}\otimes \1\big),
\end{gather*}
$\lambda$ and $\nu$ given by \eqref{eq:LambdaNuAlphaBeta}, with distribution kernels
\begin{gather*}
K_{\alpha,\beta}(g) = a\big(\tilde{w}_0^{-1} g^{-1}\big)^{\alpha} a\big(\tilde{w}_0^{-1} g^{-1}\sigma(g)\big)^{\beta}, \qquad g\in G.
\end{gather*}

We call an irreducible finite-dimensional representation $\xi$ of $M$ \textit{extendible} if there exists a~finite-dimensional irreducible representation $(\tau,E)$ of $G$ such that
\begin{gather*} E^{\overline N} = \big\{v\in E\colon \tau(\overline{n})v=v\,\forall\,\overline{n}\in\overline{N}\big\} \simeq \xi \qquad \mbox{as $M$-representations.} \end{gather*}

\begin{Corollary}\label{cor:VectorValuedKnappSteinKernels}
Assume conditions \eqref{ass:G}, \eqref{ass:H} and \eqref{ass:D}. Then for every extendible irreducible finite-dimensional $M$-representation $\xi$ the functions
\begin{gather*}
K_{\xi,\alpha,\beta}(g) = \xi\big(m\big(\tilde{w}_0^{-1} g^{-1}\big)\big)^{-1} a\big(\tilde{w}_0^{-1}g^{-1}\big)^\alpha a\big(\tilde{w}_0^{-1}g^{-1}\sigma(g)\big)^\beta, \qquad g\in G,
\end{gather*}
extend to a meromorphic family of distributions
\begin{gather*} K_{\xi,\alpha,\beta} \in (\calD'(G/P,\calV_\lambda^*)\otimes W_\nu)^{\Delta(P_H)}, \end{gather*}
for $\lambda$ and $\nu$ given by \eqref{eq:LambdaNuAlphaBeta} and $\calV_\lambda=G\times_P\big(\tilde{w}_0\xi\otimes e^{\lambda+\rho}\otimes\1\big)$, $W_\nu=\xi|_{M_H}\otimes e^{\nu+\rho_H}\otimes\1$.
Therefore, they give rise to a meromorphic family of intertwining operators
\begin{gather*} A(\xi,\alpha,\beta)\colon \ \Ind^G_P\big(\tilde{w}_0\xi\otimes e^\lambda\otimes\1\big) \to \Ind^H_{P_H}\big(\xi\otimes e^\nu\otimes\1\big). \end{gather*}
\end{Corollary}

\begin{proof}Let $(\tau,E)$ be an irreducible finite-dimensional representation of $G$ with $E^{\overline N}\simeq\xi$ as $M$-representations, then $E'=E^{\overline N}\subseteq E$ is $\overline{P}$-stable of the form $E'\simeq\xi\otimes e^\mu\otimes\1$. Now the statement follows from Theorem~\ref{thm:KnappSteinKernels} and Corollary~\ref{cor:KnappSteinTranslationPrinciple}.
\end{proof}

\begin{Remark}\label{rem:ExtendibilityForRankOneOrthogonal}For $(G,H)=(\upO(n+1,1),\upO(n,1))$ one can choose a representative $\tilde{w}_0$ of $w_0$ that centralizes $M$. Hence, $\tilde{w}_0\xi=\xi$ for all irreducible finite-dimensional representations $\xi$ of $M$. By Lemma~\ref{lem:ExistenceOfEnoughGreps} either $\xi$ or $\tilde{w}_0\xi=\xi$ is extendible, so that in this case all irreducible finite-dimensional $M$-representations are extendible.
\end{Remark}

\section[Example: $(G,H)=(\upO(n+1,1),\upO(n,1))$]{Example: $\boldsymbol{(G,H)=(\upO(n+1,1),\upO(n,1))}$}\label{sec:ExScalar}

We apply the translation principle to the symmetry breaking operators between spherical principal series of rank one orthogonal groups studied by Kobayashi--Speh~\cite{KS15} and obtain new symmetry breaking operators between non-spherical scalar-valued principal series.

\subsection{Parabolic subgroups and the symmetric pair}\label{sec:ExScalarGroups}

Let $G=\upO(n+1,1)$, $n\geq2$, realized as the subgroup of $\GL(n+2,\RR)$ preserving the indefinite bilinear form
\begin{gather*} \RR^{n+2}\to\RR, \qquad x\mapsto x_1^2+\cdots+x_{n+1}^2-x_{n+2}^2. \end{gather*}
We choose the (minimal) parabolic subgroup $P=P_{\min}=MAN\subseteq G$ such that $\fraka=\RR H_0$ with
\begin{gather*}
H_0 = \left(\begin{matrix}0&&1\\&\0_n&\\1&&0\end{matrix}\right)
\end{gather*}
and $\frakn=\frakg_\alpha$ for $\alpha\in\fraka^*$ with $\alpha(H_0)=1$. Then
\begin{gather*} M = \left\{\left(\begin{matrix}\varepsilon&&\\&m&\\&&\varepsilon\end{matrix}\right)\colon m\in\upO(n),\varepsilon=\pm1\right\} \simeq \upO(n)\times\ZZ/2\ZZ. \end{gather*}
Let $m_0=\diag(-1,1,\ldots,1,-1)$ and identify $m\in\upO(n)$ with $\diag(1,m,1)\in M$, then $M=\upO(n)\cup m_0\upO(n)$. We further identify $\overline{\frakn}=\frakg_{-\alpha}\cong\RR^n$ by
\[ \RR^n\to\overline{\frakn},\,x\mapsto\left(\begin{matrix}0&x^\top&0\\x&\0_n&x\\0&-x^\top&0\end{matrix}\right) \]
and use this identification to parametrize $\overline{N}=\exp(\overline{\frakn})$ by
\begin{gather*} \RR^n\to\overline{\frakn}\xrightarrow{\exp}\overline{N}, \qquad x\mapsto\overline{n}_x=\left(\begin{matrix}1+|x|^2/2&x^\top&|x|^2/2\\x&\1_n&x\\-|x|^2/2&-x^\top&1-|x|^2/2\end{matrix}\right). \end{gather*}
The group $M=\upO(n)\cup m_0\upO(n)$ acts on $x\in\overline{\frakn}\cong\RR^n$ by the adjoint action as follows:
\begin{gather*} \Ad(m)x = mx, \quad m\in\upO(n), \qquad \Ad(m_0)x = -x. \end{gather*}
Further, $P$ is conjugate to its opposite parabolic $\overline{P}=MA\overline{N}$ by
\begin{gather*}
\tilde{w}_0 = \diag(1,\ldots,1,-1).
\end{gather*}

Let us identify $\fraka_\CC^*\cong\CC$ by $\lambda\mapsto\lambda(H_0)$, so that $\rho=\frac{n}{2}$.

\begin{Lemma}\label{lem:AMprojection}
For $x\in\RR^n$:
\begin{gather*} \tilde{w}_0^{-1}\overline{n}_x\in\overline{N}MAN \quad \Leftrightarrow \quad x\in\RR^n\setminus\{0\}, \end{gather*}
and in this case for $\lambda\in\fraka_\CC^*\simeq\CC$ we have
\begin{gather*} a\big(\tilde{w}_0^{-1}\overline{n}_x\big)^\lambda = |x|^{2\lambda}, \qquad m\big(\tilde{w}_0^{-1}\overline{n}_x\big) = \left(\begin{matrix}1&&\\&\1_n-2\frac{xx^\top}{|x|^2}&\\&&1\end{matrix}\right). \end{gather*}
\end{Lemma}

Now let $\sigma$ be the involution of $G$ given by conjugation with the matrix $\diag(1,\ldots,1,-1,1)$. Then $H=G^\sigma\simeq\upO(n,1)$ and $(G,H)$ forms a symmetric pair. It is easy to see that the pair $(G,H)$ satisfies the assumptions in Theorem~\ref{thm:DistributionKernelOnNbar} and we write $P_H=P\cap H=M_HA_HN_H$. Then
\begin{gather*} M_H=M\cap H=\upO(n-1)\cup m_0\upO(n-1) \qquad \mbox{and} \qquad \fraka_H=\fraka=\RR H_0. \end{gather*}
Further, under the identification $\overline{\frakn}\simeq\RR^n$ the involution $\sigma$ acts by
\begin{gather*} \sigma(x)=(x_1,\ldots,x_{n-1},-x_n), \qquad x\in\RR^n, \end{gather*}
and therefore the subalgebra $\overline{\frakn}_H=\overline{\frakn}^\sigma$ is given by the standard embedding of $\RR^{n-1}$ into $\RR^n$ as the first $n-1$ coordinates.

\subsection{Principal series representations and symmetry breaking operators}\label{sec:ExScalarPrincipalSeriesSBOs}

Denote by $\sgn\colon M\to\{\pm1\}$ the character of $M=\upO(n)\cup m_0\upO(n)$ given by $\sgn(m)=1$ for $m\in\upO(n)$ and $\sgn(m_0)=-1$. Abusing notation we also write $\sgn$ for the corresponding character of $M_H$. For $\delta,\varepsilon\in\ZZ/2\ZZ$ and $\lambda,\nu\in\fraka_\CC^*\simeq\CC$ we define the scalar principal series representations (smooth normalized parabolic induction)
\begin{gather*} \pi_{\lambda,\delta} = \Ind_P^G\big(\sgn^\delta\otimes e^\lambda\otimes\1\big) \qquad \mbox{and} \qquad \tau_{\nu,\varepsilon} = \Ind_{P_H}^H\big(\sgn^\varepsilon\otimes e^\nu\otimes\1\big). \end{gather*}

We consider the space $\Hom_H(\pi_{\lambda,\delta}|_H,\tau_{\nu,\varepsilon})$ of symmetry breaking operators between $\pi_{\lambda,\delta}$ and~$\tau_{\nu,\varepsilon}$. By Theorem~\ref{thm:DistributionKernelOnNbar} every such operator is given by a distribution kernel $K\in\calD'(\overline{\frakn})^{M_HA_H,\frakn_H}$ satisfying certain invariance conditions for the action of $M_HA_H$ and $\frakn_H$. Since $m_0\in M_H$ acts by $K(x)\mapsto(-1)^{\delta+\varepsilon}K(-x)$, it is clear that $\calD'(\overline{\frakn})^{M_HA_H,\frakn_H}$ only depends on the parity of $\delta+\varepsilon$. Identifying $\overline{\frakn}\simeq\RR^n$ as above we write $\calD'\big(\RR^n\big)_{\lambda,\nu}^+$, resp.~$\calD'\big(\RR^n\big)_{\lambda,\nu}^+$, for the space $\calD'\big(\RR^n\big)^{M_HA_H,\frakn_H}$ with $\delta+\varepsilon\equiv0(2)$, resp.~$\delta+\varepsilon\equiv1(2)$. Then, taking distribution kernels is a linear isomorphism
\begin{gather*} \Hom_H(\pi_{\lambda,\delta}|_H,\tau_{\nu,\varepsilon}) \stackrel{\sim}{\to} \begin{cases}\calD'\big(\RR^n\big)_{\lambda,\nu}^+&\mbox{for $\delta+\varepsilon\equiv0(2)$,}\\\calD'\big(\RR^n\big)_{\lambda,\nu}^-&\mbox{for $\delta+\varepsilon\equiv1(2)$.}\end{cases} \end{gather*}

\subsection{Construction of symmetry breaking operators}\label{sec:ConstructionSBOs}

In this section we describe all intertwining operators in $\Hom_H(\pi_{\lambda,\delta},\tau_{\nu,\varepsilon})$ for all $\lambda,\nu\in\CC$, $\delta,\varepsilon\in\ZZ/2\ZZ$. We note that the notation is essentially due to Kobayashi--Speh~\cite{KS15}.

\subsubsection{Spherical principal series}

For $\delta=\varepsilon=0$ the representations $\pi_{\lambda,\delta}$ and $\tau_{\nu,\varepsilon}$ are spherical, i.e. possess a vector invariant under a maximal compact subgroup. In this setting, Section~\ref{sec:KStypeIntertwiningOperators} provides a meromorphic family of intertwining operators $\AA_{\lambda,\nu}\in\Hom_H(\pi_{\lambda,0}|_H,\tau_{\nu,0})$ given by the distribution kernels
\begin{gather*} K_{\lambda,\nu}^{\AA,+}(x) = |x_n|^{\lambda+\nu-\frac{1}{2}}\big(|x'|^2+x_n^2\big)^{-\nu-\frac{n-1}{2}}, \qquad x\in\RR^n. \end{gather*}
By analyzing the poles and residues of $K^{\AA,+}_{\lambda,\nu}$ explicitly, Kobayashi--Speh~\cite{KS15} completely determine the space $\calD'\big(\RR^n\big)_{\lambda,\nu}^+$ for all $\lambda,\nu\in\CC$. (Note that our parameters $(\lambda,\nu)$ are normalized such that $\pi_{\lambda,\delta}$ and $\tau_{\nu,\varepsilon}$ are unitary for $\lambda,\nu\in i\RR$. Therefore, in our notation one has to replace $(\lambda,\nu)$ by $(\lambda-\rho,\nu-\rho_H)$ to obtain Kobayashi--Speh's notation, see~\cite{KS15}.) To summarize their results let
\begin{gather*}
\LocalPoles^+ = \big\{(\lambda,\nu)\colon \lambda+\nu=-\tfrac{1}{2}-2k,k\in\NN\big\},\\
\GlobalPoles^+ = \big\{(\lambda,\nu)\colon \lambda-\nu=-\tfrac{1}{2}-2\ell,\ell\in\NN\big\},\\
L_\even = \big\{(\lambda,\nu)=(-\rho-i,-\rho_H-j)\colon i,j\in\NN,i-j\in2\NN\big\} \subseteq \GlobalPoles^+.
\end{gather*}
For $(\lambda,\nu)\in\LocalPoles^+$ with $\lambda+\nu=-\frac{1}{2}-2k$ we further define
\begin{align*}
\widetilde{K}^{\BB,+}_{\lambda,\nu}(x) &= \frac{1}{\Gamma\big(\frac{1}{2}\big(\lambda-\nu+\frac{1}{2}\big)\big)}\big(|x'|^2+x_n^2\big)^{-\nu-\frac{n-1}{2}}\delta^{(2k)}(x_n)\\
&= \sum_{i=0}^k\frac{(-1)^i(2k)!\big(\nu+\frac{n-1}{2}\big)_i}{i!(2k-2i)!\Gamma\big(\frac{1}{2}\big(\lambda-\nu+\frac{1}{2}\big)\big)}|x'|^{1-n-2\nu-2i}\delta^{(2k-2i)}(x_n),
\end{align*}
and for $(\lambda,\nu)\in\GlobalPoles^+$ with $\lambda-\nu=-\frac{1}{2}-2\ell$ we put
\begin{gather*}
\widetilde{K}^{\CC,+}_{\lambda,\nu}(x) = \sum_{j=0}^\ell\frac{2^{2\ell-2j}(\nu-\ell)_{\ell-j}}{j!(2\ell-2j)!}\big(\Delta_{\RR^{n-1}}^j\partial_n^{2\ell-2j}\delta\big)(x).
\end{gather*}
Note that both $\widetilde{K}^{\BB,+}_{\lambda,\nu}$ and $\widetilde{K}^{\CC,+}_{\lambda,\nu}$ depend holomorphically on $\nu\in\CC$ (or $\lambda\in\CC$).

\begin{Theorem}\label{thm:SBOsSpherical}The renormalized distribution
\begin{gather*} \widetilde{K}^{\AA,+}_{\lambda,\nu}(x) = \frac{K^{\AA,+}_{\lambda,\nu}(x)}{\Gamma\big(\frac{1}{2}\big(\lambda+\nu+\frac{1}{2}\big)\big)\Gamma\big(\frac{1}{2}\big(\lambda-\nu+\frac{1}{2}\big)\big)} \end{gather*}
depends holomorphically on $(\lambda,\nu)\in\CC^2$ and vanishes only for $(\lambda,\nu)\in L_\even$. More precisely,
\begin{enumerate}\itemsep=0pt
\item[$1.$] For $(\lambda,\nu)\in\CC^2-(\LocalPoles^+\cup\GlobalPoles^+)$ we have $\supp\widetilde{K}^{\AA,+}_{\lambda,\nu} = \RR^n$.
\item[$2.$] For $(\lambda,\nu)\in\LocalPoles^+-\GlobalPoles^+$ with $\lambda+\nu=-\frac{1}{2}-2k$, $k\in\NN$, we have
\begin{gather*} \widetilde{K}^{\AA,+}_{\lambda,\nu}(x) = \frac{(-1)^kk!}{(2k)!}\widetilde{K}^{\BB,+}_{\lambda,\nu}(x). \end{gather*}
In particular, $\supp\widetilde{K}^{\AA,+}_{\lambda,\nu}=\RR^{n-1}$.
\item[$3.$] For $(\lambda,\nu)\in\GlobalPoles^+-L_\even$ with $\lambda-\nu=-\frac{1}{2}-2\ell$, $\ell\in\NN$, we have
\begin{gather*} \widetilde{K}^{\AA,+}_{\lambda,\nu}(x) = \frac{(-1)^\ell\ell!\pi^{\frac{n-1}{2}}}{2^{2\ell}\Gamma\big(\nu+\frac{n-1}{2}\big)}\widetilde{K}^{\CC,+}_{\lambda,\nu}(x). \end{gather*}
In particular, $\supp\widetilde{K}^{\AA,+}_{\lambda,\nu}=\{0\}$.
\item[$4.$] For $(-\rho-i,-\rho_H-j)\in L_\even$, $2\ell=i-j$, restricting the function $(\lambda,\nu)\mapsto\widetilde{K}^{\AA,+}_{\lambda,\nu}$ to two different complex hyperplanes in $\CC^2$ through $(-\rho-i,-\rho_H-j)$ and renormalizing gives two holomorphic families of distributions:
\begin{gather*}
 \widetilde{K}^{\CC,+}_{\lambda,\nu}(x) = \frac{(-1)^\ell2^{2\ell}}{\ell!\pi^{\frac{n-1}{2}}}\Gamma(\nu+\rho_H)\widetilde{K}^{\AA,+}_{\lambda,\nu}(x) ,\qquad \lambda-\nu=-\tfrac{1}{2}-2\ell ,\\
\doublewidetilde{K}^{\AA,+}_{\lambda,\nu}(x) = \Gamma\big(\tfrac{1}{2}\big(\lambda-\nu+\tfrac{1}{2}\big)\big)\widetilde{K}^{\AA,+}_{\lambda,\nu}(x),\qquad \nu=-\rho_H-j.
\end{gather*}
Their special values at $(-\rho-i,-\rho_H-j)$ satisfy
\begin{gather*}
\supp\widetilde{K}^{\CC,+}_{-\rho-i,-\rho_H-j} = \{0\},\\
\supp\doublewidetilde{K}^{\AA,+}_{-\rho-i,-\rho_H-j} = \begin{cases}\RR^n,&\mbox{for $n$ even,}\\\RR^{n-1},&\mbox{for $n$ odd,}\end{cases}
\end{gather*}
more precisely, for $n$ odd and $2k=i+j+n-1$
\begin{gather*} \doublewidetilde{K}^{\AA,+}_{-\rho-i,-\rho_H-j} = \frac{(-1)^kk!}{(2k)!}\big(|x'|^2+x_n^2\big)^j\delta^{(2k)}(x_n). \end{gather*}
\end{enumerate}
\end{Theorem}

\begin{proof}This is a summary of the results of \cite{KS15}.
\end{proof}

These results can be used to describe the space $\calD'\big(\RR^n\big)_{\lambda,\nu}^+$:

\begin{Theorem}\label{thm:SBOsSphericalClassification}We have
\begin{gather*} \calD'\big(\RR^n\big)_{\lambda,\nu}^+ = \begin{cases}\CC\widetilde{K}^{\AA,+}_{\lambda,\nu},&\mbox{for $(\lambda,\nu)\in\CC^2\setminus L_\even$,}\\\CC\doublewidetilde{K}^{\AA,+}_{\lambda,\nu}\oplus\CC \widetilde{K}^{\CC,+}_{\lambda,\nu},&\mbox{for $(\lambda,\nu)\in L_\even$.}\end{cases} \end{gather*}
\end{Theorem}

\begin{proof}
See \cite[Theorem 1.9]{KS15} and also \cite[Theorem 4.9]{MO14a}.
\end{proof}

\subsubsection{Scalar principal series}

We use the translation principle to describe $\calD'\big(\RR^n\big)_{\lambda,\nu}^-$ in terms of $\calD'\big(\RR^n\big)_{\lambda,\nu}^+$.

Let $(\tau,E)$ be the defining representation of $G=\upO(n+1,1)$ on $E=\CC^{n+2}$. Then
\begin{gather*} i\colon \ E'=\CC\hookrightarrow E, \qquad z\mapsto(z,0,\ldots,0,z)^\top \end{gather*}
is the maximal subspace on which $N$ acts trivially, and $E'|_P=\sgn\otimes e^\alpha\otimes\1$. For the $P_H$-quotient space we choose
\begin{gather*} p\colon \ E\twoheadrightarrow \CC, \qquad (z_1,\ldots,z_{n+2})\mapsto z_{n+1}, \end{gather*}
then $E''|_{P_H}=\1\otimes\1\otimes\1$. Now let us compute the $(E')^*\otimes E''$-valued function $p\circ\tau(g)\circ i$ on $\overline{N}$. For $g=\overline{n}_x$, $x\in\RR^n$, we have
\begin{gather*} p\circ\tau(g)\circ i(1)^\top = p\big(1+|x|^2,2x,1-|x|^2\big)^\top = 2x_n. \end{gather*}
Hence, by Theorem~\ref{thm:TranslationPrinciple}, we get a linear map
\begin{gather*} x_n\colon \ \calD'\big(\RR^n\big)_{\lambda,\nu}^\pm \to \calD'\big(\RR^n\big)_{\lambda+1,\nu}^\mp, \qquad K(x)\mapsto x_nK(x). \end{gather*}

Applying this map to the meromorphic family of distributions $K^{\AA,+}_{\lambda,\nu}(x)$ we obtain another meromorphic family
\begin{gather*} K^{\AA,-}_{\lambda,\nu}(x) = x_n\cdot K^{\AA,+}_{\lambda-1,\nu} = \sgn(x_n)|x_n|^{\lambda+\nu-\frac{1}{2}}\big(|x'|^2+x_n^2\big)^{-\nu-\frac{n-1}{2}}, \qquad x\in\RR^n. \end{gather*}

To normalize $K^{\AA,-}_{\lambda,\nu}(x)$ so that it becomes holomorphic in $(\lambda,\nu)\in\CC^2$ we first consider the kernel of the map
\begin{gather}
x_n\colon \ \calD'\big(\RR^n\big)_{\lambda,\nu}^+\to\calD'\big(\RR^n\big)_{\lambda+1,\nu}^-.\label{eq:XnTranslationMap}
\end{gather}

\begin{Lemma}\label{lem:SBOsSameSignSingular}We have
\begin{gather*} \big\{K\in\calD'\big(\RR^n\big)_{\lambda,\nu}^+\colon x_nK=0\big\} = \begin{cases}\CC\widetilde{K}^{\BB,+}_{\lambda,\nu},&\mbox{for $\lambda+\nu=-\frac{1}{2}$,}\\ \CC\widetilde{K}^{\CC,+}_{\lambda,\nu},&\mbox{for $\lambda-\nu=-\frac{1}{2}$,}\\ \{0\},&\mbox{else.}\end{cases} \end{gather*}
\end{Lemma}

Note that for $(\lambda,\nu)=\big({-}\frac{1}{2},0\big)$ we have $\widetilde{K}^{\BB,+}_{\lambda,\nu}(x)=\widetilde{K}^{\CC,+}_{\lambda,\nu}(x)=\delta(x)$.

\begin{proof}If $K\in\calD'\big(\RR^n\big)_{\lambda,\nu}^+$ such that $x_nK=0$ then clearly $\supp(K)\subseteq\RR^{n-1}$. Hence,
\begin{gather*} K = \sum_{m=0}^M u_m(x')\delta^{(m)}(x_n) \end{gather*}
for some distributions $u_m\in\calD'\big(\RR^{n-1}\big)$. Since $x_n\delta^{(m)}(x_n)=-m\delta^{(m-1)}(x_n)$ we must have $u_m=0$ for $m>0$, so that $K(x)=u_0(x')\delta(x_n)$. By the classification in Theorems~\ref{thm:SBOsSpherical} and~\ref{thm:SBOsSphericalClassification} the only possibilities for such $K$ in the space $\calD'\big(\RR^n\big)_{\lambda,\nu}^+$ are
\begin{gather*} K(x) = \widetilde{K}^{\BB,+}_{\lambda,\nu}(x) = \frac{|x'|^{1-n-2\nu}\delta(x_n)}{\Gamma(-\nu)} \end{gather*}
with $\lambda+\nu=-\frac{1}{2}$ (i.e., $k=0$) and
\begin{gather*} K(x) = \widetilde{K}^{\CC,+}_{\lambda,\nu}(x) = \delta(x) \end{gather*}
with $\lambda-\nu=-\frac{1}{2}$ (i.e., $\ell=0$).
\end{proof}

\begin{Remark}\label{rem:ZerosAppear}Lemma~\ref{lem:SBOsSameSignSingular} implies that the families
\begin{gather*} K^\BB_\nu=\widetilde{K}_{-\frac{1}{2}-\nu,\nu}^{\BB,+} \qquad \mbox{and} \qquad K^\CC_\nu=\widetilde{K}_{-\frac{1}{2}+\nu,\nu}^{\CC,+}, \end{gather*}
which depend holomorphically on $\nu\in\CC$ and are not identically zero, are mapped to identically zero families by \eqref{eq:XnTranslationMap}, i.e., $x_nK^\BB_\nu=x_nK^\CC_\nu=0$ for all $\nu\in\CC$. However, the family $\widetilde{K}^{\AA,+}_{\lambda,\nu}$ which depends holomorphically on $(\lambda,\nu)\in\CC^2$ has generically full support, i.e., $\supp\widetilde{K}^{\AA,+}_{\lambda,\nu}=\RR^n$ for generic $(\lambda,\nu)\in\CC^2$, so it follows from Remark~\ref{rem:PropertiesMultMap} that also $\supp x_n\widetilde{K}^{\AA,+}_{\lambda,\nu}=\RR^n$ for generic $(\lambda,\nu)\in\CC^2$. In particular, the holomorphic family $x_n\widetilde{K}^{\AA,+}_{\lambda,\nu}$ is not identically zero.
\end{Remark}

For the full classification of $\calD'\big(\RR^n\big)_{\lambda,\nu}^-$ we also need to understand the kernel of the map
\begin{gather*} x_n\colon \ \calD'\big(\RR^n\big)_{\lambda,\nu}^-\to\calD'\big(\RR^n\big)_{\lambda+1,\nu}^+. \end{gather*}

\begin{Lemma}\label{lem:SBOsOppositeSignSingular}We have
\begin{gather*} \big\{K\in\calD'\big(\RR^n\big)_{\lambda,\nu}^-\colon x_nK=0\big\} = \{0\}. \end{gather*}
\end{Lemma}

\begin{proof}As in the previous proof any $K$ with $x_nK=0$ must be of the form $K(x)=u_0(x')\delta(x_n)$. Now, if $K\in\calD'\big(\RR^n\big)_{\lambda,\nu}^-$ then by the invariance under $m_0\in M_H$ we have $u(-x)=-u(x)$ and since $\delta(-x_n)=\delta(x_n)$ this implies
\begin{gather*} u_0(-x')=-u_0(x'). \end{gather*}
On the other hand, invariance of $K$ under $\upO(n-1)\subseteq M_H$ implies $u_0(mx')=u_0(x')$ for all $m\in\upO(n-1)$, in particular for $m=-\1$, hence
\begin{gather*} u_0(-x')=u_0(x'). \end{gather*}
This shows that $u_0=0$ and the proof is complete.
\end{proof}

To state the analogue of Theorem~\ref{thm:SBOsSpherical} for $K^{\AA,-}_{\lambda,\nu}(x)$ let
\begin{gather*}
\LocalPoles^- = \big\{(\lambda,\nu)\colon \lambda+\nu=-\tfrac{3}{2}-2k,k\in\NN\big\},\\
\GlobalPoles^- = \big\{(\lambda,\nu)\colon \lambda-\nu=-\tfrac{3}{2}-2\ell,\ell\in\NN\big\},\\
L_\odd = \big\{(\lambda,\nu)=(-\rho-i-1,-\rho_H-j)\colon i,j\in\NN,i-j\in2\NN\big\} \subseteq \GlobalPoles^-.
\end{gather*}
For $(\lambda,\nu)\in\LocalPoles^-$ with $\lambda+\nu=-\frac{3}{2}-2k$ we further define
\begin{align*}
\widetilde{K}^{\BB,-}_{\lambda,\nu}(x) &= \frac{1}{\Gamma\big(\frac{1}{2}\big(\lambda-\nu+\frac{3}{2}\big)\big)}\big(|x'|^2+x_n^2\big)^{-\nu-\frac{n-1}{2}}\delta^{(2k+1)}(x_n)\\
&= \sum_{i=0}^k\frac{(-1)^i(2k+1)!\big(\nu+\frac{n-1}{2}\big)_i}{i!(2k-2i+1)!\Gamma\big(\frac{1}{2}\big(\lambda-\nu+\frac{3}{2}\big)\big)}|x'|^{1-n-2\nu-2i}\delta^{(2k-2i+1)}(x_n),
\end{align*}
and for $(\lambda,\nu)\in\GlobalPoles^-$ with $\lambda-\nu=-\frac{3}{2}-2\ell$ we put
\begin{gather*}
\widetilde{K}^{\CC,-}_{\lambda,\nu}(x) = \sum_{j=0}^\ell\frac{2^{2\ell-2j}(\nu-\ell)_{\ell-j}}{j!(2\ell-2j+1)!}\big(\Delta_{\RR^{n-1}}^j\partial_n^{2\ell-2j+1}\delta\big)(x).
\end{gather*}

\begin{Theorem}\label{thm:SBOsOppositeSign}
The renormalized distribution
\begin{gather*} \widetilde{K}^{\AA,-}_{\lambda,\nu}(x) = \frac{K^{\AA,-}_{\lambda,\nu}(x)}{\Gamma\big(\frac{1}{2}\big(\lambda+\nu+\frac{3}{2}\big)\big)\Gamma\big(\frac{1}{2}\big(\lambda-\nu+\frac{3}{2}\big)\big)} \end{gather*}
depends holomorphically on $(\lambda,\nu)\in\CC^2$ and vanishes only for $(\lambda,\nu)\in L_\odd$. More precisely,
\begin{enumerate}\itemsep=0pt
\item[$1.$] For $(\lambda,\nu)\in\CC^2\setminus(\LocalPoles^-\cup\GlobalPoles^-)$ we have $\supp\widetilde{K}^{\AA,-}_{\lambda,\nu}=\RR^n$.
\item[$2.$] For $(\lambda,\nu)\in\LocalPoles^-\setminus\GlobalPoles^-$ with $\lambda+\nu=-\frac{3}{2}-2k$, $k\in\NN$, we have
\begin{gather*} \widetilde{K}^{\AA,-}_{\lambda,\nu}(x) = \frac{(-1)^{k+1}k!}{(2k+1)!}\widetilde{K}^{\BB,-}_{\lambda,\nu}(x). \end{gather*}
In particular, $\supp\widetilde{K}^{\AA,-}_{\lambda,\nu}=\RR^{n-1}$.
\item[$3.$] For $(\lambda,\nu)\in\GlobalPoles^-\setminus L_\odd$ with $\lambda-\nu=-\frac{3}{2}-2\ell$, $\ell\in\NN$, we have
\begin{gather*} \widetilde{K}^{\AA,-}_{\lambda,\nu}(x) = \frac{(-1)^{\ell+1}\ell!\pi^{\frac{n-1}{2}}}{2^{2\ell}\Gamma\big(\nu+\frac{n-1}{2}\big)}\widetilde{K}^{\CC,-}_{\lambda,\nu}(x). \end{gather*}
In particular, $\supp\widetilde{K}^{\AA,-}_{\lambda,\nu}=\{0\}$.
\item[$4.$]For $(-\rho-i,-\rho_H-j)\in L_\odd$, $2\ell=i-j-1$, restricting the function $(\lambda,\nu)\mapsto\widetilde{K}^{\AA,-}_{\lambda,\nu}$ to two different complex hyperplanes in $\CC^2$ through $(-\rho-i,-\rho_H-j)$ and renormalizing gives two holomorphic families of distributions:
\begin{gather*}
\widetilde{K}^{\CC,-}_{\lambda,\nu}(x) = \frac{(-1)^\ell 2^{2\ell}}{\ell!\pi^{\frac{n-1}{2}}}\Gamma(\nu+\rho_H)\widetilde{K}^{\AA,-}_{\lambda,\nu}(x), \qquad \lambda-\nu=-\tfrac{3}{2}-2\ell,\\
\doublewidetilde{K}^{\AA,-}_{\lambda,\nu}(x) = \Gamma\big(\tfrac{1}{2}\big(\lambda-\nu+\tfrac{3}{2}\big)\big)\widetilde{K}^{\AA,-}_{\lambda,\nu}(x), \qquad \nu=-\rho_H-j.
\end{gather*}
Their special values at $(-\rho-i,-\rho_H-j)$ satisfy
\begin{gather*}
\supp\widetilde{K}^{\CC,-}_{-\rho-i,-\rho_H-j} = \{0\},\\
\supp\doublewidetilde{K}^{\AA,-}_{-\rho-i,-\rho_H-j} = \begin{cases}\RR^n,&\mbox{for $n$ even,}\\\RR^{n-1},&\mbox{for $n$ odd,}\end{cases}
\end{gather*}
more precisely, for $n$ odd and $2k=n+i+j-2$
\begin{gather*}
\doublewidetilde{K}^{\AA,-}_{\lambda,\nu}(x) = \frac{(-1)^{k+1}k!}{(2k+1)!}\big(|x'|^2+x_n^2\big)^j\delta^{(2k+1)}(x_n).
\end{gather*}
\end{enumerate}
\end{Theorem}

\begin{proof}We first apply the map
\begin{gather*} x_n\colon \ \calD'\big(\RR^n\big)_{\lambda-1,\nu}^+ \to \calD'\big(\RR^n\big)_{\lambda,\nu}^- \end{gather*}
to the holomorphic family $\widetilde{K}^{\AA,+}_{\lambda-1,\nu}$ and obtain a holomorphic family of distributions
\begin{gather*} x_n\widetilde{K}^{\AA,+}_{\lambda-1,\nu}(x) = \frac{K^{\AA,-}_{\lambda,\nu}(x)}{\Gamma\big(\frac{1}{2}\big(\lambda+\nu-\frac{1}{2}\big)\big)\Gamma\big(\frac{1}{2}\big(\lambda-\nu-\frac{1}{2}\big)\big)}. \end{gather*}
By Theorem~\ref{thm:SBOsSpherical} and Lemma~\ref{lem:SBOsSameSignSingular} we have $x_n\widetilde{K}^{\AA,+}_{\lambda-1,\nu}=0$ if $\lambda+\nu=\frac{1}{2}$ or $\lambda-\nu=\frac{1}{2}$. We can therefore renormalize the kernels
\begin{gather*} \widetilde{K}^{\AA,-}_{\lambda,\nu}(x) = \frac{x_n\widetilde{K}^{\AA,+}_{\lambda-1,\nu}(x)}{\frac{1}{2}\big(\lambda+\nu-\frac{1}{2}\big)\cdot\frac{1}{2}\big(\lambda-\nu-\frac{1}{2}\big)} = \frac{K^{\AA,-}_{\lambda,\nu}(x)}{\Gamma\big(\frac{1}{2}\big(\lambda+\nu+\frac{3}{2}\big)\big)\Gamma\big(\frac{1}{2}\big(\lambda-\nu+\frac{3}{2}\big)\big)}. \end{gather*}
This shows that $\widetilde{K}^{\AA,-}_{\lambda,\nu}(x)$ depends holomorphically on $(\lambda,\nu)\in\CC^2$. Next, we apply the map
\begin{gather}\label{eq:SBOsOppositeSignInjection}
x_n\colon \ \calD'\big(\RR^n\big)_{\lambda,\nu}^- \to \calD'\big(\RR^n\big)_{\lambda+1,\nu}^+
\end{gather}
to $\widetilde{K}^{\AA,-}_{\lambda,\nu}(x)$ and find
\begin{gather*} x_n\widetilde{K}^{\AA,-}_{\lambda,\nu} = \widetilde{K}^{\AA,+}_{\lambda+1,\nu}. \end{gather*}
By Lemma~\ref{lem:SBOsOppositeSignSingular} the map \eqref{eq:SBOsOppositeSignInjection} is injective, hence $\widetilde{K}^{\AA,-}_{\lambda,\nu}=0$ if and only if $\widetilde{K}^{\AA,+}_{\lambda+1,\nu}=0$, which is only the case for $(\lambda+1,\nu)\in L_\even$, i.e., $(\lambda,\nu)\in L_\odd$.
\begin{enumerate}\itemsep=0pt
\item Let $(\lambda,\nu)\in\CC^2\setminus(\LocalPoles^-\cup\GlobalPoles^-)$. Then $(\lambda+1,\nu)\in\CC^2\setminus(\LocalPoles^+\cup\GlobalPoles^+)$ and hence $\supp\big(x_n\widetilde{K}^{\AA,-}_{\lambda,\nu}\big)=\supp\big(\widetilde{K}^{\AA,+}_{\lambda+1,\nu}\big)=\RR^n$. This implies $\supp\big(\widetilde{K}^{\AA,-}_{\lambda,\nu}\big)=\RR^n$.
\item Let $(\lambda,\nu)\in\LocalPoles^-\setminus\GlobalPoles^-$ with $\lambda+\nu=-\frac{3}{2}-2k$, then it is easy to see, using $x_n\delta^{(m)}(x_n)=-m\delta^{(m-1)}(x_n)$, that
\begin{gather*} x_n\widetilde{K}^{\BB,+}_{\lambda-1,\nu} = 2(k+1)(\nu+k+1)\widetilde{K}^{\BB,-}_{\lambda,\nu}. \end{gather*}
Hence,
\begin{align*}
\widetilde{K}^{\AA,-}_{\lambda,\nu} &= \frac{1}{\frac{1}{2}\big(\lambda+\nu-\frac{1}{2}\big)\cdot\frac{1}{2}\big(\lambda-\nu-\frac{1}{2}\big)}x_n\widetilde{K}^{\AA,+}_{\lambda-1,\nu}\\
&= \frac{1}{(k+1)(\nu+k+1)}\frac{(-1)^{k+1}(k+1)!}{(2k+2)!}x_n\widetilde{K}^{\BB,+}_{\lambda-1,\nu} = \frac{(-1)^{k+1}k!}{(2k+1)!}\widetilde{K}^{\BB,-}_{\lambda,\nu}.
\end{align*}
\item For $(\lambda,\nu)\in\GlobalPoles^-\setminus L_\odd$ the same method as in (2) applies to show (3). Here we use that $x_n\widetilde{K}^{\CC,+}_{\lambda-1,\nu}=-4(\nu-\ell-1)\widetilde{K}^{\CC,-}_{\lambda,\nu}$ for $\lambda-\nu=-\frac{3}{2}-2\ell$.
\item Now let $(-\rho-i,-\rho_H-j)\in L_\odd$. The first statement about $\widetilde{K}^{\CC,-}_{\lambda,\nu}(x)$ follows immediately from (3). For the second statement note that $(-\rho-(i+1),-\rho_H-j)\in L_\even$ and we can use Theorem~\ref{thm:SBOsSpherical}(4). Restricting to $\nu=-\rho_H-j$ we have
\begin{align*}
\doublewidetilde{K}^{\AA,-}_{\lambda,\nu}(x) &= \Gamma\big(\tfrac{1}{2}\big(\lambda-\nu+\tfrac{3}{2}\big)\big)\widetilde{K}^{\AA,-}_{\lambda,\nu}(x)\\
&= \frac{\Gamma\big(\tfrac{1}{2}\big(\lambda-\nu+\tfrac{3}{2}\big)\big)x_n\widetilde{K}^{\AA,+}_{\lambda-1,\nu}(x)}{\frac{1}{2}\big(\lambda+\nu-\frac{1}{2}\big)\cdot\frac{1}{2} \big(\lambda-\nu-\frac{1}{2}\big)} = \frac{x_n\doublewidetilde{K}^{\AA,+}_{\lambda-1,\nu}(x)}{\frac{1}{2}\big(\lambda+\nu-\frac{1}{2}\big)}.
\end{align*}
The denominator is equal to $\tfrac{1}{2}\big(\lambda+\nu-\tfrac{1}{2}\big)=-\frac{1}{2}(n+i+j)=-(k+1)\neq0$, whence $\doublewidetilde{K}^{\AA,-}_{\lambda,\nu}$ depends holomorphically on $\lambda\in\CC$. Moreover, for even $n$ we have $\supp\doublewidetilde{K}^{\AA,+}_{\lambda-1,\nu}=\RR^n$ so that $\supp\big(x_n\doublewidetilde{K}^{\AA,+}_{\lambda-1,\nu}\big)=\RR^n$ and hence $\supp\doublewidetilde{K}^{\AA,-}_{\lambda,\nu}=\RR^n$. For odd $n$ it follows from $x_n\delta^{(m)}(x_n)=-m\delta^{(m-1)}(x_n)$ that
\begin{align*}
\doublewidetilde{K}^{\AA,-}_{\lambda,\nu}(x) &= -\frac{1}{k+1}\cdot\frac{(-1)^{k+1}(k+1)!}{(2k+2)!}x_n\big(|x'|^2+x_n^2\big)^j\delta^{(2k+2)}(x_n)\\
&= \frac{(-1)^{k+1}k!}{(2k+1)!}\big(|x'|^2+x_n^2\big)^j\delta^{(2k+1)}(x_n).\tag*{\qed}
\end{align*}
\end{enumerate}\renewcommand{\qed}{}
\end{proof}

\begin{Remark}The differential intertwining operators belonging to the distribution kernels $\widetilde{K}^{\CC,+}_{\lambda,\nu}$ are the even order conformally covariant differential operators found by Juhl~\cite{Juh09}, sometimes also referred to as \textit{Juhl operators}. Kobayashi--Speh showed that they are obtained as residue families of the non-local intertwining operators with kernels $\widetilde{K}^{\AA,+}_{\lambda,\nu}$. Theorem~\ref{thm:SBOsOppositeSign} proves that also the odd order Juhl operators with integral kernels $\widetilde{K}^{\CC,-}_{\lambda,\nu}$ can be obtained in this way. Further, the translation principle explains the following relation between even and odd order Juhl operators:
\begin{gather*} x_n\widetilde{K}^{\CC,+}_{\lambda,\nu} = -2\big(\lambda+\nu+\tfrac{1}{2}\big)\widetilde{K}^{\CC,-}_{\lambda+1,\nu} \qquad \mbox{and} \qquad x_n\widetilde{K}^{\CC,-}_{\lambda,\nu} = -\widetilde{K}^{\CC,+}_{\lambda+1,\nu}. \end{gather*}
\end{Remark}

\begin{Remark}\label{rem:PolesDisappear}Theorem~\ref{thm:SBOsSpherical} implies that the unnormalized meromorphic family $K_{\lambda,\nu}^{\AA,+}$ has poles for $(\lambda,\nu)\in\LocalPoles^+\cup\GlobalPoles^+$. Multiplication with $x_n$ gives the meromorphic family $x_nK_{\lambda,\nu}^{\AA,+}=K_{\lambda+1,\nu}^{\AA,-}$ which, by Theorem~\ref{thm:SBOsOppositeSign}, has poles for $(\lambda+1,\nu)\in\LocalPoles^-\cup\GlobalPoles^-$. This shows that for $\lambda+\nu=-\frac{1}{2}$ and $\lambda+\nu=-\frac{1}{2}$ the family $K_{\lambda,\nu}^{\AA,+}$ has a pole, while $x_nK_{\lambda,\nu}^{\AA,+}$ does not.
\end{Remark}

Using the translation principle we can finally determine the space $\calD'\big(\RR^n\big)_{\lambda,\nu}^-$ completely:

\begin{Theorem}\label{thm:SBOsOppositeSignClassification}We have
\begin{gather*}
\calD'\big(\RR^n\big)_{\lambda,\nu}^- = \begin{cases}\CC\widetilde{K}^{\AA,-}_{\lambda,\nu},&\mbox{for $(\lambda,\nu)\in\CC^2\setminus L_\odd$,}\\\CC\doublewidetilde{K}^{\AA,-}_{\lambda,\nu}\oplus\CC\widetilde{K}^{\CC,-}_{\lambda,\nu},&\mbox{for $(\lambda,\nu)\in L_\odd$.}\end{cases}%\label{eq:SBOsOppositeSignClassification}
\end{gather*}
\end{Theorem}

\begin{proof}By Lemma~\ref{lem:SBOsOppositeSignSingular} the map
\begin{gather*} x_n\colon \ \calD'\big(\RR^n\big)_{\lambda,\nu}^-\to\calD'\big(\RR^n\big)_{\lambda+1,\nu}^+ \end{gather*}
is injective and hence
\begin{gather}\label{eq:SBOsOppositeSignInequality}
\dim\calD'\big(\RR^n\big)_{\lambda,\nu}^- \leq \dim\calD'\big(\RR^n\big)_{\lambda+1,\nu}^+.
\end{gather}
For $(\lambda,\nu)\in\CC^2\setminus L_\odd$ we have $0\neq\widetilde{K}^{\AA,-}_{\lambda,\nu}\in\calD'\big(\RR^n\big)_{\lambda,\nu}^-$, and hence the left hand side of~\eqref{eq:SBOsOppositeSignInequality} is $\geq1$. For $(\lambda,\nu)\in L_\odd$ we have $\doublewidetilde{K}^{\AA,-}_{\lambda,\nu},\widetilde{K}^{\CC,-}_{\lambda,\nu}\in\calD'\big(\RR^n\big)_{\lambda,\nu}^-$ and these distributions are linearly independent since
\begin{gather*} \supp\widetilde{K}^{\CC,-}_{\lambda,\nu} = \{0\} \subsetneq \supp\doublewidetilde{K}^{\AA,-}_{\lambda,\nu}. \end{gather*}
This shows that the left hand side of~\eqref{eq:SBOsOppositeSignInequality} is $\geq2$. Now note that $(\lambda,\nu)\in L_\odd$ if and only if $(\lambda+1,\nu)\in L_\even$, and therefore~\eqref{eq:SBOsOppositeSignInequality} is actually an equality for all $(\lambda,\nu)\in\CC^2$, thanks to Theorem~\ref{thm:SBOsSphericalClassification}. This finishes the proof.
\end{proof}

\section[Example: $(\widetilde{G},\widetilde{H})=(\Pin(n+1,1),\Pin(n,1))$]{Example: $\boldsymbol{(\widetilde{G},\widetilde{H})=(\Pin(n+1,1),\Pin(n,1))}$}\label{sec:ExSpinors}

We explicitly construct intertwining operators between principal series representations of $\widetilde{G}=\Pin(n+1,1)$ and $\widetilde{H}=\Pin(n,1)$ induced from fundamental spin representations of $\widetilde{M}\simeq\Pin(n)\times\ZZ/2\ZZ$ and $\widetilde{M}_H\simeq\Pin(n-1)\times\ZZ/2\ZZ$ using the translation principle.

\subsection{Parabolic subgroups and the symmetric pair}

We consider the two-fold covering
\begin{gather*} q\colon \ \widetilde{G}=\Pin(n+1,1)\to\upO(n+1,1)=G, \end{gather*}
realized inside the Clifford algebra $\Cl(n+1,1)$ of $\RR^{n+1,1}$. For details on the definition and properties of the pin groups and Clifford algebras we refer the reader to Appendix~\ref{app:CliffordAlgebrasPinGroups}. In what follows we will write $\widetilde{S}=q^{-1}(S)$ for any subgroup $S\subseteq G$.

We choose the same parabolic subgroup $P=MAN\subseteq G$ as in Section~\ref{sec:ExScalarGroups}, then $\widetilde{M}=\Pin(n)\cdot\{1,m_0=e_1e_{n+2}\}$, where the embedding of $\Pin(n)$ into $\Pin(n+1,1)$ is the restriction of the embedding of Clifford algebras $\Cl(n)\hookrightarrow\Cl(n+1,1)$ induced by
\begin{gather*} \RR^n\hookrightarrow\RR^{n+1,1}, \qquad e_i\mapsto e_{i+1}, \qquad 1\leq i\leq n. \end{gather*}
We \looseness=-1 note that $m_0=e_1e_{n+2}$ commutes with $\Pin(n)$ and hence $\widetilde{M}\simeq\Pin(n)\times\ZZ/2\ZZ$. Both~$A$ and~$N$ split in the cover, so that $\widetilde{A}\simeq A\times\ZZ/2\ZZ$ and $\widetilde{N}\simeq N\times\ZZ/2\ZZ$, and we identify $A$ and $N$ with $\widetilde{A}_0,\widetilde{N}_0\subseteq\widetilde{G}$. Then $\widetilde{P}=\widetilde{M}AN$ is the Langlands decomposition of the parabolic subgroup~$\widetilde{P}$ of~$\widetilde{G}$.

For the Weyl group element $w_0$ we choose the representative
\begin{gather}
\tilde{w}_0 = \pm e_{n+2}\label{eq:SgnOfw_0}
\end{gather}
with the sign yet to be determined. Since $q(\tilde{w}_0)=\diag(1,\ldots,1,-1)$, the element $\tilde{w}_0$ corresponds to the representative chosen in Section~\ref{sec:ExScalarGroups} for the group $G$. Hence, by Lemma~\ref{lem:AMprojection} we find
\begin{gather*} q\big(m\big(\tilde{w}_0^{-1}\overline{n}_x^{-1}\big)^{-1}\big) = \left(\begin{matrix}1&&\\&\1_n-2\frac{xx^\top}{|x|^2}&\\&&1\end{matrix}\right), \qquad x\in\RR^n\setminus\{0\}. \end{gather*}
The matrix $\1_n-2\frac{xx^\top}{|x|^2}\in\upO(n)$ is the orthogonal reflection in $\RR^n$ at the hyperplane orthogonal to $x$, and the corresponding elements in the cover $\Pin(n)\subseteq\Cl(n)$ are $\pm\frac{x}{|x|}$. Therefore, we may choose the sign in~\eqref{eq:SgnOfw_0} so that under the identification $\widetilde{M}\simeq\Pin(n)\times\ZZ/2\ZZ$ we have
\begin{gather*} m\big(\tilde{w}_0^{-1}\overline{n}_x^{-1}\big)^{-1} = -\tfrac{x}{|x|} \in \Pin(n) \subseteq \widetilde{M}. \end{gather*}
Further we note that
\begin{gather*} \tilde{w}_0m\tilde{w}_0^{-1} = \alpha(m) = \det(m)m, \qquad m\in\Pin(n), \qquad \tilde{w}_0m_0\tilde{w}_0^{-1} = -m_0, \end{gather*}
where $\alpha$ denotes the canonical automorphism of the Clifford algebra $\Cl(n)$ and $\det\colon \Pin(n)\to\upO(n)\to\{\pm1\}$ the determinant character of $\Pin(n)$.

The symmetric pair $(G,H)=(\upO(n+1,1),\upO(n,1))$ as introduced in Section~\ref{sec:ExScalarGroups} takes in the cover the form $\big(\widetilde{G},\widetilde{H}\big)=(\Pin(n+1,1),\Pin(n,1))$. We have
\begin{gather*} \widetilde{H} = \widetilde{G}\cap\Cl(n,1), \end{gather*}
where the embedding of $\Cl(n,1)$ into $\Cl(n+1,1)$ is induced by the embedding
\begin{gather*} \RR^{n,1}\hookrightarrow\RR^{n+1,1}, \qquad x\mapsto(x_1,\ldots,x_n,0,x_{n+1}). \end{gather*}
Further, $\widetilde{P}_H=\widetilde{P}\cap\widetilde{H}=\widetilde{M}_HAN_H$ is a parabolic subgroup of $\widetilde{H}$ with $\widetilde{M}_H=\Pin(n-1)\cdot\{1,m_0\}$.

\subsection{Principal series representations and symmetry breaking operators}

For the fundamental spin representations of $\Pin(n)$ we use the notation introduced in Appendix~\ref{app:CliffordModules}. The group $\Pin(n)$ has for even $n$ one fundamental spin representation, and for odd~$n$ two fundamental spin representations. For~$n$ even we have $\zeta_n\otimes\det\simeq\zeta_n$ for the fundamental spin representation $\zeta_n$, and for $n$ odd and~$\zeta_n$ any fundamental spin representation of $\Pin(n)$ the representations $\zeta_n$ and $\zeta_n\otimes\det$ are non-equivalent.

Denote by $\sgn\colon \{\1,m_0\}\to\{\pm1\}$ the non-trivial character of the two-element group $\{\1,m_0\}$. Then for $\delta\in\ZZ/2\ZZ$ and a fundamental spin representation $\zeta_n$ of $\Pin(n)$ we consider the representation $\zeta_n\otimes\sgn^\delta$ of $\widetilde{M}=\Pin(n)\times\{\1,m_0\}$. Note that
\begin{gather}
\tilde{w}_0\big(\zeta_n\otimes\sgn^\delta\big) = [\zeta_n\otimes\det]\otimes\sgn^{1-\delta}.\label{eq:ActionOfw_0OnSpinReps}
\end{gather}

For $(\zeta_n,\SS_n)$ a fundamental spin representation of $\Pin(n)$, $\delta\in\ZZ/2\ZZ$ and $\lambda\in\fraka_\CC^*\simeq\CC$ we form the principal series representations
\begin{gather*} \pislash_{\lambda,\delta} = \Ind_{\widetilde{P}}^{\widetilde{G}}\big(\big(\zeta_n\otimes\sgn^\delta\big)\otimes e^\lambda\otimes\1\big). \end{gather*}

Similarly, for $(\zeta_{n-1},\SS_{n-1})$ a fundamental spin representation of $\Pin(n-1)$, $\delta\in\ZZ/2\ZZ$ and $\nu\in\fraka_{H,\CC}^*\simeq\CC$ we form principal series representations of $\widetilde{H}$:
\begin{gather*} \tauslash_{\nu,\varepsilon} = \Ind_{\widetilde{P}_H}^{\widetilde{H}}\big(\big(\zeta_{n-1}\otimes\sgn^\varepsilon\big)\otimes e^\nu\otimes\1\big). \end{gather*}
The main result of this section is the construction of symmetry breaking operators
\begin{gather}
A\in\Hom_{\widetilde{H}}\big(\pislash_{\lambda,\delta}|_{\widetilde{H}},\tauslash_{\nu,\varepsilon}\big).\label{eq:SBObetweenSpinors}
\end{gather}
In Section~\ref{sec:CompactPictureSBOsSpinors} we then show that this construction actually gives a full classification of symmetry breaking operators between spinor-valued principal series.

By Theorem~\ref{thm:DistributionKernelOnNbar} every intertwining operator in \eqref{eq:SBObetweenSpinors} is uniquely determined by its distribution kernel $K\in\calD'\big(\RR^n;\Hom_\CC(\SS_n,\SS_{n-1})\big)$. As explained in Section~\ref{sec:ExScalarPrincipalSeriesSBOs}, the space of distribution kernels describing intertwining operators in \eqref{eq:SBObetweenSpinors} only depends on $(\lambda,\nu)\in\CC^2$ and the parity of $\delta+\varepsilon$. We therefore denote by
\begin{gather*} \calD'\big(\RR^n;\Hom_\CC(\SS_n,\SS_{n-1})\big)_{\lambda,\nu}^\pm \subseteq \calD'\big(\RR^n;\Hom_\CC(\SS_n,\SS_{n-1})\big) \end{gather*}
the corresponding space of distribution kernels of intertwining operators, where the sign $+$ describes kernels for $\delta+\varepsilon\equiv0(2)$ and the sign $-$ describes kernels for $\delta+\varepsilon\equiv1(2)$.

\subsection{Construction of symmetry breaking operators}\label{sec:ConstructionSBOsSpinors}

We extend the representation $(\zeta_n,\SS_n)$ of $\Pin(n)$ to a representation of the Clifford algebra $\Cl(n;\CC)$ and write $(\zeta,\SS)=(\zeta_n,\SS_n)$ for short (see Appendix~\ref{app:CliffordModules}). Consider the representation $\tau'=[\zeta\otimes\det]\otimes\1$ of $\widetilde{M}\simeq\Pin(n)\times\ZZ/2\ZZ$. By Lemma~\ref{lem:ExistenceOfEnoughGreps} there exists a representation $(\tau,E)$ of $\widetilde{G}=\Pin(n+1,1)$ containing a subspace $E'\subseteq E$ invariant under $\widetilde{M}A\overline{N}$ such that $E'\simeq\tau'\otimes e^{\mu'}\otimes\1$ or $E'\simeq\tilde{w}_0\tau'\otimes e^{\mu'}\otimes\1$ for some $\mu'\in\fraka^*$. By possibly replacing $(\tau,E)$ by its twist $(\tau\circ\theta,E)$ by the Cartan involution $\theta$ we may assume that $E'\simeq\tau'\otimes e^{\mu'}\otimes\1$. It is easy to see that $\mu'=-\frac{1}{2}\alpha$. By~\eqref{eq:ActionOfw_0OnSpinReps} we have
\begin{gather*} w_0\big(\tau'\otimes e^{\mu'}\otimes\1\big) \simeq (\zeta\otimes\sgn)\otimes e^{-\mu'}\otimes\1. \end{gather*}
Further, using Lemma~\ref{lem:AMprojection} we have the following expression for the translation kernel:
\begin{gather*} \tau'\big(m\big(\tilde{w}_0^{-1}\overline{n}_x^{-1}\big)^{-1}\big)\cdot a\big(\tilde{w}_0^{-1}\overline{n}_x^{-1}\big)^{-\mu'} = [\zeta\otimes\det]\big({-}\tfrac{x}{|x|}\big)\cdot|x|^{-2\mu'} = \zeta(x). \end{gather*}
Then Corollary~\ref{cor:KnappSteinTranslationPrinciple} gives linear maps
\begin{gather*} \Hom_H\big(\pi_{\lambda,\delta}|_H,\tau_{\nu,\varepsilon}\big) \to \Hom_{\widetilde{H}}\big(\pislash_{\lambda+\frac{1}{2},1-\delta},\Ind_{\widetilde{P}_H}^{\widetilde{H}}\big(\big([\zeta\otimes\det]\otimes\sgn^\varepsilon\big)\otimes e^{\nu-\frac{1}{2}}\otimes\1\big)\big), \end{gather*}
which are on the level of integral kernels given by
\begin{gather}
\calD'\big(\RR^n\big)\to\calD'\big(\RR^n;\End_\CC(\SS)\big), \qquad K(x)\mapsto\zeta(x)\cdot K(x).\label{eq:SpinorTranslationMap}
\end{gather}
To obtain intertwining operators into $\tauslash_{\nu,\varepsilon}$ note that $\zeta_{n{-}1}$ occurs in the restriction $[\zeta_n\!{\otimes}\!\det]|_{\Pin(n{-}1)}\!$ with multiplicity one, and fix a projection $P\colon [\zeta_n\otimes\det]|_{\Pin(n-1)}\to\zeta_{n-1}$. Then we have a linear map
\begin{gather*} \Hom_H\big(\pi_{\lambda,\delta}|_H,\tau_{\nu,\varepsilon}\big) \to \Hom_{\widetilde{H}}\big(\pislash_{\lambda+\frac{1}{2},1-\delta},\tauslash_{\nu-\frac{1}{2},\varepsilon}\big), \end{gather*}
which is on the level of integral kernels given by
\begin{gather*} \calD'\big(\RR^n\big)_{\lambda,\nu}^\pm\to\calD'\big(\RR^n;\Hom_\CC(\SS_n,\SS_{n-1})\big)_{\lambda+\frac{1}{2},\nu-\frac{1}{2}}^\mp, \qquad K(x)\mapsto(P\zeta(x))\cdot K(x). \end{gather*}

By Theorems~\ref{thm:SBOsSphericalClassification} and \ref{thm:SBOsOppositeSignClassification} all symmetry breaking operators in $\Hom_H(\pi_{\lambda,\delta}|_H,\tau_{\nu,\varepsilon})$ are obtained from the meromorphic family of kernels $K^{\AA,\pm}_{\lambda,\nu}\in\calD'\big(\RR^n\big)_{\lambda,\nu}^\pm$. We therefore consider the $\Hom_\CC(\SS_n,\SS_{n-1})$-valued kernels
\begin{gather*} P\Kslash^{\AA,\pm}_{\lambda,\nu}(x) = (P\zeta(x))\cdot K^{\AA,\mp}_{\lambda-\frac{1}{2},\nu+\frac{1}{2}}(x). \end{gather*}

For $n$ odd the restriction of the $\Pin(n)$-representation $\zeta_n\otimes\det$ to $\Pin(n-1)$ is isomorphic to $\zeta_{n-1}$, so the map $P$ is an isomorphism. Therefore, we can as well study the $\End_\CC(\SS)$-valued distributions
\begin{gather*} \Kslash_{\lambda,\nu}^{\AA,\pm}(x) = \zeta(x)\cdot K_{\lambda-\frac{1}{2},\nu+\frac{1}{2}}^{\AA,\mp}(x). \end{gather*}

For $n$ even the restriction of the $\Pin(n)$-representation $\zeta_n\otimes\det$ to $\Pin(n-1)$ is isomorphic to the direct sum of $\zeta_{n-1}$ and $\zeta_{n-1}\otimes\det$. The next result shows that the poles of $P\Kslash_{\lambda,\nu}^{\AA,\pm}\in\calD'\big(\RR^n;\Hom_\CC(\SS_n,\SS_{n-1})\big)$ and $\Kslash_{\lambda,\nu}^{\AA,\pm}\in\calD'\big(\RR^n;\End_\CC(\SS)\big)$ agree and that the residues of $P\Kslash_{\lambda,\nu}^{\AA,\pm}$ agree with the composition of the residues of $\Kslash_{\lambda,\nu}^{\AA,\pm}$ with $P$. Note that
\begin{gather*} \Kslash_{\lambda,\nu}^{\AA,\pm}(x) = \sum_{i=1}^n \big(x_i\cdot K_{\lambda-\frac{1}{2},\nu+\frac{1}{2}}^{\AA,\mp}(x)\big)\underline{e_i}, \end{gather*}
where $\underline{e_i}=\zeta_n(e_i)\in\End_\CC(\SS_n)$.

\begin{Proposition}\label{prop:ProjOntoIrrSummands}Assume $n$ is even and let $u\in\calD'\big(\RR^n;\End_\CC(\SS_n)\big)$ be a distribution of the form
\begin{gather*} u(x) = \sum_{i=1}^n u_i(x)\underline{e_i} \end{gather*}
with $u_i\in\calD'\big(\RR^n\big)$. Then $\supp u=\supp Pu$, in particular $Pu=0$ if and only if $u=0$.
\end{Proposition}

\begin{proof}
We have
\begin{gather*} (Pu)(x) = \sum_{i=1}^n u_i(x)\cdot(P\circ\underline{e_i}). \end{gather*}
We claim that the operators $P\circ\underline{e_1},\ldots,P\circ\underline{e_n}$ are linearly independent, then the statement follows.

Decompose $\zeta_n\otimes\det$ into irreducible $\Pin(n-1)$-representations $\SS_n=\SS_n^1\oplus\SS_n^2$, where $P\colon \SS_n^1\to\SS_{n-1}$ is an isomorphism and $\SS_n^2=\ker P$. Now note that for $1\leq i\leq n-1$ we have
\begin{gather*} P\circ\underline{e_i} = P\circ\zeta_n(e_i) = -P\circ[\zeta_n\otimes\det](e_i) = -\zeta_{n-1}(e_i)\circ P \end{gather*}
and therefore $P\circ\underline{e_i}$ vanishes on $\SS_n^2$. On the other hand, $\underline{e_n}=\zeta_n(e_n)$ maps $\SS_n^1$ to $\SS_n^2$, so that $P\circ\underline{e_n}$ vanishes on $\SS_n^1$. Now assume that
\begin{gather*} \sum_{i=1}^n\lambda_i P\circ\underline{e_i} = 0. \end{gather*}
Restricting to $\SS_n^1$ this implies
\begin{gather*} \sum_{i=1}^{n-1}\lambda_i \zeta_{n-1}(e_i) = 0 \end{gather*}
and hence $\lambda_1=\dots=\lambda_{n-1}=0$ since $\zeta_{n-1}(e_1),\ldots,\zeta_{n-1}(e_{n-1})$ are linearly independent. Then also $\lambda_n=0$ and the proof is complete.
\end{proof}

To find the right normalization making $\Kslash^{\AA,\pm}_{\lambda,\nu}(x)$ holomorphically dependent on $(\lambda,\nu)\in\CC^2$ we first study the kernel in $\calD'\big(\RR^n\big)_{\lambda,\nu}^\pm$ of the map~\eqref{eq:SpinorTranslationMap}. Note that $\zeta(x)=\sum_{i=1}^nx_i\underline{e_i}$, and the operators $\underline{e_1},\ldots,\underline{e_n}\in\End_\CC(\SS)$ are linearly independent. Hence, for any distribution $u\in\calD'\big(\RR^n\big)$ we have
\begin{gather*} \zeta\cdot u = 0 \quad \Leftrightarrow \quad x_iu = 0, \qquad 1\leq i\leq n
\end{gather*}
and
\begin{gather}
\supp(\zeta\cdot u) = \bigcup_{i=1}^n\supp(x_iu).\label{eq:SupportOfKslash}
\end{gather}

\begin{Lemma}\label{lem:SBOsSpinorTranslationKernel}
For $(\lambda,\nu)\in\CC^2$ we have
\begin{gather*}
\big\{K\in\calD'\big(\RR^n\big)_{\lambda,\nu}^+\colon \zeta\cdot K=0\big\} = \begin{cases}\CC\widetilde{K}^{\CC,+}_{\lambda,\nu}=\CC\delta,&\mbox{for $\lambda-\nu=-\frac{1}{2}$}, \\ \{0\},&\mbox{else},\end{cases} \\
\big\{K\in\calD'\big(\RR^n\big)_{\lambda,\nu}^-\colon \zeta\cdot K=0\big\} = \{0\}.
\end{gather*}
\end{Lemma}

\begin{proof}It is easy to see that multiples of $K=\delta$ are the only distributions with $x_iK=0$ for all $1\leq i\leq n$. Then the claim follows from Theorems~\ref{thm:SBOsSpherical}, \ref{thm:SBOsSphericalClassification}, \ref{thm:SBOsOppositeSign} and~\ref{thm:SBOsOppositeSignClassification}.
\end{proof}

To state the analogues of Theorems~\ref{thm:SBOsSpherical} and~\ref{thm:SBOsOppositeSign} for $\Kslash^{\AA,\pm}_{\lambda,\nu}(x)$ we write $\dirac_{\RR^{n-1}}$ and $\dirac_n$ for the operators
\begin{gather*} \dirac_{\RR^{n-1}} = \sum_{i=1}^{n-1} \underline{e_i}\frac{\partial}{\partial x_i}, \qquad \dirac_n = \underline{e_n}\frac{\partial}{\partial x_n}. \end{gather*}

\subsubsection[The distributions $\Kslash^{\AA,+}_{\lambda,\nu}(x)$]{The distributions $\boldsymbol{\Kslash^{\AA,+}_{\lambda,\nu}(x)}$}

Let
\begin{gather*}
\underline{\LocalPoles}^+ = \big\{(\lambda,\nu)\colon \lambda+\nu=-\tfrac{3}{2}-2k,k\in\NN\big\},\\
\underline{\GlobalPoles}^+ = \big\{(\lambda,\nu)\colon \lambda-\nu=-\tfrac{1}{2}-2\ell,k\in\NN\big\},\\
\Lslash_\even = \big\{\big({-}\rho-\tfrac{1}{2}-i,-\rho_H-\tfrac{1}{2}-j\big)\colon i,j\in\NN,i-j\in2\NN\big\}\subseteq\underline{\GlobalPoles}^+.
\end{gather*}

For $(\lambda,\nu)\in\underline{\LocalPoles}^+$ with $\lambda+\nu=-\frac{3}{2}-2k$ we further define
\begin{gather*}
\widetilde{\Kslash}^{\BB,+}_{\lambda,\nu}(x) = \frac{1}{\Gamma\big(\frac{1}{2}\big(\lambda-\nu+\frac{1}{2}\big)\big)}\big(\zeta(x')\big(|x'|^2+x_n^2\big)^{-\nu-\frac{n}{2}}\delta^{(2k+1)}(x_n)\\
\hphantom{\widetilde{\Kslash}^{\BB,+}_{\lambda,\nu}(x) =}{}-(2k+1)\big(|x'|^2+x_n^2\big)^{-\nu-\frac{n}{2}}\dirac_n\delta^{(2k-1)}(x_n)\big),
\end{gather*}
and for $(\lambda,\nu)\in\underline{\GlobalPoles}^+$ with $\lambda-\nu=-\frac{1}{2}-2\ell$ we put
\begin{gather*}
\widetilde{\Kslash}^{\CC,+}_{\lambda,\nu}(x) = \sum_{j=0}^{\ell-1}\frac{2^{2\ell-2j-1}\big(\nu+\frac{1}{2}-\ell\big)_{\ell-j-1}}{j!(2\ell-2j-1)!}\big(\Delta_{\RR^{n-1}}^j\partial_n^{2\ell-2j-1}\dirac_{\RR^{n-1}}\delta\big)(x)\\
\hphantom{\widetilde{\Kslash}^{\CC,+}_{\lambda,\nu}(x) = }{} +\sum_{j=0}^\ell\frac{2^{2\ell-2j}\big(\nu+\frac{1}{2}-\ell\big)_{\ell-j}}{j!(2\ell-2j)!}\big(\Delta_{\RR^{n-1}}^j\partial_n^{2\ell-2j-1}\dirac_n\delta\big)(x).
\end{gather*}
Note that both $\widetilde{\Kslash}^{\BB,+}_{\lambda,\nu}$ and $\widetilde{\Kslash}^{\CC,+}_{\lambda,\nu}$ depend holomorphically on $\lambda\in\CC$ (or $\nu\in\CC$).

\begin{Theorem}\label{thm:SBOsSpinors+}The renormalized distribution
\begin{gather*} \widetilde{\Kslash}^{\AA,+}_{\lambda,\nu}(x) = \frac{\Kslash^{\AA,+}_{\lambda,\nu}(x)}{\Gamma\big(\frac{1}{2}\big(\lambda+\nu+\frac{3}{2}\big)\big)\Gamma\big(\frac{1}{2}\big(\lambda-\nu+\frac{1}{2}\big)\big)} \end{gather*}
depends holomorphically on $(\lambda,\nu)\in\CC^2$ and vanishes only for $(\lambda,\nu)\in\Lslash_\even$. More precisely,
\begin{enumerate}\itemsep=0pt
\item[$1.$] For $(\lambda,\nu)\in\CC^2\setminus(\underline{\LocalPoles}^+\cup\underline{\GlobalPoles}^+)$ we have $\supp\widetilde{\Kslash}^{\AA,+}_{\lambda,\nu}=\RR^n$.
\item[$2.$] For $(\lambda,\nu)\in\underline{\LocalPoles}^+\setminus\underline{\GlobalPoles}^+$ with $\lambda+\nu=-\frac{3}{2}-2k$, $k\in\NN$, we have
\begin{gather*} \widetilde{\Kslash}^{\AA,+}_{\lambda,\nu}(x) = \frac{(-1)^{k+1}k!}{(2k+1)!}\widetilde{\Kslash}^{\BB,+}_{\lambda,\nu}(x). \end{gather*}
In particular, $\supp\widetilde{\Kslash}^{\AA,+}_{\lambda,\nu}=\RR^{n-1}$.
\item[$3.$] For $(\lambda,\nu)\in\underline{\GlobalPoles}^+\setminus\Lslash_\even$ with $\lambda-\nu=-\frac{1}{2}-2\ell$, $\ell\in\NN$, we have
\begin{gather*} \widetilde{\Kslash}^{\AA,+}_{\lambda,\nu}(x) = \frac{(-1)^\ell\ell!\pi^{\frac{n-1}{2}}}{2^{2\ell}\Gamma\big(\nu+\frac{n}{2}\big)}\widetilde{\Kslash}^{\CC,+}_{\lambda,\nu}(x). \end{gather*}
In particular, $\supp\widetilde{\Kslash}^{\AA,+}_{\lambda,\nu}=\{0\}$.
\item[$4.$] For $\big({-}\rho-\frac{1}{2}-i,-\rho_H-\frac{1}{2}-j\big)\in\Lslash_\even$, $2\ell=i-j$, restricting the function $(\lambda,\nu)\mapsto\widetilde{\Kslash}^{\AA,+}_{\lambda,\nu}$ to two different complex hyperplanes in $\CC^2$ through $\big({-}\rho-\frac{1}{2}-i,-\rho_H-\frac{1}{2}-j\big)$ and renormalizing gives two holomorphic families of distributions:
\begin{gather*}
\widetilde{\Kslash}^{\CC,+}_{\lambda,\nu}(x) = \frac{(-1)^\ell2^{2\ell}}{\ell!\pi^{\frac{n-1}{2}}}\Gamma\big(\nu+\rho_H+\tfrac{1}{2}\big)\widetilde{\Kslash}^{\AA,+}_{\lambda,\nu}(x), \qquad \lambda-\nu = -\tfrac{1}{2}-(i-j),\\
\doublewidetilde{\Kslash}^{\AA,+}_{\lambda,\nu}(x) = \Gamma\big(\tfrac{1}{2}\big(\lambda-\nu+\tfrac{1}{2}\big)\big)\widetilde{\Kslash}^{\AA,+}_{\lambda,\nu}(x), \qquad \nu = -\rho_H-\tfrac{1}{2}-j.
\end{gather*}
Their special values at $(\lambda,\nu)=\big({-}\rho-\frac{1}{2}-i,-\rho_H-\frac{1}{2}-j\big)$ satisfy
\begin{gather*}
\supp\widetilde{\Kslash}^{\CC,+}_{-\rho-\frac{1}{2}-i,-\rho_H-\frac{1}{2}-j} = \{0\},\\
\supp\doublewidetilde{\Kslash}^{\AA,+}_{-\rho-\frac{1}{2}-i,-\rho_H-\frac{1}{2}-j} = \begin{cases}\RR^n,&\mbox{for $n$ even,}\\\RR^{n-1},&\mbox{for $n$ odd,}\end{cases}
\end{gather*}
more precisely, for $n$ odd and $2k=n+i+j-1$
\begin{gather*} \doublewidetilde{\Kslash}^{\AA,+}_{-\rho-\frac{1}{2}-i,-\rho_H-\frac{1}{2}-j} = \frac{(-1)^{k+1}k!}{(2k+1)!}\big(|x'|^2+x_n^2\big)^j\\
\hphantom{\doublewidetilde{\Kslash}^{\AA,+}_{-\rho-\frac{1}{2}-i,-\rho_H-\frac{1}{2}-j} =}{}\times \big(\zeta(x')\delta^{(2k+1)}(x_n)-(2k+1)\dirac_n\delta^{(2k-1)}(x_n)\big). \end{gather*}
\end{enumerate}
\end{Theorem}

\begin{proof}By Theorem~\ref{thm:SBOsOppositeSign} the distribution
\begin{gather*} \widetilde{\Kslash}^{\AA,+}_{\lambda,\nu}(x) = \zeta(x)\cdot\widetilde{K}^{\AA,-}_{\lambda-\frac{1}{2},\nu+\frac{1}{2}}(x) = \frac{\zeta(x)\cdot K^{\AA,-}_{\lambda-\frac{1}{2},\nu+\frac{1}{2}}(x)}{\Gamma\big(\frac{1}{2}\big(\lambda+\nu+\frac{3}{2}\big)\big)\Gamma\big(\frac{1}{2}\big(\lambda-\nu+\frac{1}{2}\big)\big)} \end{gather*}
depends holomorphically on $(\lambda,\nu)\in\CC^2$, and by Lemma~\ref{lem:SBOsSpinorTranslationKernel} it vanishes if and only if $\big(\lambda-\frac{1}{2},\nu+\frac{1}{2}\big)\in L_\odd$, i.e., $(\lambda,\nu)\in\Lslash_\even$. The rest is similar to the proof of Theorem~\ref{thm:SBOsSpinors-} which we carry out in detail, some arguments are even easier since in this case we do not need to renormalize the kernel and we have
\begin{gather*}
(\lambda,\nu)\in\underline{\LocalPoles}^+ \Leftrightarrow \big(\lambda-\tfrac{1}{2},\nu+\tfrac{1}{2}\big)\in\LocalPoles^-, \qquad (\lambda,\nu)\in\underline{\GlobalPoles}^+ \Leftrightarrow \big(\lambda-\tfrac{1}{2},\nu+\tfrac{1}{2}\big)\in\GlobalPoles^-.\tag*{\qed}
\end{gather*}\renewcommand{\qed}{}
\end{proof}

\subsubsection[The distributions $\Kslash^{\AA,-}_{\lambda,\nu}(x)$]{The distributions $\boldsymbol{\Kslash^{\AA,-}_{\lambda,\nu}(x)}$}

Let
\begin{gather*}
\underline{\LocalPoles}^- = \big\{(\lambda,\nu)\colon \lambda+\nu=-\tfrac{1}{2}-2k,k\in\NN\big\},\\
\underline{\GlobalPoles}^- = \big\{(\lambda,\nu)\colon \lambda-\nu=-\tfrac{3}{2}-2\ell,k\in\NN\big\},\\
\Lslash_\odd = \big\{\big({-}\rho-\tfrac{1}{2}-i,-\rho_H-\tfrac{1}{2}-j\big)\colon i,j\in\NN,i-j\in2\NN+1\big\}\subseteq\underline{\GlobalPoles}^-.
\end{gather*}

For $(\lambda,\nu)\in\underline{\LocalPoles}^-$ with $\lambda+\nu=-\frac{1}{2}-2k$ we further define
\begin{gather*}
\widetilde{\Kslash}^{\BB,-}_{\lambda,\nu}(x) = \frac{1}{\Gamma\big(\frac{1}{2}\big(\lambda-\nu+\frac{3}{2}\big)\big)}\big(\zeta(x')\big(|x'|^2+x_n^2\big)^{-\nu-\frac{n}{2}}\delta^{(2k)}(x_n)\\
\hphantom{\widetilde{\Kslash}^{\BB,-}_{\lambda,\nu}(x) =}{} -2k\big(|x'|^2+x_n^2\big)^{-\nu-\frac{n}{2}}\dirac_n\delta^{(2k-2)}(x_n)\big),
\end{gather*}
 and for $(\lambda,\nu)\in\underline{\GlobalPoles}^-$ with $\lambda-\nu=-\frac{3}{2}-2\ell$ we put
\begin{gather*}
\widetilde{\Kslash}^{\CC,-}_{\lambda,\nu}(x) = \sum_{j=0}^\ell\frac{2^{2\ell-2j}(\nu-\ell-\frac{1}{2})_{\ell-j}}{j!(2\ell-2j)!}\big(\Delta_{\RR^{n-1}}^j\partial_n^{2\ell-2j}\dirac_{\RR^{n-1}}\delta\big)(x)\\
\hphantom{\widetilde{\Kslash}^{\CC,-}_{\lambda,\nu}(x) =}{} +\sum_{j=0}^\ell\frac{2^{2\ell-2j+1}(\nu-\ell-\frac{1}{2})_{\ell-j+1}}{j!(2\ell-2j+1)!}\big(\Delta_{\RR^{n-1}}^j\partial_n^{2\ell-2j}\dirac_n\delta\big)(x).
\end{gather*}
Note that both $\widetilde{\Kslash}^{\BB,-}_{\lambda,\nu}$ and $\widetilde{\Kslash}^{\CC,-}_{\lambda,\nu}$ depend holomorphically on $\lambda\in\CC$ (or $\nu\in\CC$).

\begin{Theorem}\label{thm:SBOsSpinors-}
The renormalized distribution
\begin{gather*} \widetilde{\Kslash}^{\AA,-}_{\lambda,\nu}(x) = \frac{\Kslash^{\AA,-}_{\lambda,\nu}(x)}{\Gamma\big(\frac{1}{2}\big(\lambda+\nu+\frac{1}{2}\big)\big)\Gamma\big(\frac{1}{2}\big(\lambda-\nu+\frac{3}{2}\big)\big)} \end{gather*}
depends holomorphically on $(\lambda,\nu)\in\CC^2$ and vanishes only for $(\lambda,\nu)\in\Lslash_\odd$. More precisely,
\begin{enumerate}\itemsep=0pt
\item[$1.$] For $(\lambda,\nu)\in\CC^2\setminus(\underline{\LocalPoles}^-\cup\underline{\GlobalPoles}^-)$ we have $\supp\widetilde{\Kslash}^{\AA,-}_{\lambda,\nu}=\RR^n$.
\item[$2.$] For $(\lambda,\nu)\in\underline{\LocalPoles}^-\setminus\underline{\GlobalPoles}^-$ with $\lambda+\nu=-\frac{1}{2}-2k$, $k\in\NN$, we have
\begin{gather*} \widetilde{\Kslash}^{\AA,-}_{\lambda,\nu}(x) = \frac{(-1)^kk!}{(2k)!}\widetilde{\Kslash}^{\BB,-}_{\lambda,\nu}(x). \end{gather*}
In particular, $\supp\widetilde{\Kslash}^{\AA,-}_{\lambda,\nu}=\RR^{n-1}$.
\item[$3.$] For $(\lambda,\nu)\in\underline{\GlobalPoles}^-\setminus\Lslash_\odd$ with $\lambda-\nu=-\frac{3}{2}-2\ell$, $\ell\in\NN$, we have
\begin{gather*} \widetilde{\Kslash}^{\AA,-}_{\lambda,\nu}(x) = \frac{(-1)^{\ell+1}\ell!\pi^{\frac{n-1}{2}}}{2^{2\ell+1}\Gamma\big(\nu+\frac{n}{2}\big)}\widetilde{\Kslash}^{\CC,-}_{\lambda,\nu}(x). \end{gather*}
In particular, $\supp\widetilde{\Kslash}^{\AA,-}_{\lambda,\nu}=\{0\}$.
\item[$4.$] For $\big({-}\rho-\frac{1}{2}-i,-\rho_H-\frac{1}{2}-j\big)\in\Lslash_\odd$, $2\ell=i-j-1$, restricting the function $(\lambda,\nu)\mapsto\widetilde{\Kslash}^{\AA,-}_{\lambda,\nu}$ to two different complex hyperplanes in $\CC^2$ through $\big({-}\rho-\frac{1}{2}-i,-\rho_H-\frac{1}{2}-j\big)$ and renormalizing gives two holomorphic families of distributions:
\begin{gather*}
\widetilde{\Kslash}^{\CC,-}_{\lambda,\nu}(x) = \frac{(-1)^{\ell+1}2^{2\ell+1}}{\ell!\pi^{\frac{n-1}{2}}}\Gamma\big(\nu+\rho_H+\tfrac{1}{2}\big)\widetilde{\Kslash}^{\AA,-}_{\lambda,\nu}(x), \qquad \lambda-\nu = -\tfrac{1}{2}-(i-j),\\
\doublewidetilde{\Kslash}^{\AA,-}_{\lambda,\nu}(x) = \Gamma\big(\tfrac{1}{2}\big(\lambda-\nu+\tfrac{3}{2}\big)\big)\widetilde{\Kslash}^{\AA,-}_{\lambda,\nu}(x), \qquad \nu = -\rho_H-\tfrac{1}{2}-j.
\end{gather*}
Their special values at $(\lambda,\nu)=\big({-}\rho-\frac{1}{2}-i,-\rho_H-\frac{1}{2}-j\big)$ satisfy
\begin{gather*}
\supp\widetilde{\Kslash}^{\CC,-}_{-\rho-\frac{1}{2}-i,-\rho_H-\frac{1}{2}-j} = \{0\},\\
\supp\doublewidetilde{\Kslash}^{\AA,-}_{-\rho-\frac{1}{2}-i,-\rho_H-\frac{1}{2}-j} = \begin{cases}\RR^n,&\mbox{for $n$ even,}\\\RR^{n-1},&\mbox{for $n$ odd,}\end{cases}
\end{gather*}
more precisely, for $n$ odd and $2k=n+i+j$:
\begin{gather*} \doublewidetilde{\Kslash}^{\AA,-}_{-\rho-\frac{1}{2}-i,-\rho_H-\frac{1}{2}-j} = \frac{(-1)^kk!}{(2k)!}\big(|x'|^2+x_n^2\big)^j\big(\zeta(x')\delta^{(2k)}(x_n)-2k\dirac_n\delta^{(2k-2)}(x_n)\big). \end{gather*}
\end{enumerate}
\end{Theorem}

\begin{proof}By Theorem~\ref{thm:SBOsSpherical} the distribution
\begin{gather*} \zeta(x)\cdot\widetilde{K}^{\AA,+}_{\lambda-\frac{1}{2},\nu+\frac{1}{2}}(x) = \frac{\zeta(x)\cdot K^{\AA,+}_{\lambda-\frac{1}{2},\nu+\frac{1}{2}}(x)}{\Gamma\big(\frac{1}{2}\big(\lambda+\nu+\frac{1}{2}\big)\big)\Gamma\big(\frac{1}{2}\big(\lambda-\nu-\frac{1}{2}\big)\big)} \end{gather*}
depends holomorphically on $(\lambda,\nu)\in\CC^2$, and by Lemma~\ref{lem:SBOsSpinorTranslationKernel} it vanishes if and only if $\lambda-\nu=\frac{1}{2}$ or $\big(\lambda-\frac{1}{2},\nu+\frac{1}{2}\big)\in L_\even$. Renormalizing shows that
\begin{gather*}
\widetilde{\Kslash}^{\AA,-}_{\lambda,\nu}(x) = \frac{\zeta(x)\cdot\widetilde{K}^{\AA,+}_{\lambda-\frac{1}{2},\nu+\frac{1}{2}}(x)}{\frac{1}{2}(\lambda-\nu-\frac{1}{2})} = \frac{\zeta(x)\cdot K^{\AA,+}_{\lambda-\frac{1}{2},\nu+\frac{1}{2}}(x)}{\Gamma\big(\frac{1}{2}\big(\lambda+\nu+\frac{1}{2}\big)\big)\Gamma\big(\frac{1}{2}\big(\lambda-\nu+\frac{3}{2}\big)\big)}
\end{gather*}
depends holomorphically on $(\lambda,\nu)\in\CC^2$. To prove the remaining statements, note that
\begin{gather*}
\zeta(x)\cdot\widetilde{\Kslash}^{\AA,-}_{\lambda,\nu}(x) = \frac{\zeta(x)^2\cdot\widetilde{K}^{\AA,+}_{\lambda-\frac{1}{2},\nu+\frac{1}{2}}(x)}{\frac{1}{2}\big(\lambda-\nu-\frac{1}{2}\big)} = -\frac{|x|^2\cdot\widetilde{K}^{\AA,+}_{\lambda-\frac{1}{2},\nu+\frac{1}{2}}(x)}{\frac{1}{2}\big(\lambda-\nu-\frac{1}{2}\big)}\cdot\id_\SS
 = -\widetilde{K}^{\AA,+}_{\lambda+\frac{1}{2},\nu-\frac{1}{2}}(x)\cdot\id_\SS
\end{gather*}
and hence
\begin{gather}
\supp\widetilde{K}^{\AA,+}_{\lambda+\frac{1}{2},\nu-\frac{1}{2}} \subseteq \supp\widetilde{\Kslash}^{\AA,-}_{\lambda,\nu}.\label{eq:Kslash-SupportInclusion}
\end{gather}

1.~Let $(\lambda,\nu)\in\CC^2\setminus(\underline{\LocalPoles}^-\cup\underline{\GlobalPoles}^-)$, then $(\lambda+\frac{1}{2},\nu-\frac{1}{2})\in\CC^2\setminus(\LocalPoles^+\cup\GlobalPoles^+)$ and hence $\supp\widetilde{K}^{\AA,+}_{\lambda+\frac{1}{2},\nu+\frac{1}{2}}=\RR^n$. Now \eqref{eq:Kslash-SupportInclusion} implies $\supp\widetilde{\Kslash}^{\AA,-}_{\lambda,\nu}=\RR^n$.

2.~Let $(\lambda,\nu)\in\underline{\LocalPoles}^-\setminus\underline{\GlobalPoles}^-$ with $\lambda+\nu=-\frac{1}{2}-2k$, then $\big(\lambda-\frac{1}{2},\nu+\frac{1}{2}\big)\in\LocalPoles^+$ and by Theorem~\ref{thm:SBOsSpherical}(2) and using $x_n\delta^{(2k)}(x_n)=-2k\delta^{(2k-1)}(x_n)$ we have
\begin{gather*}
\widetilde{\Kslash}^{\AA,-}_{\lambda,\nu}(x) = \frac{\zeta(x)\cdot\widetilde{K}^{\AA,+}_{\lambda-\frac{1}{2},\nu+\frac{1}{2}}(x)}{\frac{1}{2}\big(\lambda-\nu-\frac{1}{2}\big)} = \frac{(-1)^kk!}{(2k)!}\cdot\frac{\zeta(x)\cdot\widetilde{K}^{\BB,+}_{\lambda-\frac{1}{2},\nu+\frac{1}{2}}}{\frac{1}{2}\big(\lambda-\nu-\frac{1}{2}\big)}
 = \frac{(-1)^kk!}{(2k)!}\widetilde{\Kslash}^{\BB,-}_{\lambda,\nu}(x).
\end{gather*}

3.~Let $(\lambda,\nu)\in\underline{\GlobalPoles}^-\setminus\Lslash_\odd$ with $\lambda-\nu=-\frac{3}{2}-2\ell$, then $\big(\lambda-\frac{1}{2},\nu+\frac{1}{2}\big)\in\GlobalPoles^+$ with $\big(\lambda-\frac{1}{2})-(\nu+\frac{1}{2}\big)=-\frac{1}{2}-2(\ell+1)$ and by Theorem~\ref{thm:SBOsSpherical}(3) and using $\big[\Delta_{\RR^{n-1}}^j,x_i\big]=2j\Delta^{j-1}\frac{\partial}{\partial x_i}$ we have
\begin{align*}
\widetilde{\Kslash}^{\AA,-}_{\lambda,\nu}(x) &= \frac{\zeta(x)\cdot\widetilde{K}^{\AA,+}_{\lambda-\frac{1}{2},\nu+\frac{1}{2}}(x)}{\frac{1}{2}\big(\lambda-\nu-\frac{1}{2}\big)} = \frac{(-1)^{\ell+1}(\ell+1)!\pi^{\frac{n-1}{2}}}{2^{2\ell+2}\Gamma\big(\nu+\frac{n}{2}\big)}\cdot\frac{\zeta(x)\cdot\widetilde{K}^{\CC,+}_{\lambda-\frac{1}{2},\nu+\frac{1}{2}}}{(-\ell-1)}\\
&= \frac{(-1)^{\ell+1}\ell!\pi^{\frac{n-1}{2}}}{2^{2\ell+1}\Gamma\big(\nu+\frac{n}{2}\big)}\widetilde{\Kslash}^{\CC,-}_{\lambda,\nu}(x).
\end{align*}

4.~Now let $\big({-}\rho-\frac{1}{2}-i,-\rho_H-\frac{1}{2}-j\big)\in\Lslash_\odd$. The first statement about $\widetilde{\Kslash}^{\CC,-}_{\lambda,\nu}(x)$ follows immediately from (3). For the second statement note that $(-\rho-(i+1),-\rho_H-j)\in L_\even$ and we can use Theorem~\ref{thm:SBOsSpherical}(4). Restricting to $\nu=-\rho_H-\frac{1}{2}-j$ we have
\begin{align*}
\doublewidetilde{\Kslash}^{\AA,-}_{\lambda,\nu}(x) &= \Gamma\big(\tfrac{1}{2}\big(\lambda-\nu+\tfrac{3}{2}\big)\big)\widetilde{\Kslash}^{\AA,-}_{\lambda,\nu}\\
&= \frac{\Gamma\big(\tfrac{1}{2}\big(\lambda-\nu+\tfrac{3}{2}\big)\big)\zeta(x)\cdot\widetilde{K}^{\AA,+}_{\lambda-\frac{1}{2},\nu+\frac{1}{2}}(x)}{\frac{1}{2} \big(\lambda-\nu-\frac{1}{2}\big)} = \zeta(x)\cdot\doublewidetilde{K}^{\AA,+}_{\lambda-\frac{1}{2},\nu+\frac{1}{2}}(x),
\end{align*}
which clearly depends holomorphically on $\lambda\in\CC$. Moreover, $\supp\doublewidetilde{K}^{\AA,+}_{-\rho-(i+1),-\rho_H-j}=\RR^n$, resp.~$\RR^{n-1}$, so that $\supp x_i\doublewidetilde{K}^{\AA,+}_{-\rho-(i+1),-\rho_H-j}=\RR^n$, resp.~$\RR^{n-1}$, for $1\leq i\leq n-1$, and hence $\supp\doublewidetilde{\Kslash}^{\AA,-}_{-\rho-\frac{1}{2}-i,-\rho_H-\frac{1}{2}-j}=\RR^n$, resp.~$\RR^{n-1}$, by~\eqref{eq:SupportOfKslash}. For odd $n$ and $2k=n+i+j$ we further have by Theorem~\ref{thm:SBOsSpherical}(4)
\begin{align*}
\hspace{-.3cm}\doublewidetilde{\Kslash}^{\AA,-}_{-\rho-\frac{1}{2}-i,-\rho_H-\frac{1}{2}-j}(x) &= \zeta(x)\cdot\doublewidetilde{K}^{\AA,+}_{-\rho-(i+1),-\rho_H-j}(x)
 = \frac{(-1)^kk!}{(2k)!}\zeta(x)\cdot\big(|x'|^2+x_n\big)^j\delta^{(2k)}(x_n)\\
&= \frac{(-1)^kk!}{(2k)!}\big(|x'|^2+x_n^2\big)^j\big(\zeta(x')\delta^{(2k)}(x_n)-2k\dirac_n\delta^{(2k-2)}\big).\tag*{\qed}
\end{align*}\renewcommand{\qed}{}
\end{proof}

\begin{Remark}The identity
\begin{gather*} \zeta(x)\cdot\widetilde{\Kslash}^{\AA,-}_{\lambda,\nu}(x) = -\widetilde{K}^{\AA,+}_{\lambda+\frac{1}{2},\nu-\frac{1}{2}}(x)\cdot\id_\SS \end{gather*}
that was used in the previous proof can be explained by again applying the translation principle, now with the dual representation $(\tau')^*$, and then projecting onto the trivial constituent in $\tau'\otimes(\tau')^*=\End_\CC(\tau')$ which is spanned by $\id_\SS$.
\end{Remark}

\begin{Remark}\label{rem:JuhlOperatorsSpinors}The intertwining differential operators $C^\infty\big(\RR^n;\SS\big)\to C^\infty\big(\RR^{n-1},\SS\big)$ with integral kernel $\widetilde{\Kslash}^{\CC,\pm}_{\lambda,\nu}(x)$ have been obtained before \cite[Theorem~5.7]{KOSS15} and it was conjectured that these operators exhaust the space of all intertwining differential operators. Theorem~\ref{thm:ClassificationSBOsSpinors} confirms this conjecture.
\end{Remark}

\section[The compact picture of symmetry breaking operators between spinors]{The compact picture of symmetry breaking operators\\ between spinors}\label{sec:CompactPictureSBOsSpinors}

We use the method developed in~\cite{MO14a} to show that the symmetry breaking operators found in Section~\ref{sec:ExSpinors} between spinor-valued principal series span the space of all symmetry breaking operators. The notation will be as in the previous section.

\subsection[Reduction to the pair $\big(\widetilde{G}_1,\widetilde{H}_1\big)$]{Reduction to the pair $\boldsymbol{\big(\widetilde{G}_1,\widetilde{H}_1\big)}$}

To simplify computations we first reduce the study of symmetry breaking operators for the pair $\big(\widetilde{G},\widetilde{H}\big)=(\Pin(n+1,1),\Pin(n,1))$ to the pair $\big(\widetilde{G}_1,\widetilde{H}_1\big)=(\Pin(n+1,1)_1,\Pin(n,1)_1)$, where
\begin{gather*} \Pin(p,q)_1 = \Pin(p,q)_{++}\cup\Pin(p,q)_{-+}. \end{gather*}
We refer the reader to Appendix~\ref{app:CliffordAlgebrasPinGroups} for the definition of $\Pin(p,q)_{\pm\pm}$. Note that $\widetilde{P}_1=\widetilde{P}\cap\widetilde{G}_1=\widetilde{M}_1AN$ is a parabolic subgroup of $\widetilde{G}_1$ with $\widetilde{M}_1=\Pin(n)$. Similarly, $\widetilde{P}_{H,1}=\widetilde{P}_H\cap\widetilde{H}_1=\widetilde{M}_{H,1}AN_H$ is a parabolic subgroup of $\widetilde{H}_1$ with $\widetilde{M}_{H,1}=\Pin(n-1)$. Further, $\widetilde{G}_1/\widetilde{P}_1\simeq\widetilde{G}/\widetilde{P}$ and $\widetilde{H}_1/\widetilde{P}_{H,1}\simeq\widetilde{H}/\widetilde{P}_H$, and hence
\begin{gather*}
\pislash_{\lambda,\delta}|_{\widetilde{G}_1} \simeq \pislash_\lambda = \Ind_{\widetilde{P}_1}^{\widetilde{G}_1}\big(\zeta_n\otimes e^\lambda\otimes\1\big),\\
\tauslash_{\nu,\varepsilon}|_{\widetilde{H}_1} \simeq \tauslash_\nu = \Ind_{\widetilde{P}_{H,1}}^{\widetilde{H}_1}\big(\zeta_{n-1}\otimes e^\nu\otimes\1\big).
\end{gather*}
Note that the restrictions are independent of $\delta$ and $\varepsilon$.

\begin{Lemma}\label{lem:ReductionToG1H1}
For $\lambda,\nu\in\CC$ and $\delta,\varepsilon\in\ZZ/2\ZZ$ we have
\begin{gather*} \dim\Hom_{\widetilde{H}}(\pislash_{\lambda,\delta},\tauslash_{\nu,\varepsilon})+\dim\Hom_{\widetilde{H}}(\pislash_{\lambda,1-\delta},\tauslash_{\nu,\varepsilon}) = \dim\Hom_{\widetilde{H}_1}(\pislash_\lambda,\tauslash_\nu). \end{gather*}
\end{Lemma}

\begin{proof}
The proof of the two identities is analogous to \cite[Theorem 2.10]{KKP16} and uses $\widetilde{G}/\widetilde{G}_1\simeq\widetilde{M}/\widetilde{M}_1\simeq\ZZ/2\ZZ$ and $\widetilde{H}/\widetilde{H}_1\simeq\widetilde{M}_H/\widetilde{M}_{H,1}\simeq\ZZ/2\ZZ$.
\end{proof}

In the remaining part of this section we determine $\dim\Hom_{\widetilde{H}_1}(\pislash_\lambda,\tauslash_\nu)$ using the technique developed in \cite{MO14a}.

\subsection{Harish-Chandra modules}

Let $\theta$ denote the Cartan involution of $G=\upO(n+1,1)$ given by conjugation with the matrix $\diag(1,\ldots,1,-1)$. Then $\theta$ lifts to a Cartan involution of $\widetilde{G}$ which restricts to Cartan involutions of $\widetilde{G}_1$, $\widetilde{H}$ and $\widetilde{H}_1$. This gives the following maximal compact subgroups:
\begin{gather*} \widetilde{K}_1=\widetilde{G}_1^\theta=\Pin(n+1), \qquad \widetilde{K}_{H,1}=\widetilde{H}_1^\theta=\Pin(n). \end{gather*}
The embedding $\widetilde{K}_{H,1}\subseteq\widetilde{K}_1$ is given by the embedding $\Cl(n)\subseteq\Cl(n+1)$ induced from the standard embedding $\RR^n\subseteq\RR^{n+1},\,x'\mapsto(x',0)$.

The method in \cite{MO14a} actually constructs and classifies intertwining operators between the underlying Harish-Chandra modules of $\pislash_\lambda$ and $\tauslash_\nu$ which we denote by $(\pislash_\lambda)_\HC$ and $(\tauslash_\nu)_\HC$. As vector space $(\pislash_\lambda)_\HC$ is the space of all $\widetilde{K}_1$-finite vectors in $\pislash_\lambda$, or equivalently the sum of all irreducible $\widetilde{K}_1$-subrepresentations of $\pislash_\lambda$. Hence, $(\pislash_\lambda)_\HC$ clearly carries a $\widetilde{K}_1$-action. It is further invariant under the action of the Lie algebra $\frakg$ of $\widetilde{G}_1$ and hence carries the structure of a $\big(\frakg,\widetilde{K}_1\big)$-module.

Since $(\pislash_\lambda)_\HC$ is dense in $\pislash_\lambda$, we obtain an injective map
\begin{gather*} \Hom_{\widetilde{H}_1}(\pislash_\lambda|_{\widetilde{H}_1},\tauslash_\nu) \hookrightarrow \Hom_{(\frakh,\widetilde{K}_{H,1})}((\pislash_\lambda)_\HC|_{(\frakh,\widetilde{K}_{H,1})},(\tauslash_\nu)_\HC). \end{gather*}
Using the method in \cite{MO14a} we determine in this section the dimension of the space of intertwining operators between the Harish-Chandra modules. This gives an upper bound for the dimension of the space of intertwining operators between the representations $\pislash_\lambda$ and $\tauslash_\nu$. From the explicit construction of intertwining operators in Section~\ref{sec:ExSpinors} we also have a lower bound, and it turns out that these bounds agree, so that the injective map in fact is a bijection.

\subsection[$K$-types]{$\boldsymbol{K}$-types}\label{sec:SpinorKtypes}

We study the representations $\pislash_\lambda$ and $\tauslash_\nu$ in the compact picture, i.e. on sections of spin bundles over $\widetilde{K}_1/\widetilde{M}_1\simeq S^n$ and $\widetilde{K}_{H,1}/\widetilde{M}_{H,1}\simeq S^{n-1}$. More precisely,
\begin{gather*} \pislash_\lambda|_{\widetilde{K}_1} = C^\infty\big(\widetilde{K}_1\times_{\widetilde{M}_1}\zeta_n\big). \end{gather*}
By Frobenius Reciprocity and using the classical branching rules for spin groups we easily obtain
\begin{align*}
(\pislash_\lambda)_\HC|_{\widetilde{K}_1} &\simeq \begin{cases}\displaystyle\bigoplus_{i=0}^\infty\big(\zeta_{n+1,i}^+\oplus\zeta_{n+1,i}^-\big),&\mbox{for $n$ even,}\\\displaystyle\bigoplus_{i=0}^\infty\zeta_{n+1,i},&\mbox{for $n$ odd,}\end{cases}
\end{align*}
where $\zeta_{n+1,i}^{(\pm)}$ denote the higher spin representations of $\Pin(n+1)$ (see Appendix~\ref{app:HigherSpinReps}). By Appendix~\ref{app:BranchingLaws}
\begin{alignat*}{3}
& \zeta_{n+1,i}^\pm|_{\widetilde{K}_{H,1}} \simeq \bigoplus_{j=0}^i\zeta_{n,j}\qquad && \mbox{($n$ even)},&\\
& \zeta_{n+1,i}|_{\widetilde{K}_{H,1}}\simeq \bigoplus_{j=0}^i\big(\zeta_{n,j}^+\oplus\zeta_{n,j}^-\big)\qquad && \mbox{($n$ odd)}.&
\end{alignat*}
We use the notation $\alpha=(i,\pm)$ and $\alpha'=j$ for even $n$, and $\alpha=i$, $\alpha'=(j,\pm)$ for odd $n$, and write $\alpha'\subset\alpha$ if $j\leq i$. Denote by $\calE(\alpha)$ the subspace of $C^\infty\big(\widetilde{K}_1\times_{\widetilde{M}_1}\zeta_n\big)$ which is isomorphic to the $\widetilde{K}_1$-representation $\zeta_{n+1,i}^\pm$ if $\alpha=(i,\pm)$ and $\zeta_{n+1,i}$ if $\alpha=i$, and similar for the corresponding subspace $\calE'(\alpha')$ of $C^\infty\big(\widetilde{K}_{H,1}\times_{\widetilde{M}_{H,1}}\zeta_{n-1}\big)$. In this notation
\begin{gather*} (\pislash_\lambda)_\HC|_{\widetilde{K}_1} = \bigoplus_\alpha\calE(\alpha), \qquad (\tauslash_\nu)_\HC|_{\widetilde{K}_{H,1}} = \bigoplus_{\alpha'}\calE'(\alpha'). \end{gather*}
Further,
\begin{gather*} \calE(\alpha)|_{\widetilde{K}_{H,1}} = \bigoplus_{\alpha'\subset\alpha}\calE(\alpha,\alpha'), \end{gather*}
where $\calE(\alpha,\alpha')\simeq\calE'(\alpha')$. We fix a non-zero $\widetilde{K}_{H,1}$-intertwining operator
\begin{gather*} R_{\alpha,\alpha'}\colon \ \calE(\alpha,\alpha')\stackrel{\sim}{\to}\calE'(\alpha') \end{gather*}
for each pair $(\alpha,\alpha')$ with $\alpha'\subset\alpha$.

\subsection{Scalar identities for symmetry breaking operators}\label{sec:ScalarIdentities}

Write $\frakg=\frakk\oplus\fraks$ for the eigenspace decomposition of $\frakg$ under $\theta$ and $\frakh=\frakk_H\oplus\fraks_H$ for the one of~$\frakh$. We identify $\fraks\simeq\RR^{n+1}$ via
\begin{gather*} \RR^{n+1}\stackrel{\sim}{\longrightarrow}\fraks, \qquad Y\mapsto\left(\begin{matrix} \0_{n+1}&Y\\Y^\top&0\end{matrix}\right). \end{gather*}
Consider the map $\omega\colon \fraks_\CC\to C^\infty\big(\widetilde{K}_1/\widetilde{M}_1\big)=C^\infty(S^n)$ given by
\begin{gather*} \omega(Y)(x) = Y^\top x, \qquad x\in S^n,Y\in\CC^{n+1}. \end{gather*}
Multiplication by $\omega(Y)$ defines an operator $M(\omega(Y))\colon C^\infty\big(\widetilde{K}_1\times_{\widetilde{M}_1}\zeta_n\big)\to C^\infty\big(\widetilde{K}_1\times_{\widetilde{M}_1}\zeta_n\big)$. From the weights of $\fraks_\CC$ it is easy to see that $M(\omega(Y))$ maps $\calE(\alpha)$ to $\calE(\beta)$ if and only if
\begin{alignat*}{3}
& \alpha=(i,\pm) \quad \mbox{and} \quad \beta\in\{(i+1,\pm),(i,\mp),(i-1,\pm)\} \qquad && \mbox{($n$ even)},& \\
 & \alpha=i \quad \mbox{and}\quad \beta\in\{i+1,i,i-1\} \qquad && \mbox{($n$ odd)}.&
\end{alignat*}
Write $\alpha\leftrightarrow\beta$ if this is the case.

We use the same notation for $\alpha'$ and $\beta'$ and write $M(\omega'(Y))$ for multiplication by the function $\omega'(Y)(x)=Y^\top x$, $x\in S^{n-1}$, $Y\in\CC^n$. Further write $(\alpha,\alpha')\leftrightarrow(\beta,\beta')$ if $\alpha'\subset\alpha$, $\beta'\subset\beta$ and $\alpha\leftrightarrow\beta$, $\alpha'\leftrightarrow\beta'$.

For $(\alpha,\alpha')\leftrightarrow(\beta,\beta')$ define
\begin{gather*}
\omega_{\alpha,\alpha'}^{\beta,\beta'}(Y) = \proj_{\calE(\beta,\beta')}\circ M(\omega(Y))|_{\calE(\alpha,\alpha')},\\
\omega_{\alpha'}^{\beta'}(Y) = \proj_{\calE'(\beta')}\circ M(\omega'(Y))|_{\calE'(\alpha')}.
\end{gather*}
Then both $R_{\beta,\beta'}\circ\omega_{\alpha,\alpha'}^{\beta,\beta'}$ and $\omega_{\alpha'}^{\beta'}\circ R_{\alpha,\alpha'}$ are $\widetilde{K}_{H,1}$-intertwining operators $\fraks_{H,\CC}\otimes\calE(\alpha,\alpha')\to\calE'(\beta')$. Since $\calE'(\beta')$ occurs in $\fraks_{H,\CC}\otimes\calE(\alpha,\alpha')$ with multiplicity at most one, we have
\begin{gather}
R_{\beta,\beta'}\circ\omega_{\alpha,\alpha'}^{\beta,\beta'}(Y) = \lambda_{\alpha,\alpha'}^{\beta,\beta'}\cdot\omega_{\alpha'}^{\beta'}(Y)\circ R_{\alpha,\alpha'}\label{eq:DefLambdas}
\end{gather}
for some constant $\lambda_{\alpha,\alpha'}^{\beta,\beta'}\in\CC$.

Now, every $\widetilde{K}_{H,1}$-intertwining operator $T\colon (\pislash_\lambda)_\HC\to(\tauslash_\nu)_\HC$ is uniquely determined by its restriction to the subspaces $\calE(\alpha,\alpha')$ on which it is given by
\begin{gather}
T|_{\calE(\alpha,\alpha')} = t_{\alpha,\alpha'}\cdot R_{\alpha,\alpha'}\label{eq:IntertwinersAsScalars}
\end{gather}
for some scalars $t_{\alpha,\alpha'}\in\CC$. In~\cite{MO14a} it is proven that $T$ is $(\frakh,\widetilde{K}_{H,1})$-intertwining if and only if for all $(\alpha,\alpha')$ and $\alpha'\leftrightarrow\beta'$ we have
\begin{gather}
(2\nu+\sigma_{\beta'}'-\sigma_{\alpha'}')t_{\alpha,\alpha'} = \sum_{\substack{\beta\\(\alpha,\alpha')\leftrightarrow(\beta,\beta')}} \lambda_{\alpha,\alpha'}^{\beta,\beta'}(2\lambda+\sigma_\beta-\sigma_\alpha)t_{\beta,\beta'}.\label{eq:GeneralScalarIdentity}
\end{gather}
Here the constants $\sigma_\alpha$ and $\sigma_{\alpha'}'$ are essentially the eigenvalues of the Casimir element on the $K$-types $\calE(\alpha)$ and $\calE'(\alpha')$. From \cite[Section~3.a]{BOO96} it follows that
\begin{gather*} \sigma_\beta-\sigma_\alpha = \begin{cases}2i+n+1, & \mbox{for $(\alpha,\beta)=((i,\pm),(i+1,\pm))$ resp.~$(i,i+1)$,}\\0, & \mbox{for $(\alpha,\beta)=((i,\pm),(i,\mp))$ resp.~$(i,i)$,}\\-2i-n+1, & \mbox{for $(\alpha,\beta)=((i,\pm),(i-1,\pm))$ resp.~$(i,i-1)$,}\end{cases} \end{gather*}
and similarly for $\sigma_{\beta'}'-\sigma_{\alpha'}'$. Hence, the only missing constants in \eqref{eq:GeneralScalarIdentity} are the numbers $\lambda_{\alpha,\alpha'}^{\beta,\beta'}$ which we now compute after choosing explicit operators $R_{\alpha,\alpha'}$.

\subsection[Explicit embeddings of $K$-types]{Explicit embeddings of $\boldsymbol{K}$-types}

We realize the $K$-types $\zeta_{n+1,i}^{(\pm)}$ explicitly inside $\pislash_\lambda$.

Assume first that $n$ is even, then the $\widetilde{M}_1$-representation $\zeta_n$ is equivalent to the restriction of the $\widetilde{K}_1$-representation $\zeta_{n+1}=\zeta_{n+1}^+$ to $\widetilde{M}_1$. Hence
\begin{gather*} \pislash_\lambda|_{\widetilde{K}_1} = \Ind_{\widetilde{M}_1}^{\widetilde{K}_1}\big(\zeta_{n+1}|_{\widetilde{M}_1}\big) \simeq \Ind_{\widetilde{M}_1}^{\widetilde{K}_1}(\1)\otimes\zeta_{n+1} = C^\infty\big(\widetilde{K}_1/\widetilde{M}_1,\SS_{n+1}\big). \end{gather*}
The underlying Harish-Chandra module is in this picture given by restrictions of $\SS_{n+1}$-valued polynomials on $\RR^{n+1}$ to the sphere $\widetilde{K}_1/\widetilde{M}_1\simeq S^n\subseteq\RR^{n+1}$. Denote by $\calM_i\big(\RR^{n+1},\SS_{n+1}\big)$ the space of monogenic polynomials of degree $i$ (see Appendix~\ref{app:HigherSpinReps} for details). By the Fischer decomposition~\eqref{eq:FischerDecomposition} we have
\begin{gather*} C^\infty\big(S^n,\SS_{n+1}\big)_\HC = \bigoplus_{i=0}^\infty\big(\calM_i\big(\RR^{n+1},\SS_{n+1}\big)\oplus\underline{x}\calM_i\big(\RR^{n+1},\SS_{n+1}\big)\big), \end{gather*}
where we identify a homogeneous polynomial $\phi\colon \RR^{n+1}\to\SS_{n+1}$ with its restriction to the unit sphere $S^n\subseteq\RR^{n+1}$. Then $\calM_i\big(\RR^{n+1},\SS_{n+1}\big)$ carries the representation $\zeta_{n+1,i}^+=\zeta_{n+1,i}$ and $\underline{x}\calM_i\big(\RR^{n+1},\SS_{n+1}\big)$ carries the representation $\zeta_{n+1,i}^-=\zeta_{n+1,i}\otimes\det$. In the notation of Section~\ref{sec:SpinorKtypes} this means
\begin{gather*} \calE(\alpha) = \begin{cases}\calM_i\big(\RR^{n+1},\SS_{n+1}\big),&\mbox{for $\alpha=(i,+)$,}\\\underline{x}\calM_i\big(\RR^{n+1},\SS_{n+1}\big),&\mbox{for $\alpha=(i,-)$.}\end{cases} \end{gather*}

Now let $n$ be odd, then the restriction of the $\widetilde{K}_1$-representation $\zeta_{n+1}$ to $\widetilde{M}_1$ decomposes into $\zeta_n^+\oplus\zeta_n^-$, where $\zeta_n^+=\zeta_n$ and $\zeta_n^-=\zeta_n\otimes\det$. Hence
\begin{gather*} \pislash_\lambda|_{\widetilde{K}_1} = \Ind_{\widetilde{M}_1}^{\widetilde{K}_1}\big(\zeta_n^+\big) \subseteq \Ind_{\widetilde{M}_1}^{\widetilde{K}_1}\big(\zeta_{n+1}|_{\widetilde{M}_1}\big) \simeq C^\infty\big(\widetilde{K}_1/\widetilde{M}_1,\SS_{n+1}\big). \end{gather*}
Again, the underlying Harish-Chandra module is given by
\begin{gather*} C^\infty\big(S^n,\SS_{n+1}\big)_\HC = \bigoplus_{i=0}^\infty\big(\calM_i\big(\RR^{n+1},\SS_{n+1}\big)\oplus\underline{x}\calM_i\big(\RR^{n+1},\SS_{n+1}\big)\big) \end{gather*}
but in this case $\calM_i\big(\RR^{n+1},\SS_{n+1}\big)$ and $\underline{x}\calM_i\big(\RR^{n+1},\SS_{n+1}\big)$ are isomorphic $\widetilde{K}_1$-representations, more precisely they both are equivalent to $\zeta_{n+1,i}$. A short calculation using Lemma~\ref{lem:FundSpinBranchingOdd} shows that the $K$-types belonging to $\pislash_\lambda$ are given by
\begin{gather*} \calM_i\big(\RR^{n+1},\SS_{n+1}\big)\to\calM_i\big(\RR^{n+1},\SS_{n+1}\big)\oplus\underline{x}\calM_i\big(\RR^{n+1},\SS_{n+1}\big), \qquad \phi\mapsto\phi+\underline{x}\phi^\vee, \end{gather*}
where $\phi^\vee(x)=\gamma(\phi(x))$, $\gamma\colon \SS_{n+1}\to\SS_{n+1}$ being the map defined in~\eqref{eq:DefGamma} that intertwines $\zeta_{n+1}$ with $\zeta_{n+1}\otimes\det$. In the notation of Section~\ref{sec:SpinorKtypes} this reads
\begin{gather*} \calE(\alpha) = \big\{\phi+\underline{x}\phi^\vee\colon \phi\in\calM_i(\RR^{n+1},\SS_{n+1})\big\} \qquad \mbox{for $\alpha=i$.} \end{gather*}

Using the same identifications for the $\widetilde{K}_{H,1}$-types $\calE'(\alpha')$ we now fix for each pair $\alpha'\subset\alpha$ a~$\widetilde{K}_{H,1}$-intertwining operator
\begin{gather*} S_{\alpha,\alpha'}\colon \ \calE'(\alpha')\hookrightarrow\calE(\alpha,\alpha') \end{gather*}
as follows: For $n$ even set
\begin{gather*}\begin{split}&
S_{\alpha,\alpha'}\big(\phi+\underline{x}'\phi^\vee\big) = I_{j\to i}(\phi), \qquad \alpha=(i,+),\alpha'=j,\\
&S_{\alpha,\alpha'}\big(\phi+\underline{x}'\phi^\vee\big)= \underline{x}I_{j\to i}(\phi^\vee), \qquad \alpha=(i,-),\alpha'=j,\end{split}
\end{gather*}
and for $n$ odd we put
\begin{gather*}\begin{split}&
S_{\alpha,\alpha'}(\phi) = I_{j\to i}(\phi)+\underline{x}I_{j\to i}(\phi)^\vee, \qquad \alpha=i,\alpha'=(j,+),\\
&S_{\alpha,\alpha'}(\underline{x}'\phi)= I_{j\to i}(\phi)^\vee+\underline{x}I_{j\to i}(\phi), \qquad \alpha=i,\alpha'=(j,-).\end{split}
\end{gather*}
Then $S_{\alpha,\alpha'}$ is a $\widetilde{K}_{H,1}$-equivariant isomorphism and we can define the $\widetilde{K}_{H,1}$-intertwining operator $R_{\alpha,\alpha'}\colon \calE(\alpha,\alpha')\to\calE'(\alpha')$ by $R_{\alpha,\alpha'}=S_{\alpha,\alpha'}^{-1}$. For this particular choice of operators $R_{\alpha,\alpha'}$ we can now compute the proportionality constants $\lambda_{\alpha,\alpha'}^{\beta,\beta'}$ defined in~\eqref{eq:DefLambdas}.

\begin{Lemma}\label{lem:ExplicitConstantsLambda}Let $0\leq j\leq i$.
\begin{enumerate}\itemsep=0pt
\item[$1.$] For $n$ even and $(\alpha,\alpha')=((i,\pm),j)$ we have
\begin{gather*} \lambda_{\alpha,\alpha'}^{\beta,\beta'} = \begin{cases}\dfrac{n+2j-1}{n+2i+1},&(\beta,\beta')=((i+1,\pm),j+1),\vspace{1mm}\\
\pm\dfrac{2(n+2j-1)}{(n+2i-1)(n+2i+1)}\sqrt{-1},&(\beta,\beta')=((i,\mp),j+1),\vspace{1mm}\\
-\dfrac{n+2j-1}{n+2i-1},&(\beta,\beta')=((i-1,\pm),j+1),\vspace{1mm}\\
\mp\dfrac{i-j+1}{n+2i+1}\sqrt{-1},&(\beta,\beta')=((i+1,\pm),j),\vspace{1mm}\\
\dfrac{(n+2i)(n+2j-1)}{(n+2i-1)(n+2i+1)},&(\beta,\beta')=((i,\mp),j),\vspace{1mm}\\
\pm\dfrac{n+i+j-1}{n+2i-1}\sqrt{-1},&(\beta,\beta')=((i-1,\pm),j),\vspace{1mm}\\
-\dfrac{(i-j+1)(i-j+2)}{(n+2i+1)(n+2j-3)},&(\beta,\beta')=((i+1,\pm),j-1),\vspace{1mm}\\
\pm\dfrac{2(i-j+1)(n+i+j-1)}{(n+2i-1)(n+2i+1)(n+2j-3)}\sqrt{-1},&(\beta,\beta')=((i,\mp),j-1),\vspace{1mm}\\
\dfrac{(n+i+j-2)(n+i+j-1)}{(n+2i-1)(n+2j-3)},&(\beta,\beta')=((i-1,\pm),j-1).\end{cases} \end{gather*}
\item[$2.$] For $n$ odd and $(\alpha,\alpha')=(i,(j,\pm))$ we have
\begin{gather*} \lambda_{\alpha,\alpha'}^{\beta,\beta'} = \begin{cases}\dfrac{n+2j-1}{n+2i+1},&(\beta,\beta')=(i+1,(j+1,\pm)),\vspace{1mm}\\
\mp\dfrac{2(n+2j-1)}{(n+2i-1)(n+2i+1)},&(\beta,\beta')=(i,(j+1,\pm)),\vspace{1mm}\\
-\dfrac{n+2j-1}{n+2i-1},&(\beta,\beta')=(i-1,(j+1,\pm)),\vspace{1mm}\\
\pm\dfrac{i-j+1}{n+2i+1},&(\beta,\beta')=(i+1,(j,\mp)),\vspace{1mm}\\
\dfrac{(n+2i)(n+2j-1)}{(n+2i-1)(n+2i+1)},&(\beta,\beta')=(i,(j,\mp)),\vspace{1mm}\\
\mp\dfrac{n+i+j-1}{n+2i-1},&(\beta,\beta')=(i-1,(j,\mp)),\vspace{1mm}\\
-\dfrac{(i-j+1)(i-j+2)}{(n+2i+1)(n+2j-3)},&(\beta,\beta')=(i+1,(j-1,\pm)),\vspace{1mm}\\
\mp\dfrac{2(i-j+1)(n+i+j-1)}{(n+2i-1)(n+2i+1)(n+2j-3)},&(\beta,\beta')=(i,(j-1,\pm)),\vspace{1mm}\\
\dfrac{(n+i+j-2)(n+i+j-1)}{(n+2i-1)(n+2j-3)},&(\beta,\beta')=(i-1,(j-1,\pm)).\end{cases} \end{gather*}
\end{enumerate}
\end{Lemma}

\begin{proof}Since $S_{\alpha,\alpha'}=R_{\alpha,\alpha'}^{-1}$ the defining equation \eqref{eq:DefLambdas} for $\lambda_{\alpha,\alpha'}^{\beta,\beta'}$ is equivalent to
\begin{gather*} \omega_{\alpha,\alpha'}^{\beta,\beta'}(Y)\circ S_{\alpha,\alpha'} = \lambda_{\alpha,\alpha'}^{\beta,\beta'}\cdot S_{\beta,\beta'}\circ\omega_{\alpha'}^{\beta'}(Y), \qquad Y\in\fraks_{H,\CC}. \end{gather*}
Let first $n$ be even and $\alpha=(i,+)$, $\alpha'=j$. Then for $Y=e_k$, $1\leq k\leq n$, and $\phi+\underline{x}'\phi^\vee\in\calE'(\alpha')$, $\phi\in\calM_j(\RR^n;\SS_n)\}$, we have
\begin{align*}
 \omega(Y)\big(\phi+\underline{x}'\phi^\vee\big) & = x_k\phi+\underline{x}'(x_k\phi)^\vee = \big(\phi_k^+-\underline{x}'\phi_k^0+\phi_k^-\big)+\underline{x}'\big(\phi_k^+-\underline{x}'\phi_k^0+\phi_k^-\big)^\vee\\
& = \big(\phi_k^++\underline{x}'\big(\phi_k^+\big)^\vee\big) -\big(\big(\phi_k^0\big)^\vee+\underline{x}'\big(\phi_k^0\big)^\vee{}^\vee\big) + \big(\phi_k^-+\underline{x}'\big(\phi_k^-\big)^\vee\big).
\end{align*}
Hence,
\begin{gather*} \omega_{\alpha'}^{\beta'}(Y)\phi = \begin{cases}\phi_k^+,&\beta'=j+1,\\-\big(\phi_k^0\big)^\vee,&\beta'=j,\\\phi_k^-,&\beta'=j-1\end{cases} \end{gather*}
and
\begin{gather*} S_{\beta,\beta'}\circ\omega_{\alpha'}^{\beta'}(Y)\phi = \begin{cases}I_{j+1\to i+1}\big(\phi_k^+\big),&(\beta,\beta')=((i+1,+),j+1),\\
\underline{x}I_{j+1\to i}\big(\big(\phi_k^+\big)^\vee\big),&(\beta,\beta')=((i,-),j+1),\\
I_{j+1\to i-1}\big(\phi_k^+\big),&(\beta,\beta')=((i-1,+),j+1),\\
-I_{j\to i+1}\big(\big(\phi_k^0\big)^\vee\big),& (\beta,\beta')=((i+1,+),j),\\
-\underline{x}I_{j\to i}\big(\phi_k^0\big),&(\beta,\beta')=((i,-),j),\\
-I_{j\to i-1}\big(\big(\phi_k^0\big)^\vee\big),&(\beta,\beta')=((i-1,+),j),\\
I_{j-1\to i+1}\big(\phi_k^-\big),&(\beta,\beta')=((i+1,+),j-1),\\
\underline{x}I_{j-1\to i}\big(\big(\phi_k^-\big)^\vee\big),&(\beta,\beta')=((i,-),j-1),\\
I_{j-1\to i-1}\big(\phi_k^-\big),&(\beta,\beta')=((i-1,+),j-1).\end{cases} \end{gather*}
Note that $\phi^\vee=-\sqrt{-1}\underline{e_{n+1}}\phi$ by Lemma~\ref{lem:FundSpinBranchingEven}(2). On the other hand,
\begin{gather*} S_{\alpha,\alpha'}\phi = I_{j\to i}(\phi) \end{gather*}
and
\begin{gather*} \omega(Y)S_{\alpha,\alpha'}\phi = x_kI_{j\to i}(\phi) = (I_{j\to i}(\phi))_k^+ - \underline{x}(I_{j\to i}(\phi))_k^0 + (I_{j\to i}(\phi))_k^-. \end{gather*}
Then Lemma~\ref{lem:CocycleVsEmbedding} gives the constants $\lambda_{\alpha,\alpha'}^{\beta,\beta'}$ for $(\alpha,\alpha')=((i,+),j)$. The case $\alpha=(i,-)$ and $\alpha'=j$ is treated similarly. The verification for odd $n$ is left to the reader.
\end{proof}

\subsection[Multiplicities for symmetry breaking operators between spinor-valued principal series]{Multiplicities for symmetry breaking operators\\ between spinor-valued principal series}

Inserting the explicit constants $\lambda_{\alpha,\alpha'}^{\beta,\beta'}$ determined in Lemma~\ref{lem:ExplicitConstantsLambda} into the scalar identities \eqref{eq:GeneralScalarIdentity} we obtain an explicit characterization of symmetry breaking operators in terms of scalar identities. We assume that $n$ is even for the statement of these identities, the case of odd~$n$ is similar.

\begin{Corollary}\label{cor:ScalarIdentities} Assume $n$ is even. Then a $\widetilde{K}_{H,1}$-intertwining operator $T\colon (\pislash_\lambda)_\HC\to(\tauslash_\nu)_\HC$ is $\big(\frakh,\widetilde{K}_{H,1}\big)$-intertwining if and only if the scalars $t_{(i,\pm),j}$ defined by \eqref{eq:IntertwinersAsScalars} satisfy the following three relations:
\begin{gather*}
 (n+2i-1)(n+2i+1)\big(\nu+\rho_H+\tfrac{1}{2}+j\big)t_{(i,\pm),j}\\
 \qquad {}= (n+2i-1)(n+2j-1)\big(\lambda+\rho+\tfrac{1}{2}+i\big)t_{(i+1,\pm),j+1} \pm 2(n+2j-1)\sqrt{-1}\lambda t_{(i,\mp),j+1}\notag\\
 \qquad\quad {}- (n+2i+1)(n+2j-1)\big(\lambda-\rho+\tfrac{1}{2}-i\big)t_{(i-1,\pm),j+1},\notag\\
 (n+2i-1)(n+2i+1)\nu t_{(i,\pm),j}\\
 \qquad {}= \mp(i-j+1)(n+2i-1)\sqrt{-1}\big(\lambda+\rho+\tfrac{1}{2}+i\big)t_{(i+1,\pm),j}\notag\\
 \qquad\quad {}+ (n+2i)(n+2j-1)\lambda t_{(i,\mp),j}\notag\\
 \qquad\quad {}\pm (n+2i+1)(n+i+j-1)\sqrt{-1}\big(\lambda-\rho+\tfrac{1}{2}-i\big)t_{(i-1,\pm),j},\notag\\
 (n+2i-1)(n+2i+1)(n+2j-3)\big(\nu-\rho_H+\tfrac{1}{2}-j\big)t_{(i,\pm),j}\\
 \qquad {}= -(i-j+1)(i-j+2)(n+2i-1)\big(\lambda+\rho+\tfrac{1}{2}+i\big)t_{(i+1,\pm),j-1}\notag\\
 \qquad\quad {}\pm 2(i-j+1)(n+i+j-1)\sqrt{-1}\lambda t_{(i,\mp),j-1}\notag\\
 \qquad\quad {}+ (n+2i+1)(n+i+j-2)(n+i+j-1)\big(\lambda-\rho+\tfrac{1}{2}-i\big)t_{(i-1,\pm),j-1}.\notag
\end{gather*}
\end{Corollary}

As previously carried out in \cite[Section 4.3]{MO14a} for similar relations, we solve this system to obtain the following result about multiplicities of symmetry breaking operators:

\begin{Theorem}\label{thm:MultiplicitiesForG1H1}
\begin{gather*} \dim\Hom_{(\frakh,\widetilde{K}_{H,1})}\big((\pislash_\lambda)_\HC|_{(\frakh,\widetilde{K}_{H,1})},(\tauslash_\nu)_\HC\big)=\begin{cases}3,&\mbox{for $(\lambda,\nu)\in\Lslash=\Lslash_\even\cup\Lslash_\odd$,}\\ 2,&\mbox{else.}\end{cases} \end{gather*}
\end{Theorem}

\begin{proof}We assume $n$ is even, the case of $n$ odd is treated similarly. It is more convenient to work with the scalars
\begin{gather*} s_{i,j}^\pm = t_{(i,+),j} \pm (-1)^{i-j}t_{(i,-),j}, \end{gather*}
then the identities in Corollary~\ref{cor:ScalarIdentities} become
\begin{gather}
 (n+2i-1)(n+2i+1)\big(\nu+\rho_H+\tfrac{1}{2}+j\big)s_{i,j}^\pm\notag\\
 \qquad {}= (n+2i-1)(n+2j-1)\big(\lambda+\rho+\tfrac{1}{2}+i\big)s_{i+1,j+1}^\pm\notag\\
 \qquad\quad {}\pm(-1)^{i-j+1}2(n+2j-1)\sqrt{-1}\lambda s_{i,j+1}^\pm\notag\\
 \qquad\quad {}- (n+2i+1)(n+2j-1)\big(\lambda-\rho+\tfrac{1}{2}-i\big)s_{i-1,j+1}^\pm,\label{eq:ScalarIdentity1}\\
 \big((n+2i-1)(n+2i+1)\nu \mp(-1)^{i-j} (n+2i)(n+2j-1)\lambda\big)s_{i,j}^\pm\notag\\
 \qquad{} = -(i-j+1)(n+2i-1)\sqrt{-1}\big(\lambda+\rho+\tfrac{1}{2}+i\big)s_{i+1,j}^\pm\notag\\
 \qquad\quad {}+ (n+2i+1)(n+i+j-1)\sqrt{-1}\big(\lambda-\rho+\tfrac{1}{2}-i\big)s_{i-1,j}^\pm,\label{eq:ScalarIdentity2}\\
 (n+2i-1)(n+2i+1)(n+2j-3)\big(\nu-\rho_H+\tfrac{1}{2}-j\big)s_{i,j}^\pm\notag\\
 \qquad {}= -(i-j+1)(i-j+2)(n+2i-1)\big(\lambda+\rho+\tfrac{1}{2}+i\big)s_{i+1,j-1}^\pm\notag\\
 \qquad\quad {} \pm(-1)^{i-j+1}2(i-j+1)(n+i+j-1)\sqrt{-1}\lambda s_{i,j-1}^\pm\notag\\
 \qquad\quad {}+ (n+2i+1)(n+i+j-2)(n+i+j-1)\big(\lambda-\rho+\tfrac{1}{2}-i\big)s_{i-1,j-1}^\pm.\label{eq:ScalarIdentity3}
\end{gather}
Note that each identity only involves scalars $s_{i,j}^\pm$ with one choice of sign $\pm$, so that the system of equations degenerates into two systems of equations, one for $s_{i,j}^+$ and one for $s_{i,j}^-$. Fixing a~sign~$\pm$ we have to find the dimension of the space of tuples $(s_{i,j}^\pm)_{0\leq j\leq i}$ satisfying the identities \eqref{eq:ScalarIdentity1}--\eqref{eq:ScalarIdentity3}. Let us first visualize the $\widetilde{K}_1$-$\widetilde{K}_{H,1}$-type picture in a diagram:
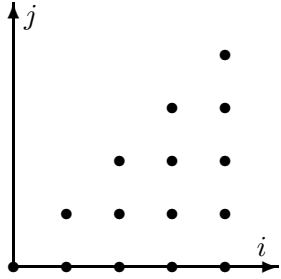
\setlength{\unitlength}{4pt}
\begin{center}\begin{picture}(26,26)
\thicklines
\put(0,0){\vector(1,0){25}}
\put(0,0){\vector(0,1){25}}

\multiput(0,0)(10,0){3}{\circle*{1}}
\multiput(5,5)(10,0){2}{\circle*{1}}
\multiput(10,10)(10,0){2}{\circle*{1}}
\multiput(15,15)(10,0){1}{\circle*{1}}
\multiput(20,20)(10,0){1}{\circle*{1}}

\multiput(5,0)(10,0){2}{\circle*{1}}
\multiput(10,5)(10,0){2}{\circle*{1}}
\multiput(15,10)(10,0){1}{\circle*{1}}
\multiput(20,15)(10,0){1}{\circle*{1}}
\put(23,1){$i$}
\put(1,23){$j$}
\end{picture}\end{center}
Here the dot at position $(i,j)$ represents the scalar $s_{i,j}^\pm$. Now, each identity is a linear relation between certain neighboring dots in this diagram, visualized as
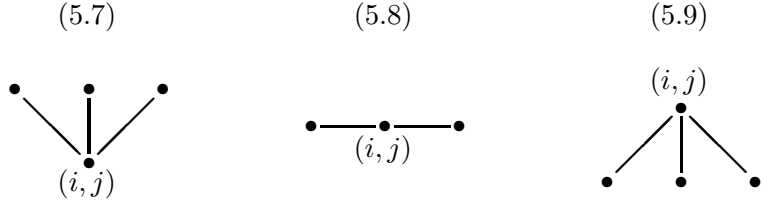
\setlength{\unitlength}{3.5pt}
\begin{center}\begin{picture}(80,19)
\thicklines
\put(0,10){\circle*{1}}
\put(8,2){\circle*{1}}
\put(8,10){\circle*{1}}
\put(16,10){\circle*{1}}
\put(1,9){\line(1,-1){6}}
\put(15,9){\line(-1,-1){6}}
\put(8,3){\line(0,1){6}}
\put(4.6,-1){$(i,j)$}
\put(4.6,17){\eqref{eq:ScalarIdentity1}}

\put(32,6){\circle*{1}}
\put(40,6){\circle*{1}}
\put(48,6){\circle*{1}}
\put(33,6){\line(1,0){6}}
\put(41,6){\line(1,0){6}}
\put(36.6,2.6){$(i,j)$}
\put(36.6,17){\eqref{eq:ScalarIdentity2}}

\put(64,0){\circle*{1}}
\put(72,8){\circle*{1}}
\put(72,0){\circle*{1}}
\put(80,0){\circle*{1}}
\put(65,1){\line(1,1){6}}
\put(79,1){\line(-1,1){6}}
\put(72,1){\line(0,1){6}}
\put(68.6,10){$(i,j)$}
\put(68.6,17){\eqref{eq:ScalarIdentity3}}
\end{picture}\end{center}
Note that the only coefficients in the identities that can possibly vanish are those involving $\lambda$ and $\nu$. Identity \eqref{eq:ScalarIdentity1} at $(i,i)$ gives
\begin{gather}
(n+2j-1)\big(\lambda+\rho+\tfrac{1}{2}+i\big)s_{i+1,i+1}^\pm = (n+2i+1)\big(\nu+\rho_H+\tfrac{1}{2}+i\big)s_{i,i}^\pm.\label{eq:ScalarIdentityDiagonal}
\end{gather}
It is easy to see that the dimension of the space of diagonal sequences $(s_{i,i}^\pm)$ satisfying~\eqref{eq:ScalarIdentityDiagonal} is equal to $1$ if $(\lambda,\nu)\in\CC^2\setminus\Lslash$ and equal to $2$ if $(\lambda,\nu)\in\Lslash$ (see \cite[Lemma~4.4]{MO14a} for the same argument in a similar setting).

\textbf{Step 1.} Now let us first assume that $\lambda\notin-\rho-\frac{1}{2}-\NN$, in particular $(\lambda,\nu)\in\CC^2\setminus\Lslash$. Then identity~\eqref{eq:ScalarIdentity2} can be used to define $s_{i+1,j}^\pm$ in terms of $s_{i,j}^\pm$ and $s_{i-1,j}^\pm$ since the coefficient $\big(\lambda+\rho+\frac{1}{2}+i\big)$ of $s_{i+1,j}^\pm$ never vanishes. This shows that each diagonal sequence $(s_{i,i}^\pm)$ uniquely determines the numbers $s_{i,j}^\pm$ for all $0\leq j\leq i$ and hence, for both signs $\pm$ the dimension of the space of sequences~$\big(s_{i,j}^\pm\big)$ satisfying \eqref{eq:ScalarIdentity1}--\eqref{eq:ScalarIdentity3} is $1$.

\textbf{Step 2.} Now assume $\lambda=-\rho-\frac{1}{2}-k$, $k\geq0$, but $(\lambda,\nu)\in\CC^2\setminus\Lslash$. Then the extension argument from Step~1 using \eqref{eq:ScalarIdentity2} still works for $i>k$ and hence any diagonal sequence uniquely determines the numbers $s_{i,j}^\pm$ for $j>k$. Next, we can repeatedly use \eqref{eq:ScalarIdentity1} with $j\leq k$ to define~$s_{i,j}^\pm$ in terms of~$s_{i+1,j+1}^\pm$, $s_{i,j+1}^\pm$ and $s_{i-1,j+1}^\pm$. Note that the coefficient $\big(\nu+\rho_H+\frac{1}{2}+j\big)$ of $s_{i,j}^\pm$ never vanishes since $j\leq k$ and $(\lambda,\nu)\notin\Lslash$. Therefore, also in this case, for both signs $\pm$ the space of sequences $(s_{i,j}^\pm)$ satisfying \eqref{eq:ScalarIdentity1}--\eqref{eq:ScalarIdentity3} is one-dimensional.

\textbf{Step 3.} Finally assume $(\lambda,\nu)=\big({-}\rho-\frac{1}{2}-k,-\rho_H-\frac{1}{2}-\ell\big)\in\Lslash$, $0\leq\ell\leq k$. We only discuss the case where $k-\ell\in2\NN$, the case $k-\ell\in2\NN+1$ is treated similarly. We divide the $\widetilde{K}_1$--$\widetilde{K}_{H,1}$-type picture into four regions accoding to whether $i\leq k$ or $i>k$ and $j\leq\ell$ or $j>\ell$:
\setlength{\unitlength}{4pt}
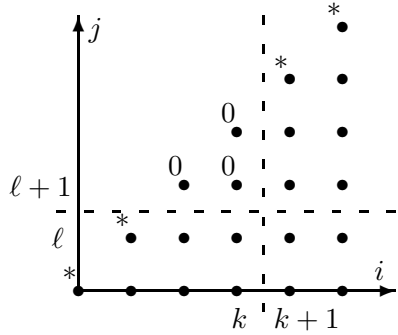
\begin{center}\begin{picture}(41,30)
\thicklines
\put(0,2){\vector(1,0){30}}
\put(0,2){\vector(0,1){26}}
\dashline{1}(-2,9.5)(30,9.5)
\dashline{1}(17.5,0)(17.5,28)
\multiput(0,2)(5,5){6}{\circle*{1}}
\multiput(5,2)(5,5){5}{\circle*{1}}
\multiput(10,2)(5,5){4}{\circle*{1}}
\multiput(15,2)(5,5){3}{\circle*{1}}
\multiput(20,2)(5,5){2}{\circle*{1}}
\put(25,2){\circle*{1}}
\put(-1.5,3){$*$}
\put(3.5,8){$*$}
\put(8.5,13){$0$}
\put(13.5,13){$0$}
\put(13.5,18){$0$}
\put(18.5,23){$*$}
\put(23.5,28){$*$}
\put(28,3){$i$}
\put(1,26){$j$}
\put(14.5,-1.5){$k$}
\put(18.5,-1.5){$k+1$}
\put(-2.5,6){$\ell$}
\put(-6.5,11){$\ell+1$}
\end{picture}\end{center}
Then from the diagonal identity \eqref{eq:ScalarIdentityDiagonal} it follows that the diagonal entries in the upper left region are $0$ whereas the diagonal entries in the upper right and the lower left region can be chosen independently (as indicated by the stars and zeros). This shows that the space of diagonal sequences satisfying~\eqref{eq:ScalarIdentityDiagonal} is two-dimensional. Using~\eqref{eq:ScalarIdentity2} as before shows that in fact all scalars in the upper left region have to be~$0$. Now we study what happens near the intersection of the two dashed lines. For instance, we have two identities that relate $s_{k,\ell}^\pm$ and $s_{k-1,\ell}^\pm$, namely~\eqref{eq:ScalarIdentity2} for $(i,j)=(k,\ell)$:
\begin{gather}
\big((n+2k-1)(n+2\ell)\mp(n+2k)(n+2\ell-1)\big)s_{k,\ell}^\pm\nonumber\\
\qquad {} = 2(n+2k)(n+k+\ell-1)\sqrt{-1}s_{k-1,\ell}^\pm,\label{eq:ScalarIdentityDepIndep1}
\end{gather}
and \eqref{eq:ScalarIdentity3} for $(i,j)=(k,\ell+1)$:
\begin{gather}
\pm(k-\ell)s_{k,\ell}^\pm = (n+2k)(n+k+\ell-1)\sqrt{-1}s_{k-1,\ell}^\pm.\label{eq:ScalarIdentityDepIndep2}
\end{gather}
Let us first consider the positive sign scalars $s_{i,j}^+$. In this case the coefficient of $s_{k,\ell}^+$ in \eqref{eq:ScalarIdentityDepIndep1} is
\begin{gather*} (n+2k-1)(n+2\ell)-(n+2k)(n+2\ell-1) = 2(k-\ell) \end{gather*}
so that \eqref{eq:ScalarIdentityDepIndep1} is a multiple of \eqref{eq:ScalarIdentityDepIndep2}. Hence, the two identities \eqref{eq:ScalarIdentityDepIndep1} and~\eqref{eq:ScalarIdentityDepIndep2} are dependent and do not force $s_{k,\ell}^\pm$ and $s_{k-1,\ell}^\pm$ to be $0$. As above one can uniquely extend any diagonal sequence to the whole lower left region $s_{i,j}^+$ with $i\leq k$, $j\leq\ell$. For the negative sign scalars $s_{i,j}^-$ the opposite is true: the relations~\eqref{eq:ScalarIdentityDepIndep1} and \eqref{eq:ScalarIdentityDepIndep2} are independent and hence $s_{k-1,\ell}^-=s_{k,\ell}^-=0$. Extending the zeros in the other direction forces all scalars $s_{i,j}^-$ in the lower left region to be $0$. The same happens for the upper right region, where~\eqref{eq:ScalarIdentity1} for $(i,j)=(k+1,\ell)$ and \eqref{eq:ScalarIdentity2} for $(i,j)=(k+1,\ell+1)$ are two dependent resp. independent relations for $s_{k+1,\ell+1}^+$ and $s_{k+2,\ell+1}^+$, resp.~$s_{k+1,\ell+1}^-$ and~$s_{k+2,\ell+1}^-$. This shows that the space of scalars in the upper left, lower left and upper right region satisfying~\eqref{eq:ScalarIdentity1}--\eqref{eq:ScalarIdentity3} is two-dimensional for $s_{i,j}^+$ and trivial for $s_{i,j}^-$. It remains to extend such tuples to the lower right region. Assuming that we already have defined~$s_{i,j}^\pm$ in the first three regions, there are two linear relations for $s_{k+1,\ell}^\pm$ and $s_{k+2,\ell}^\pm$, namely~\eqref{eq:ScalarIdentity2} for $(i,j)=(k+1,\ell)$:
\begin{gather}
\big((n+2k+3)(n+2\ell)\pm(n+2k+2)(n+2\ell-1)\big)s_{k+1,\ell}^\pm\nonumber\\
\qquad{}- 2(k-\ell+2)\sqrt{-1}s_{k+2,\ell}^\pm = *,\label{eq:ScalarIdentityDepIndep3}
\end{gather}
and \eqref{eq:ScalarIdentity3} for $(i,j)=(k+1,\ell+1)$:
\begin{gather}
\mp(n+k+\ell+1)s_{k+1,\ell}^\pm - (k-\ell+2)\sqrt{-1}s_{k+2,\ell}^\pm = *.\label{eq:ScalarIdentityDepIndep4}
\end{gather}
Here we write $*$ for the right hand side which is a linear combination of scalars from the other three regions which we already specified. Now, for the negative sign scalars the coefficient of~$s_{k+1,\ell}^-$ in~\eqref{eq:ScalarIdentityDepIndep3} is equal to
\begin{gather*} (n+2k+3)(n+2\ell)-(n+2k+2)(n+2\ell-1) = 2(n+k+\ell+1) \end{gather*}
and the two relations are dependent. This means that every choice of $s_{k+1,\ell}^-\in\CC$ can be uniquely extended to $s_{i,j}^-$, $0\leq j\leq i$, using the extension methods from~(1) and~(2) (with $s_{i,j}^-=0$ if $i\leq k$ or $j>\ell$). Hence, the dimension of the space of tuples $(s_{i,j}^-)$ satisfying \eqref{eq:ScalarIdentity1}--\eqref{eq:ScalarIdentity3} is $1$. For the positive sign scalars, it is easy to see that the relations~\eqref{eq:ScalarIdentityDepIndep3} and~\eqref{eq:ScalarIdentityDepIndep4} are independent, and therefore $s_{k+1,\ell}^+$ and $s_{k+2,\ell}^+$ are uniquely determined by the chosen scalars in the other three regions. Together with the extension techniques outlined in Steps~1 and~2 this implies that the dimension of the space of tuples $(s_{i,j}^+)$ satisfying \eqref{eq:ScalarIdentity1}--\eqref{eq:ScalarIdentity3} is $2$. This proves the claim.
\end{proof}

\subsection[Classification of symmetry breaking operators between spinor-valued principal series]{Classification of symmetry breaking operators\\ between spinor-valued principal series}

Combining Theorems~\ref{thm:SBOsSpinors+}, \ref{thm:SBOsSpinors-} and Proposition~\ref{prop:ProjOntoIrrSummands} with Lemma~\ref{lem:ReductionToG1H1} and Theorem~\ref{thm:MultiplicitiesForG1H1} we obtain a full classification of symmetry breaking operators between spinor-valued principal series representations:

\begin{Theorem}\label{thm:ClassificationSBOsSpinors}
We have
\begin{gather*}
\calD'\big(\RR^n;\Hom_\CC(\SS_n,\SS_{n-1})\big)_{\lambda,\nu}^+ =
\begin{cases}
\CC\big(P\widetilde{\Kslash}^{\AA,+}_{\lambda,\nu}\big),&\mbox{for $(\lambda,\nu)\in\CC^2\setminus\Lslash_\even$,}\\
\CC\big(P\doublewidetilde{\Kslash}^{\AA,+}_{\lambda,\nu}\big)\oplus\CC\big(P\widetilde{\Kslash}^{\CC,+}_{\lambda,\nu}\big), &\mbox{for $(\lambda,\nu)\in\Lslash_\even$,}
\end{cases}\\
\calD'\big(\RR^n;\Hom_\CC(\SS_n,\SS_{n-1})\big)_{\lambda,\nu}^- =
\begin{cases}
\CC\big(P\widetilde{\Kslash}^{\AA,-}_{\lambda,\nu}\big),&\mbox{for $(\lambda,\nu)\in\CC^2\setminus\Lslash_\odd$,}\\
\CC\big(P\doublewidetilde{\Kslash}^{\AA,-}_{\lambda,\nu}\big)\oplus\CC\big(P\widetilde{\Kslash}^{\CC,-}_{\lambda,\nu}\big),&\mbox{for $(\lambda,\nu)\in\Lslash_\odd$.}
\end{cases}
\end{gather*}
\end{Theorem}

\subsection{Symmetry breaking operators at reducibility points}

We study symmetry breaking operators between irreducible constituents of $\slashed{\pi}_{\lambda,\delta}$ and $\slashed{\tau}_{\nu,\varepsilon}$ at reducibility points. For this we first describe the irreducible constituents.

\begin{Lemma}\label{lem:SpinReducibility}\quad
\begin{enumerate}\itemsep=0pt
\item[$1.$] The representation $\slashed{\pi}_{\lambda,\delta}$ is reducible if and only if $\lambda\in\pm(\rho+\frac{1}{2}+\NN)$.
\item[$2.$] For $\lambda=-\rho-\frac{1}{2}-i$, $i\in\NN$, the representation $\slashed{\pi}_{\lambda,\delta}$ has a unique irreducible subrepresenta\-tion~$\calF_\delta(i)$ which is finite-dimensional, and the quotient $\calT_\delta(i)=\slashed{\pi}_{\lambda,\delta}/\calF_\delta(i)$ is irreducible.
\item[$3.$] For $\lambda=\rho+\frac{1}{2}+i$, $i\in\NN$, the representation $\slashed{\pi}_{\lambda,\delta}$ has a unique irreducible subrepresentation isomorphic to $\calT_{1-\delta}(i)\otimes\det$ an irreducible quotient isomorphic to $\calF_{1-\delta}(i)\otimes\det$, where $\det\colon \Pin(n+1,1)\to\upO(n+1,1)\to\{\pm1\}$ is the determinant character.
\end{enumerate}
\end{Lemma}

\begin{proof}We first describe the $K$-type decomposition of $\slashed{\pi}_{\lambda,\delta}$. The maximal compact subgroup $\widetilde{K}$ of $\widetilde{G}$ is a semidirect product of $\Pin(n+1)$ with the two-element group which acts on $\Pin(n+1)$ via the canonical automorphism $\alpha$ of the Clifford algebra $\Cl(n+1)$ (see Appendix~\ref{app:CliffordAlgebrasPinGroups} for details). As remarked in Section~\ref{sec:SpinorKtypes}, the restriction of $\slashed{\pi}_{\lambda,\delta}$ to $\Pin(n+1)\simeq\widetilde{K}_1\subseteq\widetilde{K}$ decomposes into a multiplicity-free direct sum of higher spin representations. For $n$ even, it is the direct sum of $\zeta_{n+1,i}^\pm$, $i\geq0$, and it is easy to see that $\zeta_{n+1,i}^+\oplus\zeta_{n+1,i}^-$ extends uniquely to an irreducible representation $\widetilde{\zeta}_{n+1,i}$ of $\widetilde{K}$. For $n$ odd, the restriction of $\slashed{\pi}_{\lambda,\delta}$ to $\widetilde{K}_1$ is the direct sum of $\zeta_{n+1,i}$, $i\geq0$, and there are two inequivalent ways of extending $\zeta_{n+1,i}$ to an irreducible $\widetilde{K}$-representation. For fixed $\delta\in\ZZ/2\ZZ$ let $\widetilde{\zeta}_{n+1,i}$ denote the extension which occurs in $\slashed{\pi}_{\lambda,\delta}$. Then, for both even and odd $n$ we have
\begin{gather*} (\slashed{\pi}_{\lambda,\delta})_\HC = \bigoplus_{i=0}^\infty\widetilde{\zeta}_{n+1,i}. \end{gather*}
The Lie algebra action of $\frakg$ maps $\widetilde{\zeta}_{n+1,i}$ into the direct sum of $\widetilde{\zeta}_{n+1,i+1}$, $\widetilde{\zeta}_{n+1,i}$ and $\widetilde{\zeta}_{n+1,i-1}$. From \cite{BOO96} and the computations in Section~\ref{sec:ScalarIdentities} it follows that one can reach $\widetilde{\zeta}_{n+1,i+1}$ if and only if $2\lambda+2i+n+1\neq0$, and one can reach $\widetilde{\zeta}_{n+1,i-1}$ if and only if $2\lambda-2i-n+1\neq0$. Then statements~(1) and~(2) follow with
\begin{gather*}
\calT_\delta(i)_\HC \simeq \bigoplus_{k=i+1}^\infty\widetilde{\zeta}_{n+1,k} \qquad \mbox{and} \qquad \calF_\delta(i)_\HC \simeq \bigoplus_{k=0}^i\widetilde{\zeta}_{n+1,k}.
\end{gather*}
To prove (3) observe that the standard intertwining operators in this situation (see, e.g., Remark~\ref{rem:ClassicalKnappStein})
\begin{gather*} \Ind_{\widetilde{P}}^{\widetilde{G}}\big(\big(\zeta_n\otimes\sgn^\delta\big)\otimes e^\lambda\otimes\1\big) \to \Ind_{\widetilde{P}}^{\widetilde{G}}\big(\big([\zeta_n\otimes\det]\otimes\sgn^{1-\delta}\big)\otimes e^{-\lambda}\otimes\1\big) \end{gather*}
map irreducible quotients to irreducible subrepresentations. Since the character $\det$ of $\widetilde{M}$ extends to $\widetilde{G}$ we have
\begin{gather*} \Ind_{\widetilde{P}}^{\widetilde{G}}\big(\big([\zeta_n\otimes\det]\otimes\sgn^{1-\delta}\big)\otimes e^{-\lambda}\otimes\1\big) \simeq \Ind_{\widetilde{P}}^{\widetilde{G}}\big(\big(\zeta_n\otimes\sgn^{1-\delta}\big)\otimes e^{-\lambda}\otimes\1\big)\otimes\det, \end{gather*}
and the claim follows by specializing to $\lambda=\pm\big(\rho+\frac{1}{2}+i\big)$.
\end{proof}

Note that the representations $\calF_\delta(i)$ and $\calT_\delta(i)$ depend on the chosen spin-representation~$\zeta$ of~$\widetilde{M}$ that $\slashed{\pi}_{\lambda,\delta}$ is induced from. However, as in the case of the full principal series, the multiplicities of intertwining operators turn out to be independent of $\zeta$.

Denote by $\calT'_\varepsilon(j)$ and $\calF'_\varepsilon(j)$ the corresponding composition factors of $\slashed{\tau}_{\nu,\varepsilon}$ at $\nu=\pm\big(\rho_H+\frac{1}{2}+j\big)$, $j\in\NN$.

\begin{Theorem}\label{thm:MultiplicitiesCompositionFactors} For the representations $\pi\in\{\calT_\delta(i),\calF_\delta(i)\}$ and $\tau\in\{\calT'_\varepsilon(j),\calF'_\varepsilon(j)\}$ the multipli\-cities $\dim\Hom_{\widetilde{H}}(\pi|_{\widetilde{H}},\tau)$ are given by
\begin{center}
\begin{tabular}{c|cc}
\diagonal{.2em}{.83cm}{$\pi$}{$\tau$} & $\calF'_\varepsilon(j)$ & $\calT'_\varepsilon(j)$ \\
\hline
$\calF_\delta(i)$ & $1$ & $0$\\
$\calT_\delta(i)$ & $0$ & $1$\\
\multicolumn{3}{c}{for $0\leq j\leq i$,}\\
\multicolumn{3}{c}{$i+j\equiv\delta+\varepsilon(2)$,}
\end{tabular}
\qquad
\begin{tabular}{c|cc}
\diagonal{.2em}{.83cm}{$\pi$}{$\tau$} & $\calF'_\varepsilon(j)$ & $\calT'_\varepsilon(j)$ \\
\hline
$\calF_\delta(i)$ & $0$ & $0$\\
$\calT_\delta(i)$ & $1$ & $0$\\
\multicolumn{3}{c}{otherwise.}\\
\multicolumn{3}{c}{\phantom{otherwise.}}
\end{tabular}
\end{center}
\end{Theorem}

\begin{proof}We only treat the case of even $n$, for odd $n$ similar arguments can be used. Write
\begin{gather*} (\slashed{\pi}_{\lambda,\delta})_\HC \simeq \bigoplus_{i=0}^\infty\widetilde{\zeta}_{n+1,i} \qquad \mbox{and} \qquad (\slashed{\tau}_{\nu,\varepsilon})_\HC \simeq \bigoplus_{j=0}^\infty\widetilde{\zeta}_{n,j} \end{gather*}
for the decomposition into irreducible representations of $\widetilde{K}$ and $\widetilde{K}_H$. Then, as in the proof of Theorem~\ref{thm:MultiplicitiesForG1H1}, an intertwining operator $(\slashed{\pi}_{\lambda,\delta})_\HC\to(\slashed{\tau}_{\nu,\varepsilon})_\HC$ is given by a sequence $(s_{i,j})_{0\leq j\leq i}$ of scalars, describing the action of the operator between $\widetilde{\zeta}_{n+1,i}$ and $\widetilde{\zeta}_{n,j}$. Analyzing the action of $\widetilde{H}/\widetilde{H}_1$ it is easy to see that intertwining operators are described by the scalars $s_{i,j}=s_{i,j}^+$ if $\delta+\varepsilon\equiv0(2)$ and by the scalars $s_{i,j}=s_{i,j}^-$ if $\delta+\varepsilon\equiv1(2)$. We can therefore use the identi\-ties~\eqref{eq:ScalarIdentity1},~\eqref{eq:ScalarIdentity2} and~\eqref{eq:ScalarIdentity3} for classification.

Next, observe that all intertwining operators between the irreducible constituents can be obtained from intertwining operators between the full principal series. For instance, intertwining operators $\calT_\delta(i)\to\calT'_\varepsilon(j)$ are intertwining operators $\slashed{\pi}_{\lambda,\delta}\to\slashed{\tau}_{\nu,1-\varepsilon}\otimes\det$ for $\lambda=-\rho-\frac{1}{2}-i$ and $\nu=\rho_H+\frac{1}{2}+j$ which vanish on the finite-dimensional subrepresentation of $\slashed{\pi}_{\lambda,\delta}$ (hence factor to the quotient $\calT_\delta(i)$) and whose image is contained in the irreducible subrepresentation $\calT'_\varepsilon(j)\subseteq\slashed{\tau}_{\nu,1-\varepsilon}\otimes\det$. By replacing the spin representation $\zeta_{n-1}$ of $\widetilde{M}_H$, that $\slashed{\tau}_{\nu,1-\varepsilon}$ is induced from, by $\zeta_{n-1}\otimes\det$ we may as well consider intertwining operators $\slashed{\pi}_{\lambda,\delta}\to\slashed{\tau}_{\nu,1-\varepsilon}$. Then the multiplicity $\dim\Hom_{(\frakh,\widetilde{K}_H)}(\calT_\delta(i)_\HC,\calT'_\varepsilon(j))$ is the dimension of the space of sequences $(s_{k,\ell})_{0\leq\ell\leq k}$ satisfying \eqref{eq:ScalarIdentity1}, \eqref{eq:ScalarIdentity2} and \eqref{eq:ScalarIdentity3} (with $+$ for $\delta+(1-\varepsilon)\equiv0(2)$ and $-$ for $\delta+(1-\varepsilon)\equiv1(2)$) such that $s_{k,\ell}=0$ whenever $k\leq i$ or $\ell\leq j$. By Theorem~\ref{thm:ClassificationSBOsSpinors} there is up to scalar multiples a unique intertwiner $\slashed{\pi}_{\lambda,\delta}\to\slashed{\tau}_{\nu,1-\varepsilon}$, and the arguments in the proof of Theorem~\ref{thm:MultiplicitiesForG1H1} show that for this intertwiner we have $s_{k,\ell}=0$ whenever $k\leq i$ or $\ell\leq j$ if and only if $j\leq i$ and $i+j\equiv\delta+\varepsilon(2)$. The other multiplicities are computed similarly.
\end{proof}

\appendix

\section{Clifford algebras, pin groups and their representations}\label{app:CliffordSpin}

We recall the basic definitions for Clifford algebras, pin groups and their representations. Most of the results are well-known or follow easily from the standard literature.

\subsection{Clifford algebras and pin groups}\label{app:CliffordAlgebrasPinGroups}

For $p,q\geq0$ we let $\RR^{p,q}=\big(\RR^{p+q},Q_{p,q}\big)$, where $Q=Q_{p,q}$ is the following quadratic form on $\RR^{p+q}$:
\begin{gather*} Q(x) = -x_1^2-\cdots-x_p^2+x_{p+1}^2+\cdots+x_{p+q}^2. \end{gather*}
Abusing notation, we also write $Q(x,y)$ for the associated symmetric bilinear form on $\RR^{p+q}$.

We define the Clifford algebra $\Cl(p,q)$ to be the unital $\RR$-algebra generated by $\RR^{p,q}$ subject to the relation
\begin{gather*} xy+yx = 2Q(x,y)\1. \end{gather*}
For the standard basis vectors $e_1,\ldots,e_{p+q}\in\RR^{p,q}$ this implies
\begin{gather*} e_i^2 = Q(e_i)\1, \qquad e_ie_j+e_je_i = 0, \qquad i\neq j. \end{gather*}
For $(p,q)=(n,0)$ we also write $\Cl(n)=\Cl(n,0)$ for short. The Clifford algebra $\Cl(p,q)$ has a~natural grading into even and odd elements
\begin{gather*} \Cl(p,q) = \Cl(p,q)_\even \oplus \Cl(p,q)_\odd. \end{gather*}
The map $\alpha$ of $\Cl(p,q)$ which acts by $+1$ on $\Cl(p,q)_\even$ and by $-1$ on $\Cl(p,q)_\odd$ is an algebra involution called the \textit{canonical automorphism}.

Denote by $\Cl(p,q;\CC)=\Cl(p,q)\otimes_\RR\CC$ and $\Cl(n;\CC)=\Cl(n)\otimes_\RR\CC$ the complexifications of~$\Cl(p,q)$ and~$\Cl(n)$. Abusing notation, we will also use $Q$ for the extension of the symmetric $\RR$-bilinear form $Q_{p,q}$ on $\RR^{p+q}$ to a symmetric $\CC$-bilinear form on $\CC^{p+q}$. Note that $\Cl(p,q;\CC)\simeq\Cl(p+q;\CC)$ as $\CC$-algebras.

We define the groups $\Pin(p,q)$ and $\Spin(p,q)$ by
\begin{gather*}
\Pin(p,q) = \big\{v_1\cdots v_k\colon k\geq0,v_i\in\RR^{p,q},Q(v_i)=\pm1\big\} \subseteq \Cl(p,q),\\
\Spin(p,q) = \Pin(p,q)\cap\Cl(p,q)_\even.
\end{gather*}
Further, put $\Pin(n)=\Pin(n,0)$ and $\Spin(n)=\Spin(n,0)$. Then $\Spin(n)$ is connected and $\Pin(n)=\Spin(n)\cup\Pin(n)_-$, where $\Pin(n)_-=\Spin(n)\cdot e_1$. For $p,q>0$ the group $\Pin(p,q)$ has four connected components
\begin{gather*} \Pin(p,q)_{++},\quad \Pin(p,q)_{+-},\quad \Pin(p,q)_{-+}\quad \mbox{and} \quad \Pin(p,q)_{--}, \end{gather*}
where the first~$+$ (resp.~$-$) means that the number of $v_i$'s with $Q(v_i)=1$ in the product of an element $x=v_1\cdots v_k\in\Pin(p,q)$ is even (resp.~odd), and the second $+$ (resp.~$-$) the same for the number of $v_i$'s with $Q(v_i)=-1$. Then clearly $\Spin(p,q)=\Pin(p,q)_{++}\cup\Pin(p,q)_{--}$, so $\Spin(p,q)$ has two connected components.

For all $p,q\geq0$ the group $\Pin(p,q)$ is a double cover of the indefinite orthogonal group $\upO(p,q)$, the covering map being
\begin{gather*} q\colon \ \Pin(p,q)\to\upO(p,q), \qquad q(x)y = \alpha(x)yx^{-1}. \end{gather*}
The restriction of $q$ to $\Spin(p,q)$ induces a double covering
\begin{gather*} q\colon \ \Spin(p,q)\to\SO(p,q). \end{gather*}

\subsection{Clifford modules and fundamental spin representations}\label{app:CliffordModules}

We now describe the irreducible representations of $\Cl(n;\CC)$. For $n=2m$ even there is only one irreducible representation, and for $n=2m+1$ odd there are two. To construct these we choose the maximal isotropic subspaces
\begin{gather*} W = \CC w_1\oplus\cdots\oplus\CC w_m \qquad \mbox{and} \qquad W' = \CC w_1'\oplus\cdots\oplus\CC w_m' \end{gather*}
of $\CC^n$, where $w_i=\frac{1}{2}\big(e_{2i-1}+\sqrt{-1}e_{2i}\big)$, $w_i'=\frac{1}{2}\big(e_{2i-1}-\sqrt{-1}e_{2i}\big)$. Then
\begin{gather*} \CC^n = \begin{cases}W\oplus W',&\mbox{for $n$ even},\\W\oplus W'\oplus\CC e_{2m+1},&\mbox{for $n$ odd.}\end{cases} \end{gather*}
We define an irreducible representation $\zeta_n$ of $\Cl(n;\CC)$ on $\SS_n=\bigwedge^\bullet W$ by
\begin{gather*}
\zeta_n(w)\xi_1\wedge\dots\wedge\xi_k = w\wedge\xi_1\wedge\dots\wedge\xi_k,\\
\zeta_n(w')\xi_1\wedge\dots\wedge\xi_k = 2\sum_{i=1}^k(-1)^{i-1}Q(w',\xi_i)\xi_1\wedge\dots\wedge\widehat{\xi_i}\wedge\dots\wedge\xi_k,
\end{gather*}
for $w\in W$, $w'\in W'$, and in case $n$ is odd additionally
\begin{gather*}
\zeta_n(e_{2m+1})\xi_1\wedge\dots\wedge\xi_k = (-1)^k\sqrt{-1}\xi_1\wedge\dots\wedge\xi_k.
\end{gather*}
Then for $n$ even $(\zeta_n,\SS_n)$ is the unique irreducible representation of $\Cl(n;\CC)$ and $\zeta_n\circ\alpha\simeq\zeta_n$, where $\alpha$ is the canonical automorphism of $\Cl(n;\CC)$. More precisely, the map
\begin{gather}
\gamma\colon \ \SS_n\to\SS_n, \qquad \gamma(\xi_1\wedge\dots\wedge\xi_k)=(-1)^k\xi_1\wedge\dots\wedge\xi_k\label{eq:DefGamma}
\end{gather}
intertwines $\zeta_n\circ\alpha$ and $\zeta_n$. For $n$ odd, $\zeta_n^+=\zeta_n$ and $\zeta_n^-=\zeta_n\circ\alpha$ are inequivalent and we obtain two irreducible inequivalent representations $(\zeta_n^\pm,\SS_n)$ of $\Cl(n;\CC)$. For $x\in\CC^n$ we write
\begin{gather*} \underline{x} = \zeta_n(x) \end{gather*}
whenever the representation $\zeta_n$ is clear from the context.

The restriction of any irreducible complex representation of the Clifford algebra $\Cl(n;\CC)$ to $\Pin(n)\subseteq\Cl(n)\subseteq\Cl(n;\CC)$ defines an irreducible representation of~$\Pin(n)$. These representations are called \textit{fundamental spin representations} and will also be denoted by $(\zeta_n,\SS_n)$ for $n$ even and $(\zeta_n^\pm,\SS_n)$ for~$n$ odd. Note that for $n$ even $\zeta_n\otimes\det\simeq\zeta_n$ and for $n$ odd $\zeta_n^\pm\otimes\det\simeq\zeta_n^\mp$, where $\det$ is the one-dimensional representation of $\Pin(n)$ given by the determinant character $\det\colon \Pin(n)\to\upO(n)\to\{\pm1\}$.

For $n$ even the restriction of $\zeta_n$ to $\Spin(n)\subseteq\Pin(n)$ decomposes into the direct sum of two irreducible representations of $\Spin(n)$ according to the decomposition
\begin{gather*} \textstyle\bigwedge^\bullet W = \bigwedge^\even W\oplus\bigwedge^\odd W. \end{gather*}
They have highest weights $\big(\frac{1}{2},\ldots,\frac{1}{2},\pm\frac{1}{2}\big)$ in the standard notation. For $n$ odd the restrictions of $\zeta_n^+$ and $\zeta_n^-$ to $\Spin(n)$ define equivalent irreducible representations of $\Spin(n)$ whose highest weight is $\big(\frac{1}{2},\ldots,\frac{1}{2}\big)$.

\subsection{Higher spin representations on monogenic polynomials}\label{app:HigherSpinReps}

Let $i\geq0$. For $n$ even the direct sum of the two irreducible $\Spin(n)$-representations with highest weight $\big(i+\frac{1}{2},\frac{1}{2},\ldots,\frac{1}{2},\pm\frac{1}{2}\big)$ extends uniquely to an irreducible representation of $\Pin(n)$ which we denote by $\zeta_{n,i}$. Note that $\zeta_{n,i}\otimes\det\simeq\zeta_{n,i}$. For $n$ odd the irreducible $\Spin(n)$-representation with highest weight $\big(i+\frac{1}{2},\frac{1}{2},\ldots,\frac{1}{2}\big)$ has two inequivalent extensions to $\Pin(n)$ which we denote by~$\zeta_{n,i}^\pm$.

We realize the representations $\zeta_{n,i}^{(\pm)}$ as monogenic polynomials on $\RR^n$. For each $i\geq0$ the group $\Pin(n)$ acts on the space $\Pol_i\big(\RR^n;\SS_n\big)$ of $\SS_n$-valued homogeneous polynomials on~$\RR^n$ of degree $i$ by
\begin{gather*} (g\cdot\phi)(x) = \zeta_n(g)\phi\big(q(g)^{-1}x\big), \qquad g\in\Pin(n),\qquad x\in\RR^n. \end{gather*}
We consider the Dirac operator $\dirac$ on $\Pol\big(\RR^n;\SS_n\big)=\bigoplus_{i=0}^\infty\Pol_i\big(\RR^n;\SS_n\big)$ given by
\begin{gather*} \dirac\phi(x) = \sum_{k=1}^n\underline{e_k}\,\frac{\partial\phi}{\partial x_k}(x). \end{gather*}
Then the space
\begin{gather*} \calM_i\big(\RR^n;\SS_n\big) = \big\{\phi\in\Pol_i\big(\RR^n;\SS_n\big)\colon \dirac\phi=0\big\} \end{gather*}
of \textit{homogeneous monogenic polynomials of degree $i$} is invariant under the action of $\Pin(n)$ and defines an irreducible representation $\zeta_{n,i}$ of $\Pin(n)$. For $n$ even, $\zeta_{n,i}\otimes\det\simeq\zeta_{n,i}$, the isomorphism being
\begin{gather*} \calM_i\big(\RR^n;\SS_n\big)\to\calM_i\big(\RR^n;\SS_n\big), \qquad \phi\mapsto\gamma\circ\phi. \end{gather*}
For $n$ odd, $\zeta_{n,i}^+=\zeta_{n,i}$ and $\zeta_{n,i}^-=\zeta_{n,i}\otimes\det$ define inequivalent irreducible representations $\big(\zeta_{n,i}^\pm,\calM_i\big(\RR^n;\SS_n\big)\big)$ of $\Pin(n)$.

\subsection{Branching laws}\label{app:BranchingLaws}

We give the explicit branching laws for the restriction of the $\Pin(n+1)$-representations $\zeta_{n+1,i}^{(\pm)}$ to $\Pin(n)$.

\subsubsection{Fundamental spin representations}

We use the explicit realizations of the representations $\zeta_n^{(\pm)}$ given in Section~\ref{app:CliffordModules}. Then $\SS_{n+1}=\SS_n$ for even $n$, and $\SS_{n+1}=\SS_n\oplus(\SS_n\wedge w_{m+1})$ for odd $n=2m+1$. Recall the map $\gamma\colon \SS_n\to\SS_n$ from~\eqref{eq:DefGamma}.

For $n$ even, the restriction of $\big(\zeta_{n+1}^\pm,\SS_{n+1}\big)$ to $\Cl(n;\CC)$ stays irreducible and is isomorphic to $(\zeta_n,\SS_n)$:

\begin{Lemma}\label{lem:FundSpinBranchingEven}Assume $n=2m$ is even.
\begin{enumerate}\itemsep=0pt
\item[$1.$] The explicit isomorphisms $\zeta_n\simeq\zeta_{n+1}^\pm|_{\Cl(n;\CC)}$ are given by
\begin{gather*}
(\zeta_n,\SS_n) \to \big(\zeta_{n+1}^+,\SS_{n+1}\big), \qquad \omega\mapsto\omega,\\
(\zeta_n,\SS_n) \to \big(\zeta_{n+1}^-,\SS_{n+1}), \qquad \omega\mapsto\gamma(\omega).
\end{gather*}
\item[$2.$] $\zeta_{n+1}(e_{n+1})=\sqrt{-1}\gamma$.
\end{enumerate}
\end{Lemma}

\begin{proof}This follows immediately from the definition of $\zeta_n$ and $\zeta_{n+1}^\pm$ and the fact that $\gamma$ intertwines $\zeta_n$ with $\zeta_n\circ\alpha$.
\end{proof}

For $n$ odd, the restriction of $(\zeta_{n+1},\SS_{n+1})$ to $\Cl(n;\CC)$ decomposes into the direct sum of $(\zeta_n^+,\SS_n)$ and $(\zeta_n^-,\SS_n)$:

\begin{Lemma}\label{lem:FundSpinBranchingOdd}Assume $n=2m+1$ is odd.
\begin{enumerate}\itemsep=0pt
\item[$1.$] The explicit embeddings $\zeta_n^\pm\hookrightarrow\zeta_{n+1}|_{\Cl(n;\CC)}$ are given by
\begin{gather*}
(\zeta_n^+,\SS_n)\hookrightarrow(\zeta_{n+1},\SS_{n+1}), \qquad \omega\mapsto\omega-\sqrt{-1}\omega\wedge w_{m+1},\\
(\zeta_n^-,\SS_n)\hookrightarrow(\zeta_{n+1},\SS_{n+1}), \qquad \omega\mapsto\gamma(\omega)+\sqrt{-1}\gamma(\omega)\wedge w_{m+1}.
\end{gather*}
In particular, the image of the embedding $\zeta_n^\pm\hookrightarrow\zeta_{n+1}$ is equal to
\begin{gather*} \SS_{n+1}^\pm = \big\{\omega\mp\sqrt{-1}\omega\wedge w_{m+1}\colon \omega\in\SS_n\big\}. \end{gather*}
\item[$2.$] $\zeta_{n+1}(e_{n+1})|_{\SS_{n+1}^\pm}=\mp\gamma$. In particular, the map $\frac{1}{2}(\id\pm\zeta_{n+1}(e_{n+1})\circ\gamma)$ is the canonical projection $\SS_{n+1}\to\SS_{n+1}^\pm$.
\end{enumerate}
\end{Lemma}

\begin{proof}It is easy to see that for $\omega\in\SS_n\subseteq\SS_{n+1}$ we have
\begin{alignat*}{3}
&\zeta_{n+1}(e_n)\omega= \gamma(\omega)\wedge w_{m+1}, \qquad & & \zeta_{n+1}(e_n)(\omega\wedge w_{m+1})= -\gamma(\omega),&\\
&\zeta_{n+1}(e_{n+1})\omega= -\sqrt{-1}\gamma(\omega)\wedge w_{m+1}, \qquad && \zeta_{n+1}(e_{n+1})(\omega\wedge w_{m+1})= -\sqrt{-1}\gamma(\omega),&
\end{alignat*}
then the claims follow.
\end{proof}

For the representations $\zeta_{n+1}^{(\pm)}$ of $\Pin(n+1)$ this implies
\begin{alignat*}{3}
&\zeta_{n+1}^\pm|_{\Pin(n)}\simeq \zeta_n \qquad && \mbox{($n$ even)},&\\
&\zeta_{n+1}|_{\Pin(n)}\simeq \zeta_n^+\oplus\zeta_n^- \qquad && \mbox{($n$ odd)}.&
\end{alignat*}

\subsubsection{Higher spin representations}

Using classical branching laws for the pair $(\so(n+1),\so(n))$ of Lie algebras, it is easy to see that the branching laws for the higher spin representations are
\begin{alignat*}{3}
&\zeta_{n+1,i}^\pm|_{\Pin(n)}\simeq \bigoplus_{j=0}^i\zeta_{n,j} \qquad && \mbox{($n$ even)},& \\
&\zeta_{n+1,i}|_{\Pin(n)}\simeq \bigoplus_{j=0}^i\big(\zeta_{n,j}^+\oplus\zeta_{n,j}^-\big) \qquad && \mbox{($n$ odd)}.&
\end{alignat*}
To make this branching explicit in the realizations on monogenic polynomials, we use the classical Gegenbauer polynomials $C_n^\lambda(z)$ given by (see, e.g., \cite[Chapter~10.9, equation~(18)]{EMOT81})
\begin{gather*} C_n^\lambda(z) = \sum_{m=0}^{\lfloor n/2\rfloor}\frac{(-1)^m(\lambda)_{n-m}}{m!(n-2m)!}(2z)^{n-2m}. \end{gather*}
The polynomial $u=C_n^\lambda$ satisfies the differential equation (see \cite[Chapter 10.9, equation~(14)]{EMOT81})
\begin{gather*} (1-z^2)u'' - (2\lambda+1)zu' + n(n+2\lambda)u = 0. \end{gather*}

\begin{Lemma}\label{lem:BranchingMonogenics}Let $(\zeta_{n+1},\SS_{n+1})$ be a fundamental spin representation of $\Pin(n+1)$ and assume $(\zeta_n,\SS_n)$ occurs in the restriction of $\zeta_{n+1}$ to $\Pin(n)$. If we identify $\SS_n$ with a subspace of $\SS_{n+1}$, then for every $0\leq j\leq i$ the map
\begin{align*}
I_{j\to i}\colon \ & \calM_j\big(\RR^n;\SS_n\big)\to \calM_i\big(\RR^{n+1},\SS_{n+1}\big),\\
& I_{j\to i}\phi(x',x_{n+1}) = (n+i+j-1)|x|^{i-j}\phi(x')C_{i-j}^{\frac{n-1}{2}+j}\big(\tfrac{x_{n+1}}{|x|}\big)\\
&\hphantom{I_{j\to i}\phi(x',x_{n+1}) =}{} +(n+2j-1)|x|^{i-j-1}\underline{x'e_{n+1}}\phi(x')C_{i-j-1}^{\frac{n+1}{2}+j}\big(\tfrac{x_{n+1}}{|x|}\big)
\end{align*}
is $\Pin(n)$-intertwining, where $x=(x',x_{n+1})\in\RR^{n+1}$.
\end{Lemma}

\begin{proof}We first show the intertwining property. Let $g\in\Pin(n)$, then $(q(g)x)'=q(g)x'$ and $(q(g)x)_{n+1}=x_{n+1}$. Further,
\begin{align*}
\zeta_{n+1}(g)\circ\underline{x'e_{n+1}} &= \zeta_{n+1}(g)\circ\zeta_{n+1}(x')\circ\zeta_{n+1}(e_{n+1})\\
&= \zeta_{n+1}\big(\alpha(g)x'g^{-1}\big)\circ\zeta_{n+1}\big(\alpha(g)e_{n+1}g^{-1}\big)\circ\zeta_{n+1}(g)\\
&= \zeta_{n+1}(q(g)x')\circ\zeta_{n+1}(q(g)e_{n+1})\circ\zeta_{n+1}(g)\\
&= \underline{(q(g)x')e_{n+1}}\circ\zeta_{n+1}(g).
\end{align*}
Then the intertwining property follows. It remains to show that $I_{j\to i}\phi$ is monogenic, and for this we abbreviate
\begin{gather*} p=(n+i+j-1)C_{i-j}^{\frac{n-1}{2}+j} \qquad \mbox{and} \qquad q=(n+2j-1)C_{i-j-1}^{\frac{n+1}{2}+j}. \end{gather*}
Then
\begin{gather*} I_{j\to i}\phi(x',x_{n+1}) = |x|^{i-j}\phi(x')p\big(\tfrac{x_{n+1}}{|x|}\big)+|x|^{i-j-1}\underline{x'e_{n+1}}\phi(x')q\big(\tfrac{x_{n+1}}{|x|}\big). \end{gather*}
Applying the Dirac operator to $I_{j\to i}\phi$ yields, after a short computation:
\begin{gather*}
 \dirac(I_{j\to i}\phi)(x',x_{n+1}) = \sum_{k=1}^n\underline{e_k}\,\frac{\partial(I_{j\to i}\phi)}{\partial x_i}(x',x_{n+1})+\underline{e_{n+1}\,}\frac{\partial(I_{j\to i}\phi)}{\partial x_{n+1}}(x',x_{n+1})\\
\qquad{} = |x|^{i-j-2}\underline{x'}\phi(x')\big((i-j)p(z)-zp'(z)+(i-j-1)zq(z)+\big(1-z^2\big)q'(z)\big)\\
\qquad\quad{} +|x|^{i-j-1}\underline{e_{n+1}}\phi(x')\big((i-j)zp(z)+\big(1-z^2\big)p'(z)-(n+i+j-1)q(z)\\
\qquad\quad{}+(i-j-1)z^2q(z)+\big(1-z^2\big)zq'(z)\big),
\end{gather*}
where $z=\frac{x_{n+1}}{|x|}$. This vanishes if and only if
\begin{gather*}
(i-j)p-zp' = -(i-j-1)zq-(1-z^2)q',\\
(i-j)zp+\big(1-z^2\big)p' = (n+i+j-1)q-(k-j-1)z^2q-\big(1-z^2\big)zq'.
\end{gather*}
Multiplying the first equation with $z$ and subtracting it from the second one yields
\begin{gather*} p' = (n+i+j-1)q. \end{gather*}
If we now insert this into the first equation we obtain
\begin{gather*} \big(1-z^2\big)p''-(n+2j)zp'+(i-j)(n+i+j-1)p = 0, \end{gather*}
which is the Gegenbauer differential equation with solution $p(z)=C_{i-j}^{\frac{n-1}{2}+j}(z)$, the classical Gegenbauer polynomial. Finally, using (see \cite[Chapter~10.9, equation~(23)]{EMOT81})
\begin{gather*} \frac{{\rm d}}{{\rm d}z}C_n^\lambda(z) = 2\lambda C_{n-1}^{\lambda+1}(z) \end{gather*}
and renormalizing $p$ and $q$ shows that indeed $\dirac(I_{j\to i}\phi)=0$ and the proof is complete.
\end{proof}

\subsection{Multiplication with coordinates}

By the Fischer decomposition we have
\begin{gather}
\Pol\big(\RR^n;\SS_n\big) = \bigoplus_{i,j=0}^\infty\underline{x}^j\calM_i\big(\RR^n;\SS_n\big).\label{eq:FischerDecomposition}
\end{gather}
Note that $\underline{x}^2=-|x|^2$. We now study how the product of a monogenic polynomial with a coordinate function $x_k$ decomposes according to this decomposition.

\begin{Lemma}\label{lem:MultWithCocycle}Let $\phi\in\calM_i\big(\RR^n;\SS_n\big)$ and $1\leq k\leq n$. Then
\begin{gather*} x_k\phi\in\calM_{i+1}\big(\RR^n;\SS_n\big)\oplus\underline{x}\calM_i\big(\RR^n;\SS_n\big)\oplus|x|^2\calM_{i-1}\big(\RR^n;\SS_n\big). \end{gather*}
More precisely, $x_k\phi=\phi_k^+-\underline{x}\phi_k^0+|x|^2\phi_k^-$ with
\begin{gather*}
\phi_k^+ = x_k\phi+\frac{1}{n+2i}\left(\underline{xe_k}\phi-|x|^2\frac{\partial\phi}{\partial x_k}\right)\in\calM_{i+1}\big(\RR^n;\SS_n\big),\\
\phi_k^0 = \frac{1}{n+2i}\left(\underline{e_k}\phi-\frac{2}{n+2i-2}\underline{x}\frac{\partial\phi}{\partial x_k}\right)\in\calM_i\big(\RR^n;\SS_n\big),\\
\phi_k^- = \frac{1}{n+2i-2}\frac{\partial\phi}{\partial x_k}\in\calM_{i-1}\big(\RR^n;\SS_n\big).
\end{gather*}
\end{Lemma}

\begin{proof}This follows from the following identities which are easily verified:
\begin{alignat*}{3}
&\dirac(x_k\phi)= \underline{e_k}\phi, \qquad & & \dirac(\underline{xe_k}\phi)= -(n+2i)\underline{e_k}\phi+2\underline{x}\frac{\partial\phi}{\partial x_k},&\\
& \dirac(\underline{e_k}\phi)= -2\frac{\partial\phi}{\partial x_k}, \qquad && \dirac\left(\underline{x}\frac{\partial\phi}{\partial x_k}\right)= -(n+2i-2)\frac{\partial\phi}{\partial x_k},&\\
&\dirac\frac{\partial\phi}{\partial x_k}= 0, \qquad && \dirac\left(|x|^2\frac{\partial\phi}{\partial x_k}\right)= 2\underline{x}\frac{\partial\phi}{\partial x_k}.&\tag*{\qed}
\end{alignat*}\renewcommand{\qed}{}
\end{proof}

Now, recall the intertwining map $I_{j\to i}\colon \calM_j\big(\RR^n;\SS_n\big)\to\calM_i\big(\RR^{n+1},\SS_{n+1}\big)$ from Lemma~\ref{lem:BranchingMonogenics}. The next result relates the decompositions of the polynomials $x_k(I_{j\to i}\phi)(x',x_{n+1})$ and $x_k\phi(x')$ for $1\leq k\leq n$.

\begin{Lemma}\label{lem:CocycleVsEmbedding}For $0\leq j\leq i$ and $1\leq k\leq n$ we have
\begin{gather*}
(I_{j\to i}\phi)_k^+ = \frac{n+2j-1}{n+2i+1}I_{j+1\to i+1}\big(\phi_k^+\big) + \frac{i-j+1}{n+2i+1}I_{j\to i+1}\big(\underline{e_{n+1}}\phi_k^0\big)\\
\hphantom{(I_{j\to i}\phi)_k^+ =}{} -\frac{(i-j+1)(i-j+2)}{(n+2i+1)(n+2j-3)}I_{j-1\to i+1}\big(\phi_k^-\big),\\
(I_{j\to i}\phi)_k^0 = -\frac{2(n+2j-1)}{(n+2i-1)(n+2i+1)}I_{j+1\to i}\big(\underline{e_{n+1}}\phi_k^+\big)\\
\hphantom{(I_{j\to i}\phi)_k^0 =}{} +\frac{(n+2i)(n+2j-1)}{(n+2i-1)(n+2i+1)}I_{j\to i}\big(\phi_k^0\big)\\
\hphantom{(I_{j\to i}\phi)_k^0 =}{} -\frac{2(i-j+1)(n+i+j-1)}{(n+2i-1)(n+2i+1)(n+2j-3)}I_{j-1\to i}\big(\underline{e_{n+1}}\phi_k^-\big),\\
(I_{j\to i}\phi)_k^- = -\frac{n+2j-1}{n+2i-1}I_{j+1\to i-1}\big(\phi_k^+\big) - \frac{n+i+j-1}{n+2i-1}I_{j\to i-1}\big(\underline{e_{n+1}}\phi_k^0\big)\\
\hphantom{(I_{j\to i}\phi)_k^- =}{}+\frac{(n+i+j-2)(n+i+j-1)}{(n+2i-1)(n+2j-3)}I_{j-1\to i-1}\big(\phi_k^-\big).
\end{gather*}
\end{Lemma}

Note that we identify $\SS_n$ with a subspace of $\SS_{n+1}$, so that $\underline{e_{n+1}}\SS_n$ is also a subspace of $\SS_{n+1}$ which is invariant under $\Pin(n)$. More precisely, if~$\Pin(n)$ acts on~$\SS_n$ by~$\zeta_n$, then it acts on~$\underline{e_{n+1}}\SS_n$ by~$\zeta_n\otimes\det$.

\begin{proof}This a lengthy but elementary computation involving several identities for Gegenbauer polynomials. We provide these identities and leave the computation to the reader. First, we have
\begin{gather}
\frac{{\rm d}}{{\rm d}z}C_n^\lambda = 2\lambda C_{n-1}^{\lambda+1},\label{eq:G1}\\
2\lambda zC_n^{\lambda+1} - (n+1)C_{n+1}^\lambda - 2\lambda C_{n-1}^{\lambda+1} = 0,\label{eq:G3}\\
2\lambda C_n^{\lambda+1} - (n+2\lambda)C_n^\lambda = 2\lambda zC_{n-1}^{\lambda+1},\label{eq:G6}\\
2\lambda\big(1-z^2\big)C_{n-1}^{\lambda+1} = (n+2\lambda)zC_n^\lambda - (n+1)C_{n+1}^\lambda,\label{eq:G7}\\
(\lambda-1)C_{n+1}^\lambda - (n+\lambda)C_{n+1}^{\lambda-1} = (\lambda-1)C_{n-1}^\lambda.\label{eq:G8}
\end{gather}
\eqref{eq:G1} is \cite[Chapter 10.9, equation~(23)]{EMOT81}, \eqref{eq:G3} follows from \cite[Chapter~10.9, equation~(24)]{EMOT81} and \eqref{eq:G1}, \eqref{eq:G6} is a consequence of \cite[Chapter~10.9, equation~(25)]{EMOT81} and \eqref{eq:G1}, \eqref{eq:G7} is \cite[Chapter~10.9, equation~(35)]{EMOT81} and \eqref{eq:G8} is \cite[Chapter~10.9, equation~(36)]{EMOT81}. Moreover, \eqref{eq:G6} combined with \eqref{eq:G7} implies
\begin{gather}
4\lambda(\lambda+1)\big(1-z^2\big)C_{n-2}^{\lambda+2} - 2\lambda(2\lambda+1)zC_{n-1}^{\lambda+1} + n(2\lambda+n)C_n^\lambda = 0,\label{eq:G5}
\end{gather}
\eqref{eq:G5} and \eqref{eq:G6} give
\begin{gather*}
4\lambda(\lambda+1)\big(1-z^2\big)C_{n-2}^{\lambda+2} + (2\lambda+n)(2\lambda+n+1)C_n^\lambda -2\lambda(2\lambda+1)C_n^{\lambda+1} = 0,%\label{eq:G4}
\end{gather*}
and \eqref{eq:G3} together with \eqref{eq:G5} implies
\begin{gather*}
4\lambda(\lambda+1)\big(1-z^2\big)C_{n-2}^{\lambda+2} - 2\lambda(2\lambda+1)C_{n-2}^{\lambda+1} + n(n-1)C_n^\lambda = 0.%\label{eq:G9}
\tag*{\qed}
\end{gather*}
\renewcommand{\qed}{}
\end{proof}

\pdfbookmark[1]{References}{ref}
\LastPageEnding

\end{document}